\documentclass[11pt]{amsart}

\usepackage[text={475pt,660pt},centering]{geometry}

\usepackage{color}
\usepackage{esint,amssymb}
\usepackage{graphicx}
\usepackage{MnSymbol}
\usepackage{mathtools}
\usepackage[colorlinks=true, pdfstartview=FitV, linkcolor=blue, citecolor=blue, urlcolor=blue,pagebackref=false]{hyperref}
\usepackage{microtype}

\usepackage{bm}
\usepackage{scalerel} 
\usepackage{dsfont}
\usepackage{mathrsfs}
\usepackage{eufrak}
\usepackage[font={footnotesize}]{caption}

\definecolor{darkgreen}{rgb}{0,0.5,0}
\definecolor{darkblue}{rgb}{0,0,0.7}
\definecolor{darkred}{rgb}{0.9,0.1,0.1}

\makeatletter
\newtheorem*{rep@theorem}{\rep@title}
\newcommand{\newreptheorem}[2]{%
\newenvironment{rep#1}[1]{%
 \def\rep@title{#2 \ref{##1}}%
 \begin{rep@theorem}}%
 {\end{rep@theorem}}}
\makeatother

\newtheorem{theorem}{Theorem}
\newreptheorem{theorem}{Theorem}
\newtheorem{proposition}{Proposition}
\newreptheorem{proposition}{Proposition}
\newtheorem{lemma}[proposition]{Lemma}
\newreptheorem{lemma}{Lemma}

\theoremstyle{remark}

\theoremstyle{definition}
\newtheorem{definition}[proposition]{Definition}
\newtheorem{remark}[proposition]{Remark}
\newtheorem{conjecture}[proposition]{Conjecture}

\numberwithin{equation}{section}
\numberwithin{proposition}{section}

\newcommand{\Z}{\mathbb{Z}}
\newcommand{\N}{\mathbb{N}}
\newcommand{\R}{\mathbb{R}}

\newcommand{\E}{\mathbb{E}}
\renewcommand{\P}{\mathbb{P}}
\newcommand{\F}{\mathcal{F}}
\newcommand{\Zd}{\mathbb{Z}^d}
\newcommand{\Rd}{{\mathbb{R}^d}}

\newcommand{\Ed}{E_d}

\newcommand{\ep}{\varepsilon}

\renewcommand{\a}{\mathbf{a}}

\renewcommand{\subset}{\subseteq}
\newcommand{\fact}[1]{#1\mathpunct{}!}

 \newcommand{\cu}{\square}

\renewcommand{\fint}{\strokedint}

\DeclareMathOperator{\dist}{dist}

\DeclareMathOperator*{\osc}{osc}

\DeclareMathOperator{\diam}{diam}

\DeclareMathOperator{\size}{size}

\DeclareMathOperator{\argmin}{argmin}

\renewcommand{\bar}{\overline}
\renewcommand{\tilde}{\widetilde}

\newcommand{\indc}{\mathds{1}}

\newcommand{\C}{\mathscr{C}}
\newcommand{\Pa}{\mathcal{P}}
\newcommand{\Pas}{\mathcal{P}_*}

\newcommand{\G}{\mathcal{G}}
\renewcommand{\S}{\mathcal{S}}
\newcommand{\M}{\mathcal{M}}
\newcommand{\X}{\mathcal{X}}

\renewcommand{\O}{\mathcal{O}}
\newcommand{\A}{\mathcal{A}}
\newcommand{\zbar}{\overline{z}}
\newcommand{\pc}{\mathfrak{p}_{\mathfrak{c}}}
\newcommand{\p}{\mathfrak{p}}
\newcommand{\e}{\mathfrak{e}}
\newcommand{\aconn}{\leftrightarrow_\a}
\newcommand{\T}{\mathcal{T}}

\begin{document}

\title[Optimal corrector estimates on percolation clusters]{Optimal corrector estimates on percolation clusters}

\begin{abstract}
We prove optimal quantitative estimates on the first-order correctors on supercritical percolation clusters: we show that they are bounded in dimension larger than $3$ and have logarithmic growth in dimension $2$ in the sense of stretched exponential moments. The main ingredients are a renormalization scheme of the supercritical percolation cluster, following the works of Pisztora~\cite{P96}; large-scale regularity estimates developed by Armstrong and the author in~\cite{AD2}; and a nonlinear concentration inequality of Efron-Stein type which is used to transfer quantitative information from the environment to the correctors.
\end{abstract}

\author[P. Dario]{Paul Dario}
\address[P. Dario]{Tel Aviv University, School of Mathematical Sciences, Tel Aviv, 69978, Israel}
\email{pauldario@mail.tau.ac.il}

\keywords{}
\subjclass[2010]{}
\date{\today}

\maketitle

\setcounter{tocdepth}{1}
\tableofcontents

\section{Introduction}
\label{s.introduction}

\subsection{Motivation and informal summary of results}
\label{ss.summary}

We consider the random conductance model on the supercritical percolation cluster defined as follows. We let $\Zd$ be the standard hypercubic lattice and $\mathcal{B}_d$ be the set of \textit{bonds} of $\Zd$. We fix a parameter $\lambda \in (0,1)$ and we are given a function
\begin{equation} \label{assenv}
\a : \mathcal{B}_d \rightarrow \{ 0 \} \cup [\lambda , 1],
\end{equation}
the value $\a(e)$ is called the conductance of the bond $e$ and we assume that the collection $\left( \a(e) \right)_{e \in \mathcal{B}_d}$ is an i.i.d family of random variables. We assume that the probability $\p := \P \left( \a(e) \neq 0 \right)> \pc(d)$, where $\pc(d)$ is the bond percolation threshold for the lattice $\Zd$. It follows that, almost surely, there exists a unique maximal connected component of bonds with nonzero conductance which we denote by $\C_\infty = \C_\infty(\a)$. One then wishes to study the continuous time \textit{random walk} $X_t$ in the random environment $\a$ defined as follows. We select an environment $\a$ such that $0$ belongs to the infinite cluster $\C_\infty$ and start a random walker  at the origin, $X(0) = 0$. Each edge $e$ is equipped with a random clock and rings after exponential waiting time with expectation $\a(e)^{-1}$. When $X(t) = x$, the random walker waits until a clock of an edge adjacent to $x$ rings and then moves instantly across that edge. Note that the random walker is confined to the infinite cluster $\C_\infty$. This random walk is a Markov process and a common strategy to study it is to consider its generator, which is given by the random discrete elliptic PDE
\begin{equation*}
- \nabla \cdot \a \nabla u,
\end{equation*}
where the operator $- \nabla \cdot \a \nabla u $ is defined on functions $u : \C_\infty \rightarrow \R$ by, for each point $x \in \C_\infty$,
\begin{equation*}
 \nabla \cdot \a \nabla u (x) = \sum_{y \sim x} \a(\{ x, y \}) (u(y) - u(x)).
\end{equation*}
In this article, we wish to study this random elliptic PDE by studying the (random) set of \textit{harmonic functions} for this operator. In~\cite{BDKY}, it was proved, in the case when $\a$ takes only the two values $0$ and $1$, that every harmonic function $h$ with prescribed linear growth is close to a linear function: the random vector space of harmonic functions with growth at most linear is finite dimensional; its dimension is equal to $(d+1)$ almost surely. Moreover, for each harmonic function in this space, there exists a unique vector $p \in \Rd$ such that the difference $\chi_p(x) := h(x) - p \cdot x$ grows sublinearly as $x$ tends to infinity. This result was quantified and extended to the generality presented in this introduction by Armstrong and the author in~\cite{AD2}, where it is shown that the corrector is $o\left(|x|^{1-\delta}\right)$ for some small but strictly positive exponent $\delta$.

The map $\chi_p$ is called the corrector and is the central object of this article: our goal is to prove optimal bounds in terms of spatial scaling (and suboptimal with respect to stochastic integrability) on the first-order correctors. We show, in the sense of stretched exponential moments, that the correctors are bounded in dimensions $d \geq 3$, and have increments which grow like the square root of the logarithm of the distance in dimension $2$. This result is summarized in the following theorem.

\begin{theorem}[Optimal $L^\infty$ estimates for first-order correctors]  \label{realmainthm}
\begin{sloppypar}There exist an exponent $s := s(d ,\p, \lambda ) > 0$ and a constant $C:= C(d , \p , \lambda) < \infty$ such that for each $x , y \in \Zd$ and each $p \in \Rd$, \end{sloppypar}
\begin{equation} \label{main.est.perco22}
\left| \chi_p(x) - \chi_p(y)\right| \indc_{\{x,y \in \C_\infty\}} \leq \left\{  \begin{array}{lcl}
    \O_s \left( C|p| \log^\frac12 |x-y| \right) & \mbox{if } d=2, \\
     \O_s \left( C|p| \right) & \mbox{if } d\geq 3, \\
  \end{array}
\right.
\end{equation}
where, for a random variable $X$, we write $X \leq \O_s(K)$, to mean
\begin{equation*}
\E \left[ \exp \left( \left( \frac{X}{K} \right)^s \right) \right] \leq 2.
\end{equation*}
\end{theorem} 
 
Obtaining information on the corrector is important and has proved to be useful. For instance,  qualitative sublinarity of the corrector can be used to prove invariance principles for the random walker $X_t$ following the general principle described below:  if one denotes by $\chi := (\chi_1 , \ldots, \chi_d)$ the vector-valued corrector, where $\chi_i$ is the corrector such that $e_i \cdot x + \chi_i(x)$ is harmonic, then the process
\begin{equation*}
X_t + \chi(X_t) ~~\mbox{is a martingale, almost surely with respect to the environment.}
\end{equation*}
The strategy is to apply a standard martingale convergence theorem and then to derive a quenched invariance principle for the rescaled process $\ep X_{t/ \ep^2} + \ep \chi(X_{t/ \ep^2})$. Using the sublinearity of the corrector allows to prove an invariance principle for the diffusion process $X$ itself. This approach was carried out on the infinite supercritical cluster (in the case when $\a$ takes only the values $0$ and $1$) first by Sidoravicius and Sznitman in~\cite{SS} in dimension larger than $4$, and a few years later by Mathieu, Piatnitski~\cite{MP} and Berger, Biskup~\cite{BB} in all dimensions $d \geq 2$. Prior to these results, the generator of the random walk was studied by Barlow in~\cite{Ba} and by Mathieu, Remy in~\cite{MR}, who proved heat-kernel bounds for the transition probability. 

In the more general setting of i.i.d random conductances, when $\a$ can a priori take values in $[0, \infty)$, a quenched functional central limit theorem was established by Andres, Barlow, Deuschel and Hambly in~\cite{ABDH}, provided that there exists an infinite cluster of nonzero conductances, based on the previous works of Mathieu~\cite{M}, Biskup and Prescott~\cite{BP}, Barlow and Deuschel~\cite{BD}. More general models of random walks on percolation clusters with long range correlation, including random interlacements and level sets of the Gaussian free field, are studied by Procaccia, Rosenthal and Sapozhnikov in~\cite{procaccia2016}.

\smallskip

Tight bounds on the corrector are useful to derive invariance principles but they are also the crucial ingredient for the derivation of optimal error and two-scale expansion estimates for the homogenization of general boundary value problems. They can be used to obtain \textit{a Berry-Essen theorem}, in the spirit of Mourrat~\cite{mourrat2012quantitative} in the uniformly elliptic setting (see also Andres and Neukamm~\cite{AN17} for an extension of these results to degenerate and correlated environments) and are also important to obtain precise information on the Green's function for the Laplacian on the infinite cluster as well as on the transition probability for the random walk, as is explained in~\cite[Chapters 8 and 9]{armstrong2017quantitative}. They can also inform the performance of numerical algorithms for the computation of the homogenized diffusivity~\cite{mourrat2016efficient} and of solutions to the heterogeneous equation~\cite{armstrong2018iterative}.

\smallskip

The tools developed in this article come from the theory of stochastic homogenization which studies the solutions of the elliptic equation
\begin{equation*}
- \nabla \cdot \a \nabla u = 0 \mbox{ in } \Rd,
\end{equation*}
where the environment $\a$ is a random map from $\Rd$ to the set of symmetric matrices, satisfying some assumptions of ellipticity, stationarity and ergodicity. There have been recent developments in the quantitative homogenization of uniformly elliptic divergence-form equations, which started with the work of Gloria and Otto~\cite{GO1}. In this article, they were able to obtain moments bounds on the corrector with an optimal spatial scaling, by using a spectral gap inequality, which was first introduced into stochastic homogenization by Naddaf and Spencer in~\cite{NS}, to quantify the ergodicity of the coefficient field. This program was then continued by Gloria and Otto in~\cite{GO2,  GO5, GO3} and by Neukamm Gloria and Otto in~\cite{GNO2, GNO, GNO3} and has implications to random walks as explained in~\cite{EGMN}.

Another approach was later initiated by Armstrong and Smart in~\cite{AS}, who extended the techniques of Avellaneda and Lin~\cite{AL1, AL2} and the ones of Dal Maso and Modica~\cite{DM1,DM2}, and were able to obtain a large scale $C^{0,1}$-regularity theory under an assumption of finite range dependence on the environment. These results were then generalized by Armstrong, Kuusi and Mourrat to general mixing conditions and to other types of equations~\cite{AM} and improved to obtain optimal rates of convergence~\cite{AKM2, armstrong2017quantitative}. 

The theory is now well-understood in the uniformly elliptic setting. Going beyond this setting has been the subject of much research recently in different directions. In~\cite{LNO}, Lamacz, Neukamm and Otto were able to extend these results to a model of Bernoulli bond percolation, where the standard model is modified such that all the bonds in a fixed unit direction are always open. Another way of removing the ellipticity assumption can be the following: we define some (scalar) random variables $0 < \lambda \leq \mu < \infty$ according to the formulas
\begin{equation*}
\lambda := \inf_{\xi \in \Rd\setminus\{0\}} \frac{\xi \cdot \a \xi}{|\xi|^2} ~\mbox{and}~  \mu := \sup_{\xi \in \Rd\setminus\{0\}} \frac{\xi \cdot \a \xi}{|\xi|^2},
\end{equation*}
and add an assumption on the integrability of $\lambda$ and $\mu$: there exist $p , q \in [1 , \infty ]$ such that
\begin{equation} \label{cond2}
\E \left[ \lambda^{-p} \right] + \E \left[ \mu^{q} \right] < \infty.
\end{equation}
This setting was first considered by Andres, Deuschel, Slowik in~\cite{andres2015invariance} (see also~\cite{andres2016heat}), and then by Chiarini and Deuschel in~\cite{chiarini2016invariance}. They are able to obtain a quenched invariance principle for the diffusion process under the assumption $1/p + 1/q < 2/d$, which allowed them to perform a Moser iteration. In~\cite{bella2018liouville}, Bella, Fehrman and Otto, still working under the assumption $1/p + 1/q < 2/d$, were able to obtain a first-order Liouville theorem and a large scale $C^{1, \alpha}$-estimate for $\a$-harmonic functions. An extension of these results to the case of time-dependent coefficients has been carried out by~\cite{andres2018quenched}. The condition~\eqref{cond2} requires the value of the conductances to be non-zero almost surely, an extension of this model in a case when the conductance is allowed to be zero and to be small (under some moments condition) was investigated by Deuschel, Nguyen and Slowik in~\cite{deuschel2018quenched}.

\smallskip

The setting considered in this article is different from the models satisfying condition~\eqref{cond2}: we are working with the i.i.d. random conductance model, and we assume the value of the conductances to be either $0$ or larger than some deterministic constant $\lambda >0$ (see~\eqref{assenv}), with the property that $\P\left( \a(e) \neq 0 \right) > \p_c(d)$. Despite this difference, the main challenge is essentially the same: adapting the various tools and proofs, available in the uniformly elliptic setting, to the degenerate elliptic environment. To this end, we follow the strategy initiated in the previous paper~\cite{AD2} and appeal to a renormalization structure for the supercritical percolation cluster. The construction is recalled in Section~\ref{section2}, where $\Zd$ is partitioned into triadic cubes of different random sizes, well-connected in the sense of Penrose an Pisztora~\cite{PP96}. This partition allows to distinguish regions of $\Zd$ where the infinite cluster is well-behaved, its geometry is similar to the one of the lattice $\Zd$, from regions where the infinite cluster is badly-behaved. In the first case, it is rather straightforward to adapt the theory developed in the uniformly elliptic setting; problems arise where the infinite cluster is badly-behaved. In this situation the theory cannot be adapted. Fortunately, there are few regions were the cluster is badly-behaved, and the theory of stochastic homogenization in the uniformly elliptic setting is robust enough to be adapted to the supercritical cluster.

Our strategy to prove the optimal scaling estimates for the corrector relies on a concentration inequality (cf. Proposition~\ref{resampling}), which gives a convenient way to transfer quantitative information from the coefficient field to the correctors. This idea originates in an unpublished paper from Naddaf and Spencer~\cite{NS}, and was then developed by Gloria and Otto~\cite{GO1, GO2} and Gloria, Neukamm and Otto~\cite{GNO} (see also Mourrat~\cite{mourrat2011variance}) to study stochastic homogenization. More precisely, thanks to this inequality we are able to obtain quantitative estimates on the spatial average of the gradient of the corrector.

We then use the the multiscale Poincar\'e inequality stated in Proposition~\ref{multiscalepoincare} to deduce the estimates on the oscillation of the correctors stated in Theorem~\ref{realmainthm} from the bounds on the spatial average of its gradient.

We conclude this introduction by noting that in Theorem~\ref{realmainthm}, the spatial scaling is optimal while the stochastic integrability is suboptimal: we only obtain a small exponent $s > 0$ of stochastic integrability. This is due to the degenerate structure of the percolation problem and while our method can provide an explicit value for the exponent $s$, we do not expect it can be used to derive the optimal exponent. We nevertheless provide a conjecture.
\begin{conjecture}
In dimension $d=2$, fix $s < 2$, then there exists a constant $C:= C(s, \p , \lambda) < \infty$ such that for each $x , y \in \Zd$ and each $p \in \Rd$,
\begin{equation} \label{10.26conjecture}
\left| \chi_p(x) - \chi_p(y)\right| \indc_{\{x,y \in \C_\infty\}} \leq
    \O_{\frac{s(d-1)}{d}} \left( C|p| \log^\frac12 |x-y| \right).
\end{equation}
In dimension $3$, there exists a constant $C:= C(d ,s, \p , \lambda) < \infty$ such that for each $x , y \in \Zd$ and each $p \in \Rd$,
\begin{equation*}
\left| \chi_p(x) - \chi_p(y)\right| \indc_{\{x,y \in \C_\infty\}} \leq
    \O_{\frac{s(d-1)}{d}} \left( C|p| \right).
\end{equation*}
\end{conjecture}
The reason behind this conjecture is the following: in the uniformly elliptic setting, it is known that the optimal stochastic integrability is the one provided in the statement of the conjecture without the term $(d-1)/d$, see~\cite[Theorem 4.1]{armstrong2017quantitative}. The additional term $(d-1)/d$ is a surface order large deviation effect which can be heuristically explained by the following argument: in the uniformly elliptic setting and in a given ball, to design a bad environment for which one does not have a good control on the growth of the corrector, it is necessary to have a number of ill-behaved edges of order of the volume of the ball. In the percolation setting, the situation is different and one only needs a number of ill-behaved edges of the order of the surface of the ball to design a bad environment: to illustrate this fact, one can note that, if we let $R$ be the radius of the ball, then it is possible to disconnect the ball into two half-balls with only $c R^{d-1}$ closed edges. This phenomenon should result in a deterioration of the stochastic integrability by a factor $(d-1)/d$. 

\medskip

\subsection{Notation and assumptions}

\subsubsection{General notation for the probabilistic model}

We denote by~$\Zd$ the standard~$d$-dimensional hypercubic lattice. The set of bonds of $\Zd$, that is the set of unoriented pairs of nearest neighbors, is denoted by~$\mathcal{B}_d:=\left\{ \{ x,y\}\,:\, x,y\in\Zd, |x-y|_1=1 \right\}$. More specifically, given a subset $U \subseteq \Zd$, we denote by $\mathcal{B}_d(U)$ the set of the bonds of $U$, i.e., $\mathcal{B}_d(U):=\left\{ \{ x,y\}\,:\, x,y\in U, |x-y|_1 =1 \right\}$. The canonical basis of $\Rd$ is denoted by $\e_1,\ldots,\e_d$. For~$x,y\in\Zd$, we write~$x\sim y$ if~$\{ x,y\} \in \mathcal{B}_d$. For some fixed ellipticity parameter $\lambda \in (0,1]$, we define the probability space~$\Omega:= \left( \{ 0\} \cup [\lambda,1] \right)^{\mathcal{B}_d}$ and we equip this probability space with the Borel~$\sigma$-algebra $\F := \mathcal{B}\left(  \{ 0\} \cup [\lambda,1]  \right)^{\otimes \mathcal{B}_d}$. Given an edge $e \in \mathcal{B}_d$, we denote by $\a(e)$ the projection
\begin{equation*} 
  \a(e) \, \colon \left\{
  \begin{array}{lcl}
    \Omega  & \to & \{0\}\cup[\lambda,1], \\
     \left( \omega_{e'} \right)_{e' \in \mathcal{B}_d} & \mapsto& \omega_e. \\
  \end{array} \right.
\end{equation*}
We denote by $\a$ the collection $ \left( \a(e) \right)_{e \in \mathcal{B}_d}$ and we refer to this mapping as the \emph{environment}. For every set $U \subseteq \Zd$, we denote by $\F(U)\subseteq\F$ the $\sigma$-algebra generated by the mappings $ \left( \a(e) \right)_{e \in \mathcal{B}_d(U)}$.

We fix a probability measure $\P_0$ supported in $\{0\}\cup[\lambda,1]$ satisfying the property
\begin{equation} \label{assumption.00}
\p := \P_0 \left( [\lambda,1]  \right) > \pc(d).
\end{equation}
where $\pc(d)$ is the bond percolation threshold for the lattice $\Zd$.
We then equip the measurable space~$(\Omega,\F)$ with the i.i.d. probability measure $\P = \P_0^{\otimes \mathcal{B}_d}$, so that the sequence $(\a(e))_{e \in \mathcal{B}_d}$ is an i.i.d. collection of random variables of law $\P_0$. The expectation with respect to the probability measure~$\P$ is denoted by $\E$. 

Given an environment $\a$, we say that a bond $e\in\mathcal{B}_d$ is \emph{open} if $\a(e) >0$ and \emph{closed} if $\a(e)=0$. Given two vertices $x,y\in\Zd$, we say that there is a \emph{path connecting $x$ and $y$} if there exists a sequence of open edges of the form $\{x,z_1\},  \ldots, \{z_{n},z_{n+1}\},\ldots,\{z_N,y\}$. The two vertices $x$ and $y$ are then said to be \emph{connected}, which we denote by $x \aconn y$, if there exists a path connecting $x$ and $y$. A \emph{cluster} is a connected subset $\C\subseteq \Zd$. Thanks to the assumption~\eqref{assumption.00}, we know that, $\P$--almost surely, there exists a unique maximal infinite cluster (see~\cite{BK}). This cluster is denoted by $\C_\infty := \C_\infty (\a).$

We also denote by $\Ed:= \left\{ (x,y)\,:\, x,y\in\Zd, x\sim y \right\}$ the set of oriented edges. More generally, we define, for a subset $U \subseteq \Zd$, $\Ed(U):= \left\{ (x,y)\,:\, x,y\in U, x\sim y \right\}$.

For $x \in \Zd$, we define the translation $\tau_x$ on $\Omega$ to be the map
\begin{equation*}
\tau_x  \, \colon \left\{
  \begin{array}{lcl}
    \Omega  & \to & \Omega, \\
     \left( \omega_{e} \right)_{e \in \mathcal{B}_d} & \mapsto & \left( \omega_{e + x} \right)_{e \in \mathcal{B}_d}. \\
  \end{array} \right.
\end{equation*}
Note that the measure $\P$ is stationary with respect to the $\Zd$-translations: for each $x \in \Zd$,
\begin{equation} \label{Pisstat}
\left( \tau_x \right)_* \P = \P,
\end{equation}
where $\left( \tau_x \right)_* \P$ is the pushforward measure defined by the formula, for each set $A \in \F,$ $\left( \tau_x \right)_* \P (A) = \P\left( \tau_x^{-1} \left( A \right) \right)$.
 
\subsubsection{Notation for functions}
For each vector $p \in \Rd,$ we denote by $l_p$ the affine function of slope $p$, i.e., $l_p(x) = p \cdot x$. Given a function $u$ defined on a discrete set $U \subseteq \Zd$, we define its oscillation by the formula
\begin{equation*}
    \osc_U u := \sup_U u - \inf_U u.
\end{equation*}
We define a \emph{vector field} to be a function $G: \Ed \to \R$ satisfying the antisymmetry property:  for each $(x,y)\in \Ed,$ $$G(x,y) = -G(y,x).$$ For a given a function $u: \Zd \to \R$, we define its gradient $\nabla u$ to be the vector field
\begin{equation*} \label{}
(\nabla u)(x,y) := u(x)-u(y).
\end{equation*}
For a random function defined on a cluster $\C$, $u: \C\to \R$, we define $\nabla u$ to be the vector field defined on the edges of $\Zd$ by the formula
\begin{equation} \label{defvectorfieldcluster}
(\nabla u)(x,y) :=  \left\{
  \begin{array}{lcl}
     u(x)-u(y) & \mbox{if } x,y \in \C \mbox{ and } \a \left( \left\{ x,y \right\} \right) \neq 0, \\
     0 & \mbox{otherwise}. \\
  \end{array}
\right.
\end{equation}
and $\a\nabla u$ to be the vector field defined by
\begin{equation*} \label{}
\left( \a \nabla u\right)(x,y) := \a\left( \{x,y\} \right) (\nabla u) (x,y) . 
\end{equation*}
The cluster $\C$ will frequently be the infinite cluster $\C_\infty.$ We may also think of the gradient as a vector-valued operator, as it is commonly the case for continuous functions: we denote by, for any point $x \in \C$ and any function $u : \C \to \R$,
\begin{equation} \label{eq:18321302}
    \nabla u(x) := \begin{pmatrix}
\nabla u (x + \e_1 , x) \\[3mm]
\vdots \\[3mm]
\nabla u (x + \e_d , x) 
\end{pmatrix}.
\end{equation}
For $p \in\Rd$, we denote by $p$ the constant vector field, defined according to the formula
\begin{equation*} \label{}
p(x,y) := p\cdot (x-y). 
\end{equation*}
With these conventions, we have $\nabla l_p = p$. For a given vector field $G$ and a point $x \in \Zd$, we define 
\begin{equation} \label{e.magF} 
\left| G \right|(x):= \left(  \sum_{ y : (x , y) \in \Ed } \left| G(x,y) \right|^2 \right)^{\frac12}.
\end{equation}
For a given a subset $U \subseteq \Zd$, we equip the space of vector fields with a scalar product $\left\langle \cdot , \cdot \right\rangle$, defined by
\begin{equation*} \label{}
\left\langle F,G \right\rangle_U 
:= \sum_{(x,y) \in \Ed(U)} F(x,y) G(x,y).
\end{equation*}
We will also frequently make use of the following notation, given a vector field $G$, we define
\begin{equation*}
  \left \langle G \right\rangle_U =\sum_{(x,y) \in \Ed(U) } G(x,y) (x- y).
\end{equation*}
The value $\left \langle G \right\rangle_U$ belongs to the space $\Rd$. Given an environment $\a$, two functions $u,v:\Zd \to \R$, and a subset $U \subseteq \Zd$, the Dirichlet form can be written with the previous notation as
\begin{equation*} \label{}
\left\langle \nabla u,\a\nabla v \right\rangle_U =  \sum_{(x,y) \in \Ed(U)} \left( u(x)- u(y) \right) \a\left(\{x,y\} \right) \left( v(x) - v(y) \right). 
\end{equation*}
We define the elliptic operator $-\nabla\cdot\a\nabla$ by, for each function $u:\Zd \to \R$ and each point $x\in \Zd$,
\begin{equation*} \label{}
\left( -\nabla\cdot\a\nabla u\right)(x):= \sum_{x\sim y}\a(\{x,y\})(u(x)-u(y)). 
\end{equation*}
For a  given a subset $U \subseteq \Zd$, we define the random set of $\a$-harmonic functions in $U$ by
\begin{equation*}
\A(U) := \left\{ u : U \to \R \, : \, \left( -\nabla\cdot\a\nabla u\right)(x) =0, \, x \in \mathrm{int}_\a U \right\},
\end{equation*}
where $\mathrm{int}_\a U$ is the interior of $U$ with respect to the environment $\a$, defined according to the formula
\begin{equation*}
\mathrm{int}_\a U := \left\{ x \in U ~:~ \forall y \in \Zd, ~\left(  y \sim x \mbox{ and } \a(\{x,y\}) \neq 0  \right) \implies y \in U  \right\}.
\end{equation*}
If $U$ is a finite set, we denote its cardinality by $|U|$.

\smallskip

For vectors of $\Rd$, we denote by $|\cdot|$ the standard infinite norm given by $|x| = \max_{i=1,\ldots,d} |x_i|$. We define a pseudometric on the subsets of $\Zd$ by $\dist(U,V) = \inf_{x\in U, y \in V} |x-y|$.

We also use the notations $B_R(x)$ or $B(x , R)$ to denote the ball centered at $x \in \Zd$ with radius $R > 0$ with respect to the infinite norm. The ball $B_R(0)$ is simply denoted by $B_R$. 

\smallskip

\subsubsection{Notation for cubes} A \emph{cube} is a subset of $\Zd$ of the form 
\begin{equation*} \label{}
\cu := \left( z + \left( -N , N \right)^d \right) \cap \Zd, \,  N \in \N,  z\in \Zd. 
\end{equation*}
We define the \emph{center} and the \emph{size} of the cube $\cu$ to be the point $z\in \Zd$ and the integer $2N -1$. We denote its size by $\size(\cu)$. In particular, with this convention, we have $|\cu| = \left( \size(\cu) \right)^d$.  For a non-negative real number $r>0$ and a cube $\cu$, of center $z \in \Zd$ and size $(2N-1) \in \N$, we denote by $r\cu$ the cube
\begin{equation*}
r\cu := \left( z + \left( -rN, rN \right)^d \right) \cap \Zd.
\end{equation*}
This notation is non-standard; the multiplication by $r$ only affects the size of the cube but the center of the cube remains unchanged. We introduce a specific category of cubes, namely the \emph{triadic cubes}. A triadic cube is a cube of the form
\begin{equation} \label{d.triadiccube}
\cu_n(z):=  \left( z + \left( -\frac12 3^n, \frac12 3^n \right)^d \right) \cap \Zd , \, n \in\N, \, z\in 3^n\Zd. 
\end{equation}
To simplify the notation, we write $\cu_n = \cu_n(0)$. This collection of cubes enjoys a number of convenient properties. First, any two triadic cubes (of possibly different sizes) are either disjoint or else one is included in the other. Moreover, for every $m,n\in\N$ with $n\leq m$, the triadic cube $\cu_m$ can be uniquely partitioned into $3^{d(m-n)}$ disjoint triadic cubes of size $3^n$. We denote by $\T$ the collection of triadic cubes and by $\T_n$ the collection of triadic cubes of size $3^n$. 

For each integer $n \in \N$ and each cube $\cu \in \T_n$, we define the predecessor of $\cu$, to be the unique triadic cube $\tilde{\cu} \in \T_{n+1}$ such that $\cu \subseteq \tilde{\cu}$. If $\tilde{\cu}$ is the predecessor of $\cu$, then we say that $\cu$ is a successor $\tilde{\cu}$. In particular, a cube of the set $\T_0$ does not have any successor, while a cube of the set $\T \setminus \T_0$ has exactly $3^d$ successors.

\subsubsection{The $\O_s$ notation}
We introduce a series of notations and properties which will be useful to measure the stochastic integrability and sizes of random variables. Given two parameters $s,\theta >0$ and a non-negative random variable $X$, we denote by
\begin{equation*} \label{}
X \leq \O_s(\theta) \mbox{ if and only if } \E \left[  \exp\left( \left( \frac{X}{\theta} \right)^s \right)   \right] \leq 2.  
\end{equation*}
Note that, by Markov's inequality, the tail of a random variable $X$ satisfying the inequality $X\leq \O_s(\theta)$ decreases stretched exponentially fast: for every $t >0$, 
\begin{equation*} \label{}
\P \left[ X \geq \theta t  \right]  \leq 2\exp\left( -t^s \right).
\end{equation*}
For a given sequence $(Y_i)_{i\in \N}$ of non-negative random variables and a sequence $(\theta_i)_{i\in \N}$ of non-negative real numbers, we write 
\begin{equation*} \label{}
X \leq\sum_{i \in \N} Y_i \O_s(\theta_i),
\end{equation*}
to mean that there exists a sequence of non-negative random variables $(Z_i)_{i\in \N}$ such that for each integer $i\in \N$, $Z_i \leq \O_s(\theta_i)$ and
\begin{equation*}
X \leq\sum_{i \in \N} Y_i Z_i.\end{equation*}
We now record some properties pertaining to this notation. All these properties are proved in~\cite[Appendix A]{armstrong2017quantitative}. The notation is compatible with the addition, meaning that, for any stochastic integrability exponent $s>0$, there exists a constant $C$ depending only on $s$, which may be chosen to be~$1$ if $s\geq 1$, such that
\begin{equation}
\label{e.Osums}
X_1 \leq \O_s(\theta_1) \mbox{ and } X_2 \leq \O_s(\theta_2) \implies X_1 +X_2 \leq \O_s(C(\theta_1+\theta_2)).
\end{equation}
More generally, for any $s>0$, there exists a constant $C(s)<\infty$ such that, for every measure space $(X,\F,\mu)$, every jointly measurable family $\{ X(x) \}_{x\in E}$ of non-negative random variables and every measurable function $\theta : E \rightarrow \R_+$, we have
\begin{equation} \label{e.Oavgs}
\forall x\in E, \ X(x) \leq \O_s(\theta (x)) \implies \int_E X(x) \, d\mu(x) \leq \O_s \left(C \int_E \theta(x) \, d\mu(x) \right). 
\end{equation}
The constant can be chosen to be
\begin{equation} \label{est.cte.Oavgs}
 \left\{
  \begin{array}{lcl}
     C(s) = \left( \frac{1}{s\ln 2}\right)^{\frac 1s} & \mbox{if } s < 1 \\
     C(s) = 1 & \mbox{if } s \geq 1. \\
  \end{array}
\right.
\end{equation} 
From the definition, we have, for each $\lambda \in \R_+$,
\begin{equation*}
X \leq \O_s(\theta) \implies \lambda X \leq \O_s(\lambda \theta).
\end{equation*}
This notation is also compatible with the multiplication in the sense that
\begin{equation}
\label{e.Oprods}
|X_1| \leq \O_{s_1}(\theta_1)  \mbox{ and }  |X_2| \leq \O_{s_2}(\theta_2) \implies |XY| \leq \O_{\frac{s_1s_2}{s_1+s_2}}\left( \theta_1 \theta_2 \right). 
\end{equation}
It is easy to check from~\eqref{e.Oprods} that one can reduce the integrability exponent $s$, i.e., for each $0 < s' < s$, there exists a constant $C := C(s') < \infty$ such that 
\begin{equation}
\label{e.improves}
X \le \O_{s}(\theta_1) \implies X \le \O_{s'} (C \theta_1).
\end{equation}

\subsubsection{Convention for constants and exponents} \label{subsecconcstexp}
In this article, the symbols $c$ and $C$ denote positive constants which may vary from line to line. These constants depend mainly on three parameters which are fixed through the proofs: the dimension of the space $d$, the ellipticity $\lambda$ and the probability $\p = \P \left[ \a(e) \neq 0 \right]$. Usually, we use $C$ for large constants (whose value is expected to belong to the interval $[1, \infty)$) and $c$ for small constants (whose value is expected to be in the interval $(0,1]$). 

For the stochastic integrability, we use the letter $s$ and will typically have inequalities of the form $X \leq \O_s(C)$. This exponent $s$ depends on the parameters $d , \lambda$ and $\p$. Its value can also vary from line to line and is expected to be small.

In Sections~\ref{section4} and~\ref{section5}, another parameter will be involved in the dependence of the constants and exponents: the spatial integrability $q \in (2, \infty)$ (see Theorem~\ref{mainthm} below). The dependence in this additional parameter will be displayed thanks to the following convention: we write $C := C(d,\lambda, \p) < \infty$ (resp. $C := C(d,\lambda, \p, q) < \infty$) to mean that the constant $C$ depends only on the parameters $d,\lambda, \p$ (resp. $d,\lambda, \p, q$) and that its value is expected to be large. For small constants or exponents we use the notations $c := c(d , \lambda, \p) > 0$, $s := s(d , \lambda, \p) > 0$ (resp. $c := c(d , \lambda, \p, q) > 0$, $s := s(d , \lambda, \p,q) > 0$).

\subsection{Outline of the paper} The rest of the paper is organized as follows. In Section~\ref{section2}, we recall (mostly without proof) some properties of the infinite cluster which were stated and proved in~\cite{AD2} (and based on~\cite{PP96}) to develop a quantitative homogenization theory on the infinite percolation cluster. In Sections~\ref{subsecspectral} and~\ref{subsecpoincar}, we state the concentration inequality and the multiscale Poincar\'e inequality, which are the two key ingredients in the proof of Theorem~\ref{mainthm}. In Section~\ref{section3}, we use the concentration inequality and the properties of the infinite cluster recorded in Section~\ref{section2} to obtain an estimate on the spatial averages of the corrector. In Section~\ref{section4}, we use the result established in Section~\ref{section3} combined with the multiscale Poincar\'e inequality to prove the optimal $L^q$-bound on the gradient of the corrector, stated in the following theorem.

\begin{theorem}[Optimal $L^q$ estimates for first-order corrector] \label{mainthm}
\begin{sloppypar} For each $q \geq 2$, there exist an exponent $s := s(d, \p , \lambda) > 0$ and a constant $C(d, \p, \lambda, q) < \infty$ such that for each radius $R \geq 1$ and each $p \in \Rd$,\end{sloppypar}
\begin{equation} \label{e.mainthmbis}
\left( R^{-d} \sum_{x \in \C_\infty \cap B_R} \left| \chi_p(x) - \left( \chi_p \right)_{\C_\infty \cap B_R} \right|^q  \right)^\frac 1q \leq  \left\{
  \begin{array}{lcl}
    \O_s \left( C |p| \log^\frac12 R \right) & \mbox{if } d=2, \\
     \O_s \left( C|p| \right) & \mbox{if } d\geq 3. \\
  \end{array}
\right.
\end{equation}
\end{theorem}
This theorem is strictly weaker than Theorem~\ref{realmainthm}; in Section~\ref{section5} we upgrade the previous $L^q$ bound into the $L^\infty$ bound stated in Theorem~\ref{realmainthm}. In Appendix~\ref{appA}, we give a proof of the multiscale Poincar\'e inequality stated in Section~\ref{subsecpoincar}. In Appendix~\ref{appB}, we give the proof of a technical lemma used in Section~\ref{section3}.

\bigskip

\noindent \textbf{Acknowledgement.} I would like to thank Scott Armstrong, Jean-Christophe Mourrat and Chenlin Gu for helpful discussions and comments.

\section{Preliminaries} \label{section2}
In this section we record some properties about the infinite percolation cluster in the supercritical regime. Most of these properties were established in~\cite{AD2}.

\subsection{The corrector : Existence and first properties}

Denote by $\mathcal{A}_1$ the (random) vector space of $\a$-harmonic functions with at most linear growth, i.e.,
\begin{equation*}
\mathcal{A}_1 := \left\{ u : \C_\infty \rightarrow \R~ :~ \nabla \cdot \left( \a \nabla u \right) = 0 ~ \mbox{in}~ \C_\infty~ \mbox{and} ~ \lim\limits_{R  \rightarrow \infty} \frac{1}{R^{ 2}} \left\| u \right\|_{\underline{L}^2 \left(\C_\infty \cap B_R \right)}= 0 \right\}
\end{equation*}

By~\cite{BDKY}, we know that, $\P$-almost surely, the space $\mathcal{A}_1$ has dimension $(d+1)$ and that every function $u \in \mathcal{A}_1$ is close to an affine function. More precisely, the space $\mathcal{A}_1$ can be characterized as follows: there exist a collection of sublinear functions $\left\{ \chi_p \right\}_{p \in \Rd}$ defined on the infinite cluster and valued in $\R$ such that
\begin{equation*}
    \mathcal{A}_1 := \left\{ l_p  + \chi_p + c \, : \, p \in \Rd , c \in \R  \right\}.
\end{equation*}
The functions $\left\{ \chi_p \right\}_{p \in \Rd}$ are called the correctors. They are defined up to a constant and are unique. To work with these quantities, one has to be careful to only consider quantities which are invariant by adding a constant, such as the oscillation, the gradient, the difference $\chi_p(x) - \chi_p(y)$, etc. For later use, we record that the map $p \mapsto \nabla \chi_p$ is linear.

The sublinear growth of the corrector is an important property which was proved qualitatively in~\cite{de1989invariance} and quantitatively in~\cite{AD2}; by~\cite[(1.22)]{AD2}, there exist two exponents 
$\delta := \delta (d , \p , \lambda) > 0$, $s := s (d , \p , \lambda) > 0$ and a constant $C := C(d, \p, \lambda)$ such that, for each radius $R \geq 1$,
\begin{equation} \label{e.infinitenormcorr}
\osc_{ \C_\infty \cap B_R} \chi_p \leq \O_s \left( C |p| R^{1-\delta} \right).
\end{equation}

We may reformulate this property in terms of a minimal scale: by~\cite[(1.18)]{AD2}, there exists a non-negative random variable $\X$ satisfying $\X \leq \O_s(C),$ such that for each vector $p \in \Rd$ and each radius $R \geq \X$,
\begin{equation} \label{minscale.sublin}
\left\| \chi_p - \left( \chi_p \right)_{\C_\infty \cap B_R }  \right\|_{\underline{L}^2 \left( \C_\infty \cap B_R \right)}  \leq C |p| R^{1-\delta}.
\end{equation}
Moreover, the corrector satisfies the following stationarity property: for each $x,y \in \Zd$, each $p \in \Rd$ and each $z \in \Zd$,
\begin{equation} \label{e.stat.corr}
\left( \chi_p (x) - \chi_p(y) \right) \indc_{\{ x,y \in \C_\infty \} } (\a) = \left( \chi_p (x+z) - \chi_p(y+z) \right) \indc_{\{ z+ x, z+ y \in \C_\infty \} } (\tau_z \a).
\end{equation}

\medskip

\subsection{Triadic partitions of good cubes}
This section shows how to use the tools developed by Penrose and Pisztora~\cite{PP96} to obtain a renormalization structure of the infinite cluster of supercritical percolation.
\subsubsection{A general scheme for partition of good cubes}
The construction of the partition is accomplished by a stopping time argument reminiscent of a Calder\'on-Zygmund-type decomposition. We are given a notion of ``good cube" represented by an $\F$-measurable function which maps $\Omega$ into the set of all subsets of $\T$. In order words, for each environment $\a\in\Omega$, we are given a subcollection $\mathcal{G}(\a) \subseteq \T$ of triadic cubes. We think of $\cu\in\T$ as being a good cube if $\cu\in\mathcal{G}(\a)$. We frequently drop the dependence in~$\a$ and write $\mathcal{G}$ instead of $\mathcal{G}(\a)$. 

\begin{proposition}[Partition of good cubes, Proposition 2.1 of~\cite{AD2}]
\label{p.partitions}
Let $\G \subseteq \T$ be a random collection of triadic cubes, as above. Suppose that there exist constants $K,s>0$ such that
\begin{equation*}
\sup_{z\in 3^n\Zd} \P \left[ z+ \cu_n \not\in\G \right] \leq K \exp\left( -K^{-1}3^{ns} \right).
\end{equation*}
Then, $\P$--almost surely, there exists a partition $\S\subseteq \T$ of $\Zd$ into triadic cubes with the following properties:
\begin{enumerate}

\item[(i)] All predecessors of elements of $\S$ are good: for every $\cu,\cu'\in\T$, 
\begin{equation*}
\cu' \subseteq \cu \ \mbox{and} \ \cu'\in\S \implies \cu \in \G.
\end{equation*}

\item[(ii)] Neighboring elements of $\S$ have comparable sizes: for every $\cu,\cu'\in \S$ such that $\dist(\cu,\cu') \leq 1$, we have
\begin{equation*}
\frac13 \leq \frac{\size(\cu')}{\size(\cu)} \leq 3.
\end{equation*}

\item[(iii)] Estimate for the coarseness of $\S$: if we denote by $\cu_\S(x)$ the unique element of $\S$ containing a point $x\in\Zd$, then there exists $C(s,K,d)<\infty$ such that
\begin{equation*}
\size\left( \cu_\S(x) \right) \leq \O_s (C). 
\end{equation*}
\end{enumerate}
In addition, if one has the following independence property,  for every cube $\cu = z + \cu_n \in \T$,
\begin{equation}
\label{e.goodlocal}
\mbox{the event $\left\{ \cu \not\in\G \right\}$ is $\F(z + \cu_{n+1})$-measurable,}
\end{equation}
then one has the following minimal scale property:

\begin{enumerate}
\item[(iv)] Minimal scale for $\S$. For each $t \in [1, \infty )$, there exists $C := C(t,s,K,d) < \infty$, an $\N$-valued random variable $\M_t (\S)$ and exponent $r := r(t,s,K,d) > 0$ such that
\begin{equation*}
\M_t(\S) \leq \O_r(C)
\end{equation*}
and for each integer $m \in \N$ satisfying $3^m \geq \M_t(\S)$,
\begin{equation*}
\frac{1}{|\cu_m|}\sum_{x \in \cu_m} \size \left( \cu_{\S} (x) \right)^t \leq C \quad \mbox{ and } \quad \sup_{x \in \cu_m} \size \left( \cu_{\S} (x) \right) \leq 3^{\frac{dm}{d+t}}.
\end{equation*}
\end{enumerate}
\end{proposition}

\subsubsection{The partition $\Pa$ of well-connected cubes}

We apply the construction of the previous subsection to obtain a random partition~$\Pa$ of $\Zd$ which simplifies the geometry of the percolation cluster. This partition plays an important role in the rest of the paper. To obtain bounds on the ``good event" which allows us to construct the partition, we use the results of Pisztora~\cite{P96}, Penrose and Pisztora~\cite{PP96} and Antal and Pisztora~\cite{AP96}. We first recall some definitions introduced in those works. 

\begin{definition}[Crossability and crossing cluster]
We say that a cube $\cu$ is \emph{crossable} (with respect to an environment $\a \in\Omega$) if each of the $d$ pairs of opposite $(d-1)$--dimensional faces of the cube $\cu$ is joined by an open path in $\cu$. We say that a cluster~$\C \subseteq \cu$ is a \emph{crossing cluster for $\cu$} if $\C$ intersects each of the $(d-1)$--dimensional faces of $\cu$. 
\end{definition}

\begin{definition}[Good cube] \label{def.goodcube}
We say that a triadic cube $\cu\in\T$ is \emph{well-connected} if there exists a crossing cluster~$\C$ for the cube $\cu$ such that:
\begin{enumerate}

\item[(i)] Each cube $\cu'$ with $\size(\cu') \in \left[ \frac{1}{10} \size(\cu), \frac 12 \size(\cu) \right]$ and $\cu' \cap \frac 34 \cu \neq \emptyset$ is crossable;

\item[(ii)] Every path $\gamma\subseteq \cu'$ with $\diam(\gamma) \geq \frac{1}{10} \size(\cu)$ is connected to $\C$ within $\cu'$.

\end{enumerate}
We say that~$\cu\in\T$ is a \emph{good cube} if $\size(\cu) \geq 3$, $\cu$ is well-connected and each of the~$3^d$ successors of the cube~$\cu$ are well-connected. We say that $\cu\in\T$ is a \emph{bad cube} if it is not a good cube (see Figure~\ref{fig1}). 
\end{definition}

\begin{figure}
    \centering
    \includegraphics[scale= 0.6]{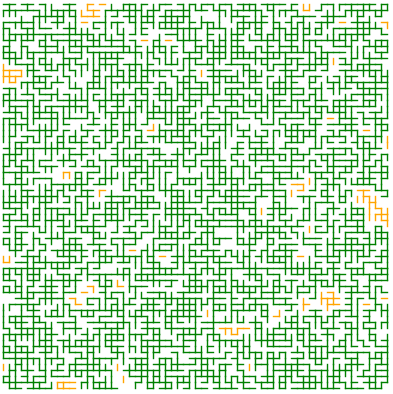}
    \caption{A good cube $\cu$. The cluster $\C_*(\cu)$ is drawn in green. Simulation by C.Gu}
    \label{fig1}
\end{figure}

The following estimate on the probability of the cube $\cu_n$ being good is a consequence~\cite[Theorem 3.2]{P96} and~\cite[Theorem 5]{PP96}, as recalled in~\cite[(2.24)]{AP96}.

\begin{lemma}[{\cite[(2.24)]{AP96}}]
\label{l.AP}
For each probability $\p \in (\p_c , 1 ]$, there exists a constant $C(d,\p)<\infty$ such that, for every integer $n\in\N$, 
\begin{equation} \label{e.pgoodness}
\P \left[ \cu_n \ \mbox{is good} \right] \geq 1 - C \exp\left( - C^{-1}3^{n} \right).
\end{equation}
\end{lemma}

It follows from Definition~\ref{def.goodcube} that, for every good cube $\cu$, there exists a unique maximal crossing cluster for $\cu$ which is contained in $\cu$. We denote this cluster by $\C_*(\cu)$. In the next lemma, we record the observation that adjacent triadic cubes which have similar sizes and are both good have connected clusters. 

\begin{lemma}[Lemma 2.8 of~\cite{AD2}]
\label{l.connectivity}
Let $n,n'\in\N$ with $|n-n'|\leq 1$ and $z,z'\in 3^n\Zd$ such that 
\begin{equation*} \label{}
\dist\left(\cu_{n}(z), \cu_{n'}(z') \right) \leq1.
\end{equation*}
Suppose also that $\cu_{n}(z)$ and $\cu_{n'}(z')$ are good cubes. Then there exists a cluster~$\C$ such that 
\begin{equation*} \label{}
\C_*( \cu_{n}(z) ) \cup \C_*( \cu_{n'}(z') )\subseteq  \C \subseteq \cu_{n}(z) \cup \cu_{n'}(z').
\end{equation*}
\end{lemma}
We next define the partition~$\Pa$ of good cubes.

\begin{definition}
\label{d.Pa}
We let $\Pa\subset \T$ be the partition~$\S$ of $\Zd$ obtained by applying Proposition~\ref{p.partitions} to the collection
\begin{equation*} \label{}
\G:= \left\{ \cu\in \T\,:\, \cu \ \mbox{ is good} \right\}.
\end{equation*}
More generally, for each point $y \in \Zd$, we let $\Pa_y\subset \T$ be the partition~$\S$ of $\Zd$ obtained by applying Proposition~\ref{p.partitions} to the collection
\begin{equation*} \label{}
\G:= \left\{ y + \cu:\, \cu \in \T\, \mbox{and} \, y + \cu \mbox{ is good} \right\}.
\end{equation*}
From the construction of $\Pa$ and $\Pa_y$, we also have
\begin{equation*}
\Pa_y = y + \Pa(\tau_{-y} \a) = \left\{ y + \cu \, : \, \cu \in \Pa(\tau_{-y} \a) \right\}.
\end{equation*}
We refer to Figure~\ref{fig2} for an illustration of the random partition~$\Pa$.
\end{definition}

\begin{figure}
    \centering
    \includegraphics[scale= 0.7]{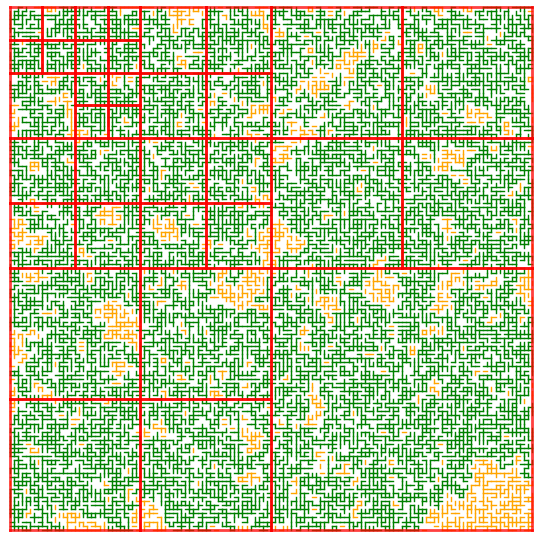}
    \caption{A realization of the partition $\Pa$ in a cube $\cu$. The cluster $\C_*(\cu)$ is drawn in Green; the clusters in yellow are the small isolated clusters. Simulation by C. Gu.}
    \label{fig2}
\end{figure}

The (random) partition $\Pa$ plays an important role throughout the rest of the paper. We denote by $\Pas$ the collection of triadic cubes which contains some elements of $\Pa$, that is
\begin{equation*} \label{}
\Pas:= \left\{ \cu\,:\,  \mbox{$\cu$ is a triadic cube and $\cu \supseteq \cu'$ for some $\cu'\in \Pa$} \right\}.
\end{equation*}
Notice that every element of $\Pas$ can be written in a unique way as a disjoint union of elements of $\Pa$.
According to Proposition~\ref{p.partitions}(i), every triadic cube containing an element of $\Pa$ is good. By Proposition~\ref{p.partitions}(iii) and Lemma~\ref{l.AP}, there exists $C(d,\p)<\infty$ such that, for every $x\in\Zd$,
\begin{equation}
\label{e.partitionO1}
\size\left( \cu_{\Pa}(x) \right) \leq \O_1(C).
\end{equation}
By the properties of the partition~$\Pa$ given in Proposition~\ref{p.partitions}(i) and (ii) and Lemma~\ref{l.connectivity}, the maximal crossing cluster $\C_*(\cu)$ of an element $\cu\in \Pas$ must satisfy $\C_*(\cu) \subseteq \C_\infty$, since the union of all crossing clusters of the elements of $\Pa$ is unbounded and connected. Notice also that, although we may not have $\C_*(\cu) = \C_\infty \cap \cu$, by definition of the partition $\Pa$ and Property~(ii) of Definition~\ref{def.goodcube}, we have that, for every cube $\cu \in \Pa$, there exists a cluster $\C$ such that
\begin{equation} \label{cubeandneighbors}
\C_\infty \cap \cu \subseteq \C \subseteq \bigcup_{\cu' \in \Pa,~ \dist \left( \cu, \cu' \right) \leq 1} \cu'.
\end{equation} 
In other words, for any cube $\cu \in \Pa$ and every pair of points $x, y \in \C_\infty \cap \cu$, there exists a path connecting $x$ to $y$ which lies in the cube $\cu$ and in its neighbors.

It is also interesting to note that, for $m \in \N$ such that $3^m \geq \M_{2d}\left( \Pa \right)$, the sets $\C_*(\cu_m),$  $\C_\infty \cap \cu_m$ and $\cu_m$ have comparable sizes: there exists a constant $C := C(d , \p) < \infty$ such that
\begin{equation} \label{e.sizecluster}
C^{-1} |\cu_m| \leq \left|\C_*(\cu_m)\right| \leq\left|  \C_\infty \cap \cu_m \right| \leq \left| \cu_m \right|.
\end{equation}
This result is a consequence of the Cauchy-Schwarz inequality and the three identities, under the assumption $3^m \geq \M_{2d} \left(\Pa \right)$ (which implies in particular that the cube $\cu_m$ is good),
\begin{equation*}
\sum_{\cu \in \Pa , \cu \subseteq \cu_m} 1 \leq \C_*(\cu_m), \qquad \sum_{\cu \in \Pa , \cu \subseteq \cu_m} \size \left( \cu_{\Pa} \right)^d = \left| \cu_m \right| \quad \mbox{ and } \quad \sum_{\cu \in \Pa , \cu \subseteq \cu_m} \size \left( \cu_{\Pa} \right)^{2d} \leq C \left| \cu_m \right|.
\end{equation*}
The first inequality comes from the fact that each cube of the partition $\Pa$ contained in the cube $\cu_m$ must have non-empty intersection with the cluster $\C_*(\cu_m)$, the second one is the preservation of the volume and the third one is a consequence of the assumption $3^m \geq \M_{2d} \left(\Pa \right)$.

Given a cube $\cu\in \Pa$, we let $\zbar(\cu)$ denote the element of the cluster~$\C_*(\cu)$ which is closest to the point~$z$; if this point is not unique, we break ties by using the lexicographical order. 

\begin{definition} \label{d.coarsenedfnct}
Given a function $u : \C_\infty \rightarrow \R$, we define \emph{the coarsened function with respect to the partition~$\Pa$} to be
\begin{equation*}
 \left[ u \right]_{\Pa}   : \left\{  \begin{array}{lcl}
 \Zd & \to & \R, \\
 x & \mapsto & u \left( \zbar \left(\cu_\Pa(x) \right) \right). \\
\end{array} \right.
\end{equation*}
\end{definition}

The reason we use the coarsened function is that it is defined on the entire lattice $\Zd$ and not on the infinite cluster. This allows to make use of the simpler and more favorable geometric structure of $\Zd$. The price to pay is the difference between $u$ and $ \left[ u \right]_{\Pa}$; it depends on the coarseness of the partition $\Pa$ and the control one has on the gradient of the function $u$ in a way that is made precise in the following proposition. The dependence on the coarseness of $\Pa$ is present via the size of the cubes of the partition. We recall the notation $|F|(x)$ for a vector field $F$ introduced in~\eqref{e.magF}.

\begin{proposition}[Lemma 3.2 of \cite{AD2}]
\label{l.coarseLs}
For each triadic cube~$\cu \in \Pa_*$, each exponent $1\leq s<\infty$ and each function $w: \C_\infty \to \R$,
\begin{equation}
\label{e.coarseLs}
\sum_{x \in \C_* \left( \cu \right)} \left| w(x) - \left[ w \right]_\Pa\!(x) \right|^s \\
\leq C^s \sum_{x \in \C_*(\cu)}  \size(\cu_\Pa(x))^{sd} \left| \nabla w\right|^s(x) .
\end{equation}
More generally, for any family of disjoint cubes $\left\{ \cu_i \right\}_{i \in I} \in \left( \Pa_* \right)^I$, we have
\begin{equation} \label{e.coarseLsgen}
\sum_{ x\in \C_*\left( \cup_{i \in I} \cu_i \right)} \left| w(x) - \left[ w \right]_\Pa\!(x) \right|^s \\
\leq C^s \sum_{ x \in \C_*\left( \cup_{i \in I} \cu_i \right)}  \size(\cu_\Pa(x))^{sd} \left| \nabla w\right|^s(x) ,
\end{equation}
where $\C_*\left( \cup_{i \in I} \cu_i \right)$ denotes the union of the maximal clusters of each connected component of the set~$ \bigcup_{i \in I} \cu_i$.
\end{proposition}

\begin{remark}
We do not have the identity $\C_*\left( \cup_{i \in I} \cu_i \right) = \bigcup_{i \in I}  \C_*\left( \cu_i \right)$. The problem is the same than the one of~\eqref{cubeandneighbors} and thus~\eqref{e.coarseLsgen} can not be directly obtained from~\eqref{e.coarseLs}. Nevertheless, we do have the inclusion
\begin{equation} \label{useablecubeandneighbor}
\C_\infty \cap \cu \subseteq \C \left( \bigcup_{\cu' \in \Pa,~ \dist \left( \cu, \cu' \right) \leq 1} \cu' \right).
\end{equation}
\end{remark}

Moreover we can control the $L^s$-norm of the vector field $\nabla  \left[ w \right]_\Pa$ by the $L^s$-norm of the map $\nabla w$ and the coarseness of the partition $\Pa$ thanks to the following proposition.

\begin{proposition}[Lemma 3.3 of \cite{AD2}] \label{p.sobcl}
\label{l.coarsegrads}
For each triadic cube~$\cu \in \Pa_*$, each exponent $1\leq s<\infty$ and each function $w: \C_\infty \to \R$,
\begin{equation} 
\label{e.coarsegrads}
\sum_{x \in \C_*(\cu)} \left| \nabla  \left[ w \right]_{\Pa} \right|^s(x)
\leq C^s \sum_{x \in \C_*(\cu)} \size(\cu_\Pa(x))^{sd-1} \left|\nabla w \right|^s(x).
\end{equation}
More generally, for any family of disjoint cubes $\left\{ \cu_i \right\}_{i \in I} \in \left( \Pa_* \right)^I$, we have
\begin{equation} \label{e.coarsegradsgen}
\sum_{x \in \C_*\left( \cup_{i \in I} \cu_i \right)} \left| \nabla  \left[ w \right]_{\Pa} \right|^s(x)
\leq C^s \sum_{x \in \C_*\left( \cup_{i \in I} \cu_i \right)} \size(\cu_\Pa(x))^{sd-1} \left|\nabla w \right|^s(x).
\end{equation}
\end{proposition}

\subsection{Solving the Poisson equation with divergence form source term}
In this section we study the existence and uniqueness of the solution of the equation $-\nabla \cdot \a \nabla u = - \nabla \cdot \xi$ on the infinite cluster $\C_\infty$. We denote by $\sum_{e \subseteq \C_\infty}$ the sum over all the edges of the infinite cluster.

\begin{proposition}[Gradient of Green's function] \label{p.green}
Given an environment $\a \in \Omega$, we select an edge $e=(x,y) \in E_d$ such that the points $x$ and $y$ belong to the infinite cluster $\C_\infty$. There exist a constant $C := C(d, \lambda) < \infty$ and a function $\nabla G(e , \cdot) : \C_\infty \rightarrow \R$, whose gradient with respect to the second variable, denoted by $\nabla \nabla G $, satisfies
\begin{equation} \label{gradientgreenbounded}
 \left\langle \nabla \nabla G(e , \cdot) , \nabla \nabla G (e , \cdot) \right\rangle_{\C_\infty} \leq C ,
\end{equation}
and is a solution to the equation
\begin{equation*}
-\nabla \cdot \a \nabla \left( \nabla G(e, \cdot ) \right) = \delta_x - \delta_y ~ \mbox{in} ~ \C_\infty.
\end{equation*}
Moreover, we have, for each pair of edges $e , e'$ of the infinite cluster,
\begin{equation} \label{symmetryGreen}
\nabla \nabla  G(e , e') = \nabla \nabla G (e', e).
\end{equation}
\end{proposition}

Proposition~\ref{p.green} can be used to solve the equation $-\nabla \cdot \a \nabla w_\xi = - \nabla \cdot \xi $. This is the objective of the following proposition.

\begin{proposition}\label{morecomplicatedgreen}
Let $\xi : E_d \to \R$ be a vector field satisfying
\begin{equation} \label{assumpsupportcluster}
\xi(x,y) = 0 \mbox{ if } \a(x,y) = 0 \mbox{ or } x,y \notin \C_\infty.
\end{equation}
If $\xi$ satisfies $\left\langle  \xi ,  \xi \right\rangle_{\C_\infty} < \infty$ then there exists a unique (a.s in the environment and up to a constant) solution $w_\xi$ of the equation
\begin{equation*}
-\nabla \cdot \a \nabla w_\xi = - \nabla \cdot \xi ~ \mbox{in} ~ \C_\infty.
\end{equation*}
Moreover, we have the following representation
\begin{equation} \label{representationformula}
\nabla w_\xi (\cdot)= \sum_{e \subseteq \C_\infty} \xi(e) \nabla \nabla G \left( e, \cdot \right) .
\end{equation}
\end{proposition}

\begin{remark}
We extend the definition of Proposition~\ref{morecomplicatedgreen} to vector-valued fields $\xi : E_d \to \R^k$. In that case, we will write
\begin{equation*}
w_\xi  := \left\{ \begin{array}{ccc}
 \C_\infty & \to &\R^k, \\
 z & \mapsto & \left( w_{\xi_1}(z) , \ldots, w_{\xi_1}(z) \right),\\
\end{array} \right.
\end{equation*}
where $\xi_1 , \ldots, \xi_k$ denote the components of the vector $\xi$; the formula~\eqref{representationformula} applies in this framework.
\end{remark}

\begin{proof}[Proof of Proposition~\ref{p.green} and~\ref{morecomplicatedgreen}]
Let $\xi$ be a vector field satisfying~\eqref{assumpsupportcluster} and $\left\langle  \xi ,  \xi \right\rangle_{\C_\infty} < \infty$. We denote by $\dot{H}^1$ the space of functions defined on the infinite cluster whose gradient is in the space $L^2(\C_\infty)$, i.e., $\dot{H}^1 := \left\{ u : \C_\infty \to \R ~:~ \left\langle \nabla u , \nabla u \right\rangle_{\C_\infty} < \infty \right\}$,
and consider the minimization problem
\begin{equation*}
\inf_{u \in \dot{H}^1 } \frac 12 \left\langle \nabla u , \a \nabla u \right\rangle_{\C_\infty} - \left\langle \xi , \nabla u \right\rangle_{\C_\infty}.
\end{equation*}
By the standard techniques of the calculus of variations, there exists a unique solution (up to  a constant) to this problem denoted by $w_\xi$. In particular, when $\xi$ is the indicator of an edge $e$, we obtain the function $\nabla G \left( e, \cdot \right)$.
To prove the identity~\eqref{symmetryGreen}, we note that
\begin{equation*}
\nabla \nabla G( e', e) = \left\langle \nabla \nabla G(e , \cdot) ,\a  \nabla \nabla G (e' , e) \right\rangle_{\C_\infty}  = \nabla \nabla G (e, e').
\end{equation*}
The representation formula~\eqref{representationformula} follows from standard arguments.
\end{proof}

\subsection{Regularity theory}
In this subsection, we record a result from the regularity theory established in~\cite{AD2} giving a Lipschitz bound for the gradient of $\a$-harmonic functions.

\begin{proposition}[Regularity theory on $\C_\infty$, Theorem 2 of~\cite{AD2}]\label{p.regularitytheory} There exist a constant $C < \infty,$ an exponent $s > 0$ and a random variable $\X$ satisfying
\begin{equation} \label{Xthm2perco}
\X \leq \O_s(C),
\end{equation}
such that for each solution $u : \C_\infty \to \R$ of the equation
\begin{equation} \label{ellipticeqperco.00}
-\nabla \cdot \a \nabla u = 0 ~\mbox{in}~ \C_\infty
\end{equation}
and each pair of radii $R , r$ such that $R \geq r \geq \X$, we have
\begin{equation*}
\left\| \nabla u \right\|_{\underline{L}^2 \left(\C_\infty \cap B_r \right) } \leq  \frac CR \left\| u - \left( u \right)_{\C_\infty \cap B_R} \right\|_{\underline{L}^2 \left(\C_\infty \cap B_R \right) }.
\end{equation*}
\end{proposition}
We introduce the notation, for each point $x \in \Zd$,
\begin{equation*}
\X (x) := \X \circ \tau_x.
\end{equation*}

This proposition is weaker than Theorem 2 of~\cite{AD2}; it is indeed a consequence of the Caccioppoli inequality and Theorem 2 (iii) of~\cite{AD2} for $k=0$. As a consequence, we obtain the following Lipschitz bound on the corrector.

\begin{proposition}[Lipschitz bound on the corrector] \label{p.lipboundcorr}
There exists a constant $C < \infty$ and an exponent $s > 0$ such that, for each edge $e = (x , y) \in E_d$ and each vector $p \in \Rd$,
\begin{equation} \label{lipboundcorrbis}
\left| \nabla \chi_p (e) \right| \indc_{\{ e \in \C_\infty \}} \leq C |p| \X^{d/2}(x).
\end{equation} 
which implies, by~\eqref{Xthm2perco},
\begin{equation} \label{lipboundcorr}
\left| \nabla \chi_p (e) \right| \indc_{\{ e \in \C_\infty \}} \leq \O_s \left( C |p| \right),
\end{equation} 
for some smaller exponent $s$ (cf. Subsection~\ref{subsecconcstexp}). The same estimate holds for the coarsened corrector
\begin{equation} \label{lipboundcoarcorr}
\left| \nabla \left[ \chi_p \right]_\Pa (e) \right| \leq \O_s(C |p|). 
\end{equation}
\end{proposition}

\begin{remark}
The same estimate than~\eqref{lipboundcorrbis} would hold with the random variable $\X^{d/2}(y)$ instead of $\X^{d/2}(x)$ in the right-hand side.
\end{remark}

\begin{proof}
By the stationarity of the law of the corrector, we can assume that the edge $e$ touches $0$, i.e., that $x = 0$. First note that, for each radius $r \geq 1$,
\begin{equation*}
\left| \nabla \chi_p (e) \right| \indc_{\{ e \in \C_\infty\} } \leq r^d \left\| \nabla \chi_p \right\|_{\underline{L}^2 \left(\C_\infty \cap B_r \right) }  .
\end{equation*}
By applying Proposition~\ref{p.regularitytheory} with $r= \X$, and taking the limit $R \rightarrow \infty$, we obtain
\begin{equation*}
 \left\| p + \nabla \chi_p \right\|_{\underline{L}^2 \left(\C_\infty \cap B_\X \right)} \leq C \X^{d/2} \liminf_{R \rightarrow \infty } \frac 1R \| l_p + \chi_p  - \left( l_p + \chi_p \right)_{\C_\infty \cap B_{R}(x')}\|_{\underline{L}^2 \left( \C_\infty \cap B_{R}(x') \right)} \leq C \X^{d/2} |p|.
\end{equation*}
A combination of the two previous displays with the integrability estimate~\eqref{Xthm2perco} yields~\eqref{lipboundcorr}. To prove~\eqref{lipboundcoarcorr}, we combine~\eqref{lipboundcorr} with Proposition~\ref{p.sobcl} and use the integrability estimate $\size \left( \cu_\Pa(x) \right) \leq \O_s(C)$, valid for each point $x \in \Zd$. This is performed in the following computation: for each edge $e= (x,y) \in \Rd$, we have
\begin{align} \label{refcomputation}
 \left| \nabla \left[ \chi_p \right]_{\Pa} (e) \right| & \leq \sum_{x' \in \C_*\left(\cu_\Pa (x) \cup\cu_\Pa (y) \right)} \left| \nabla  \left[ \chi_p  \right]_{\Pa} \right|(x') \\
& \leq  C \sum_{x' \in \C_\infty \cap B \left( x , C \size \left( \cu_\Pa (x) \right) \right)} \size \left(\cu_\Pa(x')\right)^{d-1} \left|\nabla\chi_p  \right|(x') \notag \\
& \leq C \sum_{x' \in \Zd } \indc_{\{ x' \in \C_\infty \cap B \left( x , C \size \left( \cu_\Pa (x) \right) \right)\}}  \size \left(\cu_\Pa(x')\right)^{d-1} \left|\nabla\chi_p  \right|(x'). \notag
\end{align}
Using the estimate, for each point $x \in \Zd$, $\size \left( \cu_\Pa (x) \right) \leq \O_s(C)$, we obtain
\begin{equation} \label{estimateindicatorCC}
\indc_{\left\lbrace x' \in B \left(x, C \size \left( \cu_\Pa (x) \right) \right) \right\rbrace} \leq C \frac{ \size \left( \cu_\Pa (x) \right)^{d+1}}{|x-x'|^{d+1}\vee 1} \leq \frac{\O_s (C)}{|x-x'|^{d+1}\vee 1},
\end{equation}
where we used the notation $a \vee b := \max(a,b)$. Using the summability of the map $x \mapsto \left(|x|  \vee 1\right)^{-d-1}$, the properties~\eqref{e.Oavgs} and~\eqref{e.Oprods} on the $\O_s$ notation and the Lipschitz bounds~\eqref{lipboundcorr} on the corrector, we obtain the result.
\end{proof}

We now present the two main tools to prove Theorem~\ref{mainthm}. The first one is a concentration inequality, thanks to which we obtain quantitative control on the spatial averages of the gradient at scale $R$ (see Proposition~\ref{boundspatialaveragecorrector}). We then deduce Theorem~\ref{mainthm} from Proposition~\ref{boundspatialaveragecorrector} thanks to the multiscale Poincar\'e inequality (Proposition~\ref{multiscalepoincare}).

\subsection{Concentration inequality for stretched exponential moments}\label{subsecspectral}

The following concentration inequality is a key ingredient in the proof of Proposition~\ref{boundspatialaveragecorrector} in the next section; its proof can be found in~\cite[Proposition 2.2]{armstrong2017optimal}.
\begin{proposition}[Proposition 2.2 of~\cite{armstrong2017optimal}] \label{resampling}
Fix $\beta \in (0,2)$. Let $X$ be a random variable on $(\Omega,\F,\P)$ and set for each bond $e \in \mathcal{B}_d(\Zd)$,
\begin{equation*}
X_e' = \E \left[X | \F \left( \mathcal{B}_d \setminus \{ e \} \right) \right] ~\mbox{and}~ \mathbb{V}[X] = \sum_{e \in \mathcal{B}_d} \left(X - X_e' \right)^2,
\end{equation*}
then there exists a constant $C := C(d, \beta) < \infty$ such that
\begin{equation*}
\E \left[ \exp \left( \left| X - \E[X] \right|^{\beta}\right) \right] \leq C \E \left[ \exp \left( \left( C \mathbb{V}[X] \right)^{\frac{\beta}{2-\beta}}\right) \right]^{\frac{2-\beta}{\beta}}.
\end{equation*}
\end{proposition}

\subsection{Multiscale Poincar\'e inequality} \label{subsecpoincar}

The next proposition is a version of the multiscale Poincar\'e inequality. It controls the oscillations of a function in the $L^q$-norm (left-hand side of~\eqref{heatmultscaleP.ver}) by the spatial average of the gradient of the function (right-hand side of~\eqref{heatmultscaleP.ver}). We first introduce the discrete heat kernel.

\begin{definition}[Discrete heat-kernel] \label{def.heatkerneldisc}
Let $\Phi : [ 0 , \infty) \times \Zd \to \R$ be the discrete heat kernel on the lattice $\Zd$, i.e., the solution of the parabolic equation
\begin{equation*}
\left\{ \begin{array}{ccc}
\partial_t \Phi - \Delta \Phi & = 0 & ~\mbox{in}~(0 , \infty) \times \Zd, \\
\Phi(0 , \cdot) & = \delta_0 & ~\mbox{in}~ \Zd,
\end{array} \right.
\end{equation*}
where $\Delta$ denotes the discrete Laplacian on $\Zd$. We introduce the notation, for each radius $r > 0$, $\Phi_{r^2} := \Phi\left( r^2 , \cdot \right)$. It satisfies the estimate, for some constant $C := C(d) < \infty$,
\begin{equation} \label{heatkernbound}
    \left\{ \begin{aligned}
    \Phi(t , x)& \leq \frac{C}{t^{\frac{d}{2}}} \exp \left( - \frac{|x|^2}{Ct}\right) &~\mbox{if}~ t \geq |x| \hspace{5mm}(\mbox{Gaussian regime}), \\
    \Phi(t , x) &\leq C \exp \left(  - \frac{|x|}{C} \left( 1 + \ln \frac{|x|}{t}\right) \right) &~\mbox{if}~ t \leq |x| \hspace{5mm}(\mbox{Poisson regime}).
    \end{aligned}
    \right.
\end{equation}
We refer to~\cite{De99} and~\cite{Da93} for a proof of these inequalities.
\end{definition}

\begin{proposition}[Multiscale Poincar\'e inequality, heat kernel version] \label{multiscalepoincare}

For each exponent $q \geq 1$, there exists a constant $C:= C(d , q) < \infty$ such that for each function $u : \Zd \to \R$ and each radius $R>0$,
\begin{equation} \label{heatmultscaleP.ver}
\left\| u - (u)_{B_R}  \right\|_{\underline{L}^q\left(B_R\right)} \leq C  \left( \sum_{x \in \Zd} R^{-d}e^{ - \frac{|x|}{2R} } \left( \int_{0}^{2R} r \left| \Phi_{r^2} * \nabla u (x) \right|^2 \, dr  \right)^{\frac q2} \right)^\frac1q,
\end{equation}
where the operator $*$ is the standard discrete convolution on $\Zd$ between the heat kernel $\Phi_{r^2}$ and the vector valued function $\nabla u$ (see the definition~\eqref{eq:18321302}). Moreover the dependence in the variable $q$ of the constant $C$ can be quantified as follows, for each exponent $q \geq 2$,
\begin{equation*}
C (d, q ) \leq A q^{\frac 32}
\end{equation*} 
for some constant $A := A(d) < \infty$.
\end{proposition}
The proof of this proposition relies on~\cite[Proposition D.1 and Remark D.6]{armstrong2017quantitative} and is presented in Appendix~\ref{appA}.
 
\section{Estimates of the spatial averages of the first-order correctors}\label{section3}

We now have collected all the necessary tools to prove the optimal $L^q$ bounds of the corrector, stated in Theorem~\ref{mainthm}. The strategy is to first prove Proposition~\ref{boundspatialaveragecorrector} thanks to the concentration inequality (Proposition~\ref{resampling}). We then deduce the bound on the coarsened corrector thanks to the multiscale Poincar\'e inequality (Proposition~\ref{multiscalepoincare}) and remove the coarsening thanks to Proposition~\ref{l.coarseLs}. This eventually yields Theorem~\ref{mainthm}. In this section, we use the notation introduced in~\eqref{eq:18321302} and think of the gradient of the coarsened corrector as a vector-valued function. 

\begin{proposition}\label{boundspatialaveragecorrector}
For each $R \geq 1$, and each $x \in \Rd$, the quantity $  \left(\Phi_{r^2} * \nabla \left[ \chi_p \right]_{\Pa}\right) (x)$ is well-defined and there exist an exponent $s > 0$ and a constant $C  < \infty$ such that it satisfies
\begin{equation}\label{boundspatialaveragecorrector.est}
\left| \left( \Phi_{r^2} * \nabla \left[ \chi_p \right]_{\Pa} \right) (x) \right| \leq  \O_s \left(C |p| r^{-\frac d2} \right).
\end{equation}
\end{proposition}

By the stationarity of the gradient of the corrector, it is enough to prove the result when $x = 0$. By linearity of the mapping $p \mapsto \nabla \chi_p$, we may assume $|p| = 1$. We denote by $X =  \left( \Phi_{r^2} * \nabla \left[ \chi_p \right]_{\Pa} \right) (0) $ and prove
\begin{equation*}
\left|X  \right| \leq  \O_s \left(C r^{-\frac d2} \right),
\end{equation*}
The strategy of the proof is to apply Proposition~\ref{resampling} to the random variable $X$. We decompose the argument into two lemmas. The first one focuses on the expectation of $X$.

\smallskip

\begin{lemma} \label{lemma3.4}
There exists a constant $C < \infty$ such that
\begin{equation*}
\left| \E [X] \right| \leq C  r^{-\frac d2}.
\end{equation*}
\end{lemma}

\smallskip
The second one estimates the quantity $\mathbb{V}[X]$.
\begin{lemma} \label{lemma3.5}
There exist a constant $C < \infty$ and an exponent $s > 0$, such that
\begin{equation*}
\mathbb{V}[X] \leq \O_s\left( C r^{-d}\right).
\end{equation*}
\end{lemma}
These lemmas are proved in the following two sections.

\subsection{Estimating the expectation of the spatial averages.} The main objective of this section is to show Lemma~\ref{lemma3.4}.

\begin{proof}[Proof of Lemma~\ref{lemma3.4}] The strategy of the proof is to use the stationarity and the sublinearity of the corrector to prove that the expectation of its gradient is equal to $0$. The technical difficulty which arises is that the partition $\Pa$ is not stationary. This implies that the random variable $\nabla \left[ \chi_p \right]_{\Pa} (0)$ is not stationary. To fix this issue we introduce a partition $\Pa_\mathrm{stat}$ which is stationary and equal to $\Pa$ on a set of large probability. We finally show that the error we make by considering the partition $\Pa_\mathrm{stat}$ instead of $\Pa$ is small. 

For each triplet $x,y,z \in \Zd$ with $x \sim y$, denote by $\tau_z \a$ the translated environment defined by
\begin{equation*}
\tau_z \a (\{ x,y \}) = \a (\{ x - z,y -z \}).
\end{equation*}
For $k \in \N$, we construct the partition $\Pa_\mathrm{stat}^k$ by applying~Proposition~\ref{p.partitions} to the collection of triadic cubes
\begin{equation*}
\G_\mathrm{stat}^k := \G \bigcup  \left( \bigcup_{n = k}^{\infty} \T_n\right).
\end{equation*}
Note that this collection is not a set of good cubes in the sense of Definition~\ref{def.goodcube} but it is $3^k\Zd$-translation invariant. A straightforward consequence is that the partition $\Pa_\mathrm{stat}^k$ is $3^k\Zd$-stationary: for every environment $\a$, every point $x \in \Zd$ and $z \in 3^k \Zd$, one has
\begin{equation} \label{3kstationnarypartition}
\size \left( \cu_{\Pa_\mathrm{stat}^k} (x + z) \right) (\tau_z \a) = \size \left( \cu_{\Pa_\mathrm{stat}^k} (x ) \right) ( \a).
\end{equation}
With a proof similar to the one of~\cite[Proposition 2.1 (iv)]{AD2}, we obtain
\begin{equation}\label{errorstatpart}
\P \left[ \exists x \in \cu_k, \, \cu_\Pa (x) \neq  \cu_{\Pa_\mathrm{stat}^k} (x)\right] \leq C \exp \left( - C^{-1} 3^{k}\right).
\end{equation}
For a function $u: \C_\infty \rightarrow \R$, we define the coarsened function $\left[ u \right]_{\Pa_\mathrm{stat}^k}$ with respect to the partition $\Pa_{\mathrm{stat}}^k$ by the formula
\begin{equation*}
\left[ u \right]_{\Pa_\mathrm{stat}^k} := u\left(\bar{z}_{\mathrm{stat}}\left(\cu_{\Pa_\mathrm{stat}^k} (x)\right)\right)
\end{equation*}
with the notation, for $\cu \in \T$,
\begin{equation} \label{d.corsenstatio}
\bar{z}_{\mathrm{stat}}\left(\cu\right) : = \left\{
  \begin{array}{lcl}
    \bar{z}\left(\cu\right) \mbox{ if } \bar{z}\left(\cu \right)  \in \C_\infty \mbox{ and } \cu\mbox{is a good cube}, \\
      \underset{z \in \C_\infty} {\argmin}\dist \left( z, \cu \right) \mbox{ otherwise}.
  \end{array}
\right.
\end{equation}
If there is more than one choice in the argument of the minima, we select the one which is minimal for the lexicographical order. By the stationarity of the gradient of the corrector and~\eqref{3kstationnarypartition}, we have 
\begin{equation} \label{statcoarsencorrector}
\nabla \left[ \chi_p \right]_{\Pa_\mathrm{stat}^k} \mbox{ is } 3^k \Zd-\mbox{stationnary}.
\end{equation}
We let $k \in \Zd$ be the the integer such that $3^k \leq r^{\frac 12} \leq 3^{k+1}$ and split the proof of Lemma~\ref{lemma3.4} into three steps:
\begin{enumerate}
\item[(i)] In Step 1, we prove
\begin{equation*}
\E \left[ \left| \left( \Phi_{r^2} *  \nabla \left[ \chi_p \right]_{\Pa} \right) (0)  -  \left( \Phi_{r^2} * \nabla \left[ \chi_p \right]_{\Pa_\mathrm{stat}^k} \right) (0) \right| \right]  \leq C r^{-\frac d2};
\end{equation*}
\item[(ii)] In Step 2, we prove
\begin{equation*}
\E \left[ \sum_{x \in \cu_k } \nabla \left[ \chi_p \right]_{\Pa_\mathrm{stat}^k} (x) \right]=0;
\end{equation*}
\item[(iii)] In Step 3, we use the result obtained in Step 2 to show
\begin{equation*}
\left| \E \left[ \left( \Phi_{r^2} * \nabla \left[ \chi_p \right]_{\Pa_\mathrm{stat}^k} \right) (0) \right] \right| \leq C r^{-\frac d2}.
\end{equation*}
\end{enumerate}
 Lemma~\ref{lemma3.4} is then a consequence of the main results of Steps 1 and 3.

\medskip

\textit{Step 1.} The main result of this step is a consequence of the following computation, by~\eqref{defvectorfieldcluster} and Proposition~\ref{l.coarsegrads},
\begin{align} \label{bigestimatestep1bla1}
\lefteqn{\E \left[ \left|    \left( \Phi_{r^2} * \nabla \left[ \chi_p \right]_{\Pa} \right) (0)   -  \left( \Phi_{r^2} *  \nabla \left[ \chi_p \right]_{\Pa_\mathrm{stat}^k} \right) (0) \right| \right]  } \qquad & \\ &
\leq \E \left[  \sum_{ x \in B_{r^2}}   \left|  \nabla \left[ \chi_p \right]_{\Pa}(x) -\nabla \left[ \chi_p \right]_{\Pa_\mathrm{stat}^k}(x) \right| \Phi_{r^2}(x)   \indc_{\left\{ \exists x \in B_{r^2} \, : \, \cu_{\Pa_\mathrm{stat}^k}(x) \neq \cu_{\Pa} (x) \right\}} \right] \notag \\
& \qquad + \E \left[ \left|   \sum_{x \in \Zd \setminus B_{r^2}} \left(  \nabla \left[ \chi_p \right]_{\Pa}(x) - \nabla \left[ \chi_p \right]_{\Pa_\mathrm{stat}^k}(x) \right)  \Phi_{r^2}(x)   \right| \right]. \notag
\end{align}
The first term on the right-hand side can be estimated (crudely) the following way. We denote by $U_0$ the set
\begin{equation*}
U_0 := \bigcup_{x \in B_{r^2}} \cu_{\Pa} (x),
\end{equation*} 
we then enlarge this set by adding two additional layers of cubes and define
\begin{equation*}
U_1  := \bigcup_{\cu\in \Pa, \dist \left( \cu , U_0 \right) \leq 1} \cu \quad \mbox{ and } \quad U  := \bigcup_{\cu\in \Pa, \dist \left( \cu , U_1 \right) \leq 1} \cu.
\end{equation*}
Note that, by the properties of the partition $\Pa$ and~\eqref{e.Oavgs}, we have the inequality
\begin{equation} \label{est.volU}
|U| = C \left|U_1 \right| \leq  C \left|U_0 \right| \leq C \sum_{x \in B_{r^2}} \size \left( \cu_\Pa (x) \right)^d \leq  \O_s\left(C r^{2d}\right).
\end{equation}
Also with these definitions, we have, for each point $x \in B_{r^2}$,
\begin{equation*}
\left| \nabla \left[ \chi_p \right]_{\Pa_\mathrm{stat}^k}  (x) \right| \leq \sum_{y \in \C_\infty \cap U} \left| \nabla \chi_p \right| (y).
\end{equation*}
Similarly, for each point $x \in B_{r^2}$,
\begin{equation*}
\left| \nabla \left[ \chi_p \right]_{\Pa}  (x) \right| \leq \sum_{y \in \C_\infty \cap U} \left| \nabla \chi_p \right| (y) .
\end{equation*}
This leads to the estimate
\begin{align} \label{trickintPhiR=1}
\left|   \sum_{x \in B_{r^2}}  \left(  \nabla \left[ \chi_p \right]_{\Pa}(x) -\nabla \left[ \chi_p \right]_{\Pa_\mathrm{stat}^k}(x) \right) \Phi_{r^2}(x)  \right|  & \leq C  \left( \sum_{y \in \C_\infty \cap U} \left| \nabla \chi_p \right| (y) \right) \left( \sum_{x \in B_{r^2}} \Phi_{r^2}(x) \right) \\
& \leq C \sum_{y \in \C_\infty \cap U} \left| \nabla \chi_p \right| (y). \notag
\end{align}
Using Proposition~\ref{lipboundcorr}, the estimate on the volume of $U$ given in~\eqref{est.volU} and a computation similar to the one performed in~\eqref{refcomputation}, we obtain
\begin{equation*}
\sum_{y \in \C_\infty \cap U} \left| \nabla \chi_p \right| (y) \leq \O_s\left(C r^{2d}\right).
\end{equation*}
Then by~\eqref{errorstatpart}, we also have
\begin{align*}
\P \left[\exists x \in B_{r^2} \, : \, \cu_{\Pa_\mathrm{stat}^k}(x) \neq \cu_{\Pa} (x) \right] & \leq \sum_{z \in 3^k \Zd \cap B_{r^2}} \P \left[\exists x \in z + \cu_k \, : \, \cu_{\Pa_\mathrm{stat}^k}(x) \neq \cu_{\Pa} (x) \right] \\
& \leq \frac{r^{2d}}{3^{dk}} \P \left[\exists x \in \cu_k \, : \, \cu_{\Pa_\mathrm{stat}^k}(x) \neq \cu_{\Pa} (x) \right] \\
& \leq \frac{C r^{2d}}{ 3^{dk}} \exp \left(- C^{-1} 3^k\right).
\end{align*}
In particular, since $k$ has been chosen such that $3^{k} \leq r^\frac 12 < 3^{k+1}$, for each exponent $q > 0$, there exist a constant $C := C(d,\p,\lambda,q) < \infty$ and an exponent $s := s(d,\p,\lambda,q) > 0$ such that
\begin{equation*}
\indc_{\left\{ \exists x \in B_{r^2}\, : \, \cu_{\Pa_\mathrm{stat}^k}(x) \neq \cu_{\Pa} (x) \right\}} \leq \O_s(C r^{-q}).
\end{equation*}
Combining the three previous displays with $q$ chosen large enough, the Cauchy-Schwarz inequality and~\eqref{e.Oprods}, we obtain
\begin{equation*}
\left|   \sum_{ x \in B_{r^2}}  \left(  \nabla \left[ \chi_p \right]_{\Pa}(x) -\nabla \left[ \chi_p \right]_{\Pa_\mathrm{stat}^k}(x) \right) \Phi_{r^2}(x) \right| \indc_{\left\{ \exists x \in B_{r^2} \, : \, \cu_{\Pa_\mathrm{stat}^k}(x) \neq \cu_{\Pa} (x) \right\}} \\ \leq \O_s\left( C r^{-\frac d2} \right),
\end{equation*}
which yields in particular
\begin{equation*}
\E \left[ \left| \sum_{x \in B_{r^2}}  \left(  \nabla \left[ \chi_p \right]_{\Pa}(x) -\nabla \left[ \chi_p \right]_{\Pa_\mathrm{stat}^k}(x) \right) \Phi_{r^2}(x) \right| \indc_{\left\{ \exists x \in B_{r^2} \, : \, \cu_{\Pa_\mathrm{stat}^k}(x) \neq \cu_{\Pa} (x) \right\}} \right] \\ \leq C r^{-\frac d2}.
\end{equation*}
We now focus on estimating the second term on the right-hand side of~\eqref{bigestimatestep1bla1}. With the same computation as the one we just wrote, one obtains
\begin{equation*}
\sum_{x \in B_{r^2}}  \left|   \nabla \left[ \chi_p \right]_{\Pa}(x) -\nabla \left[ \chi_p \right]_{\Pa_\mathrm{stat}^k}(x)  \right|  \leq \O_s \left(C r^{4d} \right).
\end{equation*}
The proof is identical, we only need to replace the term $\sum_{x \in B_{r^2}} \Phi_{r^2}(x)$ by $Cr^{2d}$ in~\eqref{trickintPhiR=1}. Since this result is valid for any radius $r\geq 1$, we obtain, for each integer $n \in \N$,
\begin{align*}
 \sum_{x \in \C_\infty \cap \left( \cu_{n+1} \setminus \cu_{n}\right)}   \left|   \nabla \left[ \chi_p \right]_{\Pa}(x) -\nabla \left[ \chi_p \right]_{\Pa_\mathrm{stat}^k}(x)  \right|  &\leq  \sum_{x \in \C_\infty \cap B_{3^n}}   \left|   \nabla \left[ \chi_p \right]_{\Pa}(x) -\nabla \left[ \chi_p \right]_{\Pa_\mathrm{stat}^k}(x)  \right| \\
	& \leq  \O_s \left( C 3^{4dn} \right).
\end{align*}
We then use the estimate~\eqref{heatkernbound} on the discrete heat kernel and write
\begin{align*}
\lefteqn{ \E \left[   \sum_{x \in \Zd \setminus B_{r^2}} \left|  \nabla \left[ \chi_p \right]_{\Pa}(x) - \nabla \left[ \chi_p \right]_{\Pa_\mathrm{stat}^k}(x) \right|  \Phi_{r^2}(x)  \right]  } \qquad & \\ & 
\leq \sum_{n =\lfloor 2 \log_3(r) \rfloor}^{+\infty} \E \left[ \exp \left( - \frac{3^{n}}{r}\right)r^{-d}  \sum_{x x \in \C_\infty \cap \left( \cu_{n+1} \setminus \cu_{n}\right)}   \left|   \nabla \left[ \chi_p \right]_{\Pa}(x) -\nabla \left[ \chi_p \right]_{\Pa_\mathrm{stat}^k}(x)  \right|  \right]\\ &
\leq \sum_{n =2 \log_3(r)}^{+\infty} \exp \left( - \frac{3^{n}}{r}\right)r^{-d} 3^{4dn} \\ &
\leq C \exp \left(-C^{-1} r \right).
\end{align*}
Combining the estimates of the first and the second terms of the right-hand side completes the proof of Step 1.
\begin{remark}
Most of the estimates of this proof are crude; precise results are not needed. The same argument shows the following (stronger) result: for each exponent $q > 0$, there exists a constant $C := C(d ,\p, \lambda,q ) < \infty$ such that for each radius $r \geq 1$ and each integer $k \in \N$ satisfying $3^k \leq r^{\frac 12} < 3^{k+1}$,
\begin{equation*}
\E \left[ \left|  \Phi_{r^2} *  \nabla  \left[ \chi_p \right]_{\Pa}    -   \Phi_{r^2} * \nabla  \left[ \chi_p \right]_{\Pa_\mathrm{stat}^k} \right| \right]  \leq C r^{-q}.
\end{equation*}
The proof of Lemma~\ref{lemma3.4} only requires the result with the value $q = \frac d2$.
\end{remark}

\medskip

\textit{Step 2.} We prove the main result of this step by combining the stationarity property~\eqref{statcoarsencorrector} with the sublinear growth of the corrector.
First notice that by~\eqref{e.infinitenormcorr}, we have, for each radius $r> 1$,
\begin{equation*}
\osc_{\C_\infty \cap B_r } \chi_p  \leq \O_s\left(C r^{1-\delta}\right).
\end{equation*}
By the Stokes formula, we have, for each integer $n \in \N$,
\begin{equation*}
\left| \sum_{x \in \cu_{nk} } \nabla \left[ \chi_p \right]_{\Pa_\mathrm{stat}^k} (x) \right|  = \left| \sum_{x \in \partial \cu_{nk} } \left[ \chi_p \right]_{\Pa_\mathrm{stat}^k}(x)  \mathbf{n}(x)\right| \leq C 3^{kn(d-1)}\osc_{\C_\infty \cap \cu_{nk} } \chi_p \leq \O_s\left(C  3^{kn(d- \delta)}\right),
\end{equation*}
where the map $x \mapsto \mathbf{n}(x)$ is the discrete outer normal to the cube $\cu_{nk}$. This yields
\begin{equation*}
\left|\E \left[ \sum_{x \in \cu_{nk} } \nabla \left[ \chi_p \right]_{\Pa_\mathrm{stat}^k} (x)  \right] \right| \leq C3^{kn(d- \delta)}.
\end{equation*}
We also have, by the stationarity property~\eqref{statcoarsencorrector},
\begin{align*}
\E \left[\sum_{ x\in \cu_{nk} } \nabla \left[ \chi_p \right]_{\Pa_\mathrm{stat}^k} (x) \right] &
= \sum_{z \in  \left( 3^k\Zd \cap \cu_{kn}\right)} \E \left[  \sum_{x \in z + \cu_{k} } \nabla \left[ \chi_p \right]_{\Pa_\mathrm{stat}^k} (x) \right] \\
& = \frac{3^{dkn}}{3^{dk}}\E \left[ \sum_{ x\in  \cu_{k} } \nabla \left[ \chi_p \right]_{\Pa_\mathrm{stat}^k} (x) \right] .
\end{align*}
Combining the two previous results shows
\begin{equation*}
\left| \E \left[   \sum_{ x \in  \cu_{k}} \nabla \left[ \chi_p \right]_{\Pa_\mathrm{stat}^k} (x) \right] \right| \leq  C3^{dk} 3^{-kn \delta}.
\end{equation*}
Sending $n \rightarrow\infty$ shows
\begin{equation*}
\left| \E \left[  \sum_{ x\in \cu_{k} } \nabla \left[ \chi_p \right]_{\Pa_\mathrm{stat}^k} (x) \right] \right| =0.
\end{equation*}

\medskip

\textit{Step 3.} First notice that
\begin{equation*}
 \E \left[   \left( \Phi_{r^2} * \nabla \left[ \chi_p \right]_{\Pa_{\mathrm{stat}}^k} \right)(0)  \right] = \left( \Phi_{r^2} * \E \left[  \nabla \left[ \chi_p \right]_{\Pa_{\mathrm{stat}}^k}  \right] \right) (0).
\end{equation*}
By~\eqref{statcoarsencorrector}, the function 
\begin{equation*}
f := \left\{\begin{array}{ccc}
 \Z^d& \to &\Rd \\
  x & \mapsto & \E \left[  \nabla \left[ \chi_p \right]_{\Pa_{\mathrm{stat}}^k}(x) \right]\\
\end{array} \right.
\end{equation*}
is $3^{k}\Zd$-periodic. Consequently, there exist complex coefficients $(a_{\mathbf{n}})_{\mathbf{n} \in \cu_{k}}$ such that
\begin{equation*}
f(x) = \sum_{\mathbf{n} \in \cu_{k}} a_{\mathbf{n}} \exp \left( \frac{2i \pi \mathbf{n}  \cdot  x}{3^k}\right) .
\end{equation*}
Using that $\Phi$ is the solution of the discrete heat equation, which implies that the coefficients of its discrete Fourier transform can be explicitly computed, we obtain the identity
\begin{equation*}
 \left( \Phi_{r^2} * \E \left[  \nabla \left[ \chi_p \right]_{\Pa_{\mathrm{stat}}}  \right] \right) (0) =  \sum_{\mathbf{n} \in  \cu_k} a_{\mathbf{n}}  \exp \left( - r^2 \sum_{i=1}^d  2 \left( 1 - \cos \left( \frac{2 \pi \mathbf{n}_i}{k} \right) \right)\right).
\end{equation*}
Notice that the main result of Step 2 is equivalent to the following equality
\begin{equation*}
a_{0} = 0.
\end{equation*}
Using this identity, the Cauchy-Schwarz inequality and the lower bound $1 - \cos a \geq \frac{a^2}{C}$ for $a \in [-\pi , \pi]$ and a universal constant $C$, we obtain
\begin{equation} \label{Fourier1}
 \left| \left( \Phi_{r^2} * \E \left[  \nabla \left[ \chi_p \right]_{\Pa_{\mathrm{stat}}^k}  \right] \right) (0) \right|^2 \leq C \left( \sum_{\mathbf{n} \in  \cu_k \setminus \left\{ 0 \right\} } \left| a_{\mathbf{n}} \right|^2 \right) \left( \sum_{\mathbf{n} \in  \cu_k \setminus \{ 0\}  } \exp \left( - \frac{r^2 \left|\mathbf{n}\right|^2 }{C 3^{2k}}\right) \right).
\end{equation}
Since the integer $k$ was chosen such that $3^k \leq r^{\frac 12} < 3^{k+1}$, we have
\begin{equation} \label{Fourier2}
 \sum_{\mathbf{n} \in  \cu_k \setminus \left\{ 0 \right\}  } \exp \left( -  \frac{r^2\left| \mathbf{n}\right|^2 }{C 3^{2k}}\right) \leq C \exp \left( - C^{-1} r \right).
\end{equation}
Moreover, we have
\begin{equation*}
\sum_{\mathbf{n} \in \cu_k } \left| a_{\mathbf{n}} \right|^2  \leq  \E \left[  \sum_{ x\in \cu_k }    \left| \nabla \left[ \chi_p \right]_{\Pa_{\mathrm{stat}}^k}(x) \right|^2  \right] .
\end{equation*}
With the same computation as the one performed in Step 1, we obtain
\begin{equation*}
\sum_{x \in \cu_k }  \left| \nabla \left[ \chi_p \right]_{\Pa_{\mathrm{stat}}^k}(x) \right|^2 \leq \O_s \left(C |p|^2 3^{4kd} \right).
\end{equation*}
Taking the expectation yields
\begin{equation*}
\sum_{\mathbf{n} \in \cu_k } \left| a_{\mathbf{n}} \right|^2\leq C  3^{4kd}.
\end{equation*}
Combining this inequality with~\eqref{Fourier1} and~\eqref{Fourier2}, we obtain
\begin{equation*}
\left| \left( \Phi_{r^2} * \E \left[  \nabla \left[ \chi_p \right]_{\Pa_{\mathrm{stat}}^k}  \right] \right) (0) \right|^2 \leq C r^{2d} \exp \left( - C^{-1} r \right) \leq C \exp \left( - C^{-1} r \right),
\end{equation*} 
where we increased the value of the constant $C$ in the second inequality to absorb the algebraic growth of the term~$r^{2d}$. This implies in particular the main result of Step 3 and completes the proof of Lemma~\ref{lemma3.4}.
\end{proof}

\subsection{Estimating the resampling of the spatial averages} In this section, we prove Lemma~\ref{lemma3.5} which is recalled below.

\begin{replemma}{lemma3.5}
There exist a constant $C < \infty$ and an exponent $s > 0$, such that
\begin{equation*}
\mathbb{V}[X] \leq \O_s\left(C r^{-d}\right).
\end{equation*}
\end{replemma}

\textit{Proof of Lemma~\ref{lemma3.5}.} We recall Proposition~\ref{resampling} and the notation $X =  \left( \Phi_{r^2} * \nabla \left[ \chi_p \right]_{\Pa} \right) (0) $. Given an environment $\a \in \Omega$ and a bond $e = \left\{ x,y \right\} \in \mathcal{B}_d$, we want to estimate the term $\left( X - X_e' \right)^2$. To this end, one needs to understand how changing the conductance of the bond $e$ can affect the infinite cluster $\C_\infty$ and the partition $\Pa$. This is studied in the following lemma.

\begin{lemma} \label{resamplepartition}
There exist two constants $C_0 := C_0(d) < \infty$ and $C := C(d) < \infty$ such that for each bond $e = \{ x,y \} \in \mathcal{B}_d$, each pair of environments $\a , \tilde{\a} \in \Omega$ which are equal on the set $\mathcal{B}_d \setminus \{ e \}$ and each point $z \in B \left(x , C_0  \size \left(\cu_{\Pa} (x)\right)  \right)$, one has the estimate
\begin{equation*}
\size \left(\cu_{\Pa(\tilde{\a})} (z) \right) \leq C \size \left(\cu_{\Pa(\a)} (x) \right).
\end{equation*}
Moreover, for each point $z \in \Zd \setminus B \left(x , C_0  \size \left(\cu_{\Pa} (x)\right)  \right)$, one has the identity
\begin{equation*}
\size \left(\cu_{\Pa(\tilde{\a})} (z) \right) = \size \left(\cu_{\Pa(\a)} (z) \right).
\end{equation*}
\end{lemma}

\begin{proof}[Proof of Lemma~\ref{resamplepartition}.]
The main ingredients of the proof are listed below:
\begin{enumerate}
\item If a good cube $\cu \in \Pa_*$ is such that $3 \cu \cap \left\{ x,y \right\} = \emptyset$ then $\cu$ is a good cube under the environment~$\tilde{\a}$.
\item By the properties of the partition $\Pa$, every cube $\cu \in \Pa$ which does not contain the points $x$ and $y$ is crossable under the environment $\tilde{\a}$. The predecessors of $\cu_\Pa (x)$ and $\cu_\Pa(y)$ are also crossable under the environment $\tilde{\a}$.
\item By resampling the conductance of one bond, we cannot create an isolated cluster of size larger than $C \size(\cu_\Pa(x))$, for some constant $C_0:= C_0(d) < \infty$. In particular, there exists a constant $C:= C(d) < \infty$ such that every good cube of size larger than $C \size\left( \cu_\Pa(x) \right)$ under the environment $\a$ satisfies Property~(ii) of Definition~\ref{def.goodcube} under the environment $\tilde{\a}$.
\item There exists a constant $C :=C(d) < \infty$ such that every cube of size larger than $C \size\left( \cu_\Pa(x) \right)$ intersecting the cube $\cu_\Pa(x)$ is crossable by a cluster which does not intersect the cube $\cu_\Pa(x)$.
\item  If, for a point $y \in B \left(x , C_0  \size \left(\cu_{\Pa} (x)\right)  \right)$, the size of the cube $\cu_{\Pa}(y)$ is larger than $C \size \left(\cu_{\Pa} (x)\right)$, then the point $x$ belongs to $\cu_{\Pa}(y)$ or one of its neighbors and thus $ \size \left( \cu_{\Pa}(y) \right) \leq C \size \left( \cu_{\Pa}(x) \right)$.
\end{enumerate}
Combining these properties shows that every good cube $\cu$ under the environment $\a$ satisfying the estimate $\size(\cu) \geq C \size\left( \cu_\Pa(x) \right)$ is a good cube under the environment $\tilde{\a}$. It is then straightforward to see from the previous remarks and the construction of the partition $\Pa$ in the proof of Proposition~\ref{p.partitions} that the conclusion of the lemma is valid.
\end{proof}

\begin{sloppypar}To estimate the random variable $\left( X - X_e' \right)^2$, we introduce an extended probability space by doubling the variables $\left(\Omega', \F', \P'\right) = \left( \Omega \times \Omega , \F \otimes \F, \P \otimes \P \right)$. Given an environment $(\a(e'), \tilde{\a}(e'))_{e' \in \mathcal{B}_d} \in \Omega'$, we denote by $\mathrm{pr}_1$ (resp. $\mathrm{pr}_2$) the first (resp. second) projection, i.e., $\mathrm{pr}_1\left((\a(e'), \tilde{\a}(e'))_{e' \in \mathcal{B}_d}\right) = (\a(e'))_{e' \in \mathcal{B}_d}$ (resp. $\mathrm{pr}_2\left((\a(e'), \tilde{\a}(e'))_{e' \in \mathcal{B}_d}\right) =(\tilde{\a}(e'))_{e' \in \mathcal{B}_d}$). Any random variable $Z$ defined on the space $\left(\Omega, \F, \P\right)$ can be seen as a random variable defined on the extended space $\left(\Omega', \F', \P'\right)$ by the formula $Z = Z \circ \mathrm{pr}_1$, i.e.,
$$Z \left( (\a(e'), \tilde{\a}(e'))_{e' \in \mathcal{B}_d} \right)  =  Z \left( (\a(e'))_{e' \in \mathcal{B}_d} \right).$$
Given an enlarged environment $\left(\a(e'),\tilde{\a}(e') \right)_{e' \in \mathcal{B}_d}$, we denote by $\a$ the environment $\left(\a(e') \right)_{e' \in \mathcal{B}_d}$ and by $\a^{e}$ the environment $\left( (\a(e'))_{e'\in \mathcal{B}_d \setminus \{e \}}, \tilde{\a}(e) \right)$. Similarly, given a random variable $Z$ defined on the space $\Omega$ and a bond $e \in \mathcal{B}_d$, we denote by $Z^{e}$ the random variable defined on the space $(\Omega', \F',\P')$ by the formula
\begin{equation} \label{eq:TV09572302}
Z^{e}\left(\left(\a(e'),\tilde{\a}(e') \right)_{e' \in \mathcal{B}_d} \right) := Z \left( \a^{e} \right).
\end{equation}
We denote by $\Pa^{e}$ and $\C_\infty^{e}$ the partition of good cubes and the infinite cluster under the environment~$\a^{e}$.
It follows from the previous definitions that, for almost every environment $\a \in \Omega$,
\begin{equation*}
    X(\a) - X_e'(\a)= \int_\Omega  \left( \Phi_{r^2} * \left( \nabla \left[ \chi_p \right]_{\Pa} (\a , \tilde \a) - \nabla \left[ \chi_p^e \right]_{\Pa^e}(\a , \tilde \a) \right) \right) (0)  \, \mathrm{d}\P(\tilde{\a}).
\end{equation*}
All the random variables in the proof of this section are considered as random variables on the enlarged probability space $(\Omega',\F',\P')$ unless explicitly stated.\end{sloppypar}

We denote by $\E'$ the expectation with respect to the measure $\P'$. Given a constant $C>0$, an exponent $s > 0$ and a nonnegative random variable $Z : \Omega' \mapsto \R$, we write
\begin{equation*}
Z \leq \O_s'(C) \mbox{ if and only if } \E' \left[ \exp \left( \left( \frac{Z}{C} \right)^s \right) \right] \leq 2.
\end{equation*} 
Any random variable $Z$ defined on the space $( \Omega , \F , \P)$ satisfying $Z \leq \O_s(C)$ satisfies, as a random variable defined on the extended space $( \Omega' , \F' , \P')$, the inequality $Z \leq \O_s'(C)$. From the definition~\eqref{eq:TV09572302}, we see that, for each bond $e \in \mathcal{B}_d$,
\begin{equation}\label{eq:resampleZZe}
    Z \leq \O_s (C) \implies Z^e \leq \O_s'(C).
\end{equation}
The estimate~\eqref{eq:resampleZZe} is frequently used when the random variable $Z$ is equal to the size of a cube of the partition $\Pa$ (Proposition~\ref{p.partitions}), the minimal scale $\X$ above which the regularity theory applies (Proposition~\ref{p.regularitytheory}), or the minimal scale $\mathcal{M}_t(\Pa)$ associated to the partition $\Pa$ (Proposition~\ref{p.partitions}): we have, for each point $x \in \Zd$ and each bond $e \in \mathcal{B}_d$,
\begin{equation*}
\left\{ \begin{aligned}
    \size \left( \cu_{\Pa^e}(x) \right) = \left( \size \left( \cu_{\Pa}(x) \right) \right)^e  \leq \O_s' (C), \\
    \X^e \leq \O_s'(C), \\
    \mathcal{M}_t(\Pa^e) = \left(\mathcal{M}_t(\Pa)\right)^e \leq \O_s'(C).
\end{aligned} \right.
\end{equation*}
To prove Lemma~\ref{lemma3.5}, we prove the estimate
\begin{equation} \label{statementmainestimate}
\sum_{e \in \mathcal{B}_d} \left|  \left( \Phi_{r^2} * \left( \nabla \left[ \chi_p \right]_{\Pa} - \nabla \left[ \chi_p^e \right]_{\Pa^e} \right) \right) (0) \right|^2  \leq \O_s' \left( Cr^{-d} \right).
\end{equation}
The inequality~\eqref{statementmainestimate} is sufficient to prove Result 2; indeed with the same argument as in~\cite[Lemma 2.3]{armstrong2017quantitative}, we have
\begin{align*}
\lefteqn{ \E \left[ \exp \left( \left( \frac{\sum_{e \in \mathcal{B}_d}  \left(X - X_e'\right)^2}{Cr^{-d}} \right)^s \right) \right] } \qquad & \\ &
= \int_\Omega  \exp \left( \left(  \frac{\sum_{e \in \mathcal{B}_d}  \left| \int_\Omega  \left( \Phi_{r^2} * \left( \nabla \left[ \chi_p \right]_{\Pa} - \nabla \left[ \chi_p^e \right]_{\Pa^e} \right) \right) (0)  \, \mathrm{d} \P(\tilde{\a})\right|^2  }{Cr^{-d}} \right)^s \right) \, \mathrm{d} \P(\a) \\
& \leq \int_\Omega  \exp \left( \left( \int_\Omega \frac{\sum_{e \in \mathcal{B}_d}  \left|  \left( \Phi_{r^2} * \left( \nabla \left[ \chi_p \right]_{\Pa} - \nabla \left[ \chi_p^e \right]_{\Pa^e} \right) \right) (0) \right|^2}{Cr^{-d}} \, \mathrm{d} \P(\tilde{\a}) \right)^s \right) \, \mathrm{d} \P(\a) \\
& \leq C  \int_\Omega \int_\Omega  \exp \left( \left( \frac{\sum_{e \in \mathcal{B}_d}  \left|  \left( \Phi_{r^2} * \left(  \nabla \left[ \chi_p \right]_{\Pa} -  \nabla \left[ \chi_p^e \right]_{\Pa^e} \right) \right) (0) \right|^2}{Cr^{-d}} \right)^s \right) \, \mathrm{d} \P(\a) \mathrm{d} \P(\tilde{\a}) \\
& \leq C \E'\left[ \exp \left( \left( \frac{\sum_{e \in \mathcal{B}_d}  \left|   \left( \Phi_{r^2} * \left( \nabla \left[ \chi_p \right]_{\Pa} - \nabla \left[ \chi_p^e \right]_{\Pa^e} \right) \right) (0) \right|^2}{Cr^{-d}} \right)^s \right) \right] \\
& \leq 2C.
\end{align*}
This yields, after redefinition of the constant $C$,
\begin{equation*}
\sum_{e \in \mathcal{B}_d}  \left(X - X_e'\right)^2 \leq \O_s' \left(Cr^{-d} \right).
\end{equation*}
Before starting the proof of~\eqref{statementmainestimate}, we select one of the correctors $ \chi_p^e$ arbitrarily (we recall that they are defined up to a constant). As we are interested in the gradient of the corrector, the value of the constant is not important. We want to give a meaning to the function~$\left[ \chi_p^e \right]_\Pa$
as a random variable defined on the extended probability space $\Omega'$. 

\smallskip

Since we do not necessarily have the identity $\C_\infty = \C_\infty^e$, we cannot simply write $\left[ \chi_p^e \right]_\Pa (z) =  \chi_p^e \left( z \left( \cu_\Pa (z) \right) \right)$. Nevertheless, since the two environments $\left( (\a(e'))_{e' \in \mathcal{B}_d \setminus \{ e\}}, \tilde{\a}(e) \right)$ and $ (\a(e'))_{e' \in \mathcal{B}_d}$ only differ by one bond, we have either $\C_\infty \subseteq \C_\infty^e$ or $\C_\infty^e \subset \C_\infty$. In the former case, we can define $\left[ \chi_p^e \right]_\Pa (z) =  \chi_p^e \left( z \left( \cu_\Pa (z) \right) \right)$. In the latter case, the cluster $\C_\infty \setminus \C_\infty^e$ is connected to $\C_\infty$ by the bond $e$. Without loss of generality, we denote it by $e = \{x,y\}$ and assume that $x \in \C_\infty^e$. One can then check that the function
\begin{equation} \label{defofwext}
w  := \left\{ \begin{array}{ccc}
 \C_\infty & \to &\R \\
 z & \mapsto & \chi_p^e(z) \indc_{\{ z \in \C_\infty^e\}}+ \left( p \cdot (z - x) + \chi_p^e (z) \right) \indc_{\{ z \notin \C_\infty^e \}}\\
\end{array} \right.
\end{equation}
is a solution of the equation
\begin{equation*}
- \nabla \cdot  \a \nabla \left( p\cdot x + w \right) = 0 ~\mbox{in}~ \C_\infty
\end{equation*}
and more precisely that the map $x \mapsto p \cdot x + w(x)$ belongs to the space $\A_1(\C_\infty).$ In particular, this gives the identity
$
w = \chi_p.
$
We thus define
\begin{equation*}
\left[ \chi_p^e \right]_\Pa = \left[ w \right]_\Pa.
\end{equation*}
To prove the estimate~\eqref{statementmainestimate}, we use the random variable $\left[ \chi_p^e \right]_\Pa$ and split the sum into two terms
\begin{multline} \label{splitestimate}
\left|   \left( \Phi_{r^2} * \left( \nabla \left[ \chi_p \right]_{\Pa} - \nabla \left[ \chi_p^e \right]_{\Pa^e} \right) \right) (0) \right|^2  \\ \leq  \underbrace{2 \left|  \left( \Phi_{r^2} * \left( \nabla \left[ \chi_p^e \right]_{\Pa} - \nabla \left[ \chi_p^e \right]_{\Pa^e} \right)(0) \right)  \right|^2}_{\eqref{splitestimate}-(i)} + \underbrace{2 \left| \left( \Phi_{r^2} * \left(  \nabla \left[ \chi_p \right]_{\Pa} -  \nabla \left[ \chi_p^e \right]_{\Pa} \right) \right) (0)\right|^2}_{\eqref{splitestimate}-(ii)}  .
\end{multline}
We estimate the two terms in the right side in the two steps below.

\smallskip

\textit{Step 1. Estimate of the term~\eqref{splitestimate}-(i).} We use  Lemma~\ref{resamplepartition} and Proposition~\ref{l.coarseLs} with the exponent $s=1$ to write
\begin{align*}
\lefteqn{ \left|  \left( \Phi_{r^2} * \left( \nabla \left[ \chi_p^e \right]_{\Pa} - \nabla \left[ \chi_p^e \right]_{\Pa^e} \right) \right) (0) \right|^2 } \qquad & \\ &
\leq  \left( \sum_{z \in  \Zd \cap B \left(x , C  \size \left(\cu_{\Pa}(x) \right) \right)} \left| \nabla \left[ \chi_p^e \right]_{\Pa}(z)\right| + \left| \nabla \left[ \chi_p^e \right]_{\Pa^e}(z) \right| \right)^2  \sup_{z \in B \left(x , C  \size \left(\cu_{\Pa}(x) \right) \right)}  \Phi_{r^2}^2(z) \\ &
\leq C  \left( \sum_{z \in \C_\infty^e \cap B \left(x , C  \size \left(\cu_{\Pa}(x) \right) \right) }  \size \left( \cu_{\Pa}(x) \right)^{d-1} \left( \left| \nabla \chi_p^e\right| (z) + 1 \right) \right)^2  \sup_{z \in B \left(x , C  \size \left(\cu_{\Pa}(x) \right) \right)}  \Phi_{r^2}^2(z) .
\end{align*}
The term "+1" on the right-hand side comes from the assumption $|p| = 1$ combined with the definition of the map $w$ stated in~\eqref{defofwext} (in the case $\C_\infty^e \subset \C_\infty$). We deduce that
\begin{multline} \label{eq:TV11331602}
\left| \left( \Phi_{r^2} * \left( \nabla \left[ \chi_p^e \right]_{\Pa} - \nabla \left[ \chi_p^e \right]_{\Pa^e} \right) \right) (0)  \right|^2 \\
\leq C  \size \left( \cu_{\Pa}(x) \right)^{3d-2} \sup_{z \in \C_\infty^e \cap B \left(x , C  \size \left(\cu_{\Pa}(x) \right) \right) } \left( \left| \nabla \chi_p^e \right|^2 (z) + 1 \right)  \sup_{z \in B \left(x , C  \size \left(\cu_{\Pa}(x) \right) \right)}  \Phi_{r^2}^2(z).
\end{multline}
By the heat kernel bound~\eqref{heatkernbound} stated in Definition~\ref{def.heatkerneldisc}, there exists a constant $C(d) < \infty$ such that, for each point $z \in \Zd$,
\begin{equation} \label{trickexponentialpoly}
\Phi_{r^2} \leq  \frac{C}{r^d} \left( \left(\frac{\left| z\right|}{r}\right)^{-\frac{d+1}{2}} \wedge 1 \right).
\end{equation}
We denote by $\zeta(z) := \left( |z|^{-\frac{d+1}{2}} \wedge 1 \right)$ and by $\zeta_r(z) := \frac 1{r^d} \zeta\left( \frac{z}{r} \right)$. We use the function $\zeta$ instead of the heat kernel $\Phi_{r^2}$ to complete the estimate of the term~\eqref{splitestimate}-(i) because it satisfies the inequality
\begin{equation*}
 \sup_{z \in B \left(x , C  \size \left(\cu_{\Pa}(x) \right) \right)}  \zeta_r^2(z) \leq C \size \left(\cu_{\Pa}(x)\right)^{d+1}  \inf_{z \in B \left(x , C  \size \left(\cu_{\Pa}(x) \right) \right)} \zeta_r^2(z).
\end{equation*}
In particular, the estimate~\eqref{eq:TV11331602} can be rewritten
\begin{multline*}
\left|\left( \Phi_{r^2} * \left( \nabla \left[ \chi_p^e \right]_{\Pa} - \nabla \left[ \chi_p^e \right]_{\Pa^e} \right) \right) (0)   \right|^2 \\ 
\leq C  \size \left( \cu_{\Pa}(x) \right)^{4d-1}  \sum_{z \in \C_\infty^e \cap B \left(x , C  \size \left(\cu_{\Pa}(x) \right) \right) } \zeta_r(z)^2 \left( \left| \nabla \chi_p^e \right|^2 (z) + 1 \right).
\end{multline*}
Summing over all the bonds $e \in \mathcal{B}_d$ gives
\begin{align} \label{previouscomputation}
\lefteqn{ \sum_{e \in \mathcal{B}_d} \left|\left( \Phi_{r^2} * \left( \nabla \left[ \chi_p^e \right]_{\Pa} - \nabla \left[ \chi_p^e \right]_{\Pa^e} \right) \right) (0)   \right|^2} \qquad & \\ & 
\leq C \sum_{x \in \Zd}  \size \left( \cu_{\Pa}(x) \right)^{4d-1} \sum_{z \in \C_\infty^e \cap B \left(x , C  \size \left(\cu_{\Pa}(x) \right) \right) } \zeta_r(z)^2 \left( \left| \nabla \chi_p^e \right|^2 (z) +1 \right) \notag \\
& \leq C \sum_{z \in \C_\infty^e} \zeta_r(z)^2 \left( \left| \nabla \chi_p^e\right|^2 (z) + 1 \right) \left( \sum_{x \in \Zd} \size \left( \cu_\Pa (x) \right)^{4d-1} \indc_{\left\lbrace z \in B \left(x, C \size \left( \cu_\Pa (x) \right) \right) \right\rbrace} \right). \notag
\end{align}
Using the estimate $\size \left( \cu_\Pa (x) \right) \leq \O_s'(C)$, valid for any point $x \in \Zd$, we obtain
\begin{equation*}
\indc_{\left\lbrace z \in B \left(x, C \size \left( \cu_\Pa (x) \right) \right) \right\rbrace} \leq C \frac{ \size \left( \cu_\Pa (x) \right)^{d+1}}{\left(|x-z| \vee 1 \right)^{d+1}} \leq \frac{\O_s' (C)}{\left(|x-z| \vee 1 \right)^{d+1}}.
\end{equation*}
Since the map $z \mapsto \left( |z|\vee 1\right)^{-d-1}$ is summable on $\Zd$, we use the inequality~\eqref{e.Oavgs} to obtain
\begin{equation} \label{errortermbounded}
\sum_{x \in \Zd} \size \left( \cu_\Pa (x) \right)^{4d-1} \indc_{\left\lbrace z \in B \left(x, C \size \left( \cu_\Pa (x) \right) \right) \right\rbrace} \leq \O_s'(C).
\end{equation} 
By Proposition~\ref{lipboundcorr} and the implication~\eqref{eq:resampleZZe}, we have the Lipschitz bound on the corrector
\begin{equation} \label{nablachibounded}
\left| \nabla \chi_p^e (y) \right|  \indc_{\{ y \in \C_\infty^e \}} \leq \O_s'(C),
\end{equation}
which implies
\begin{equation*}
\sum_{e \in \mathcal{B}_d} \left|\left( \Phi_{r^2} * \left( \nabla \left[ \chi_p^e \right]_{\Pa} - \nabla \left[ \chi_p^e \right]_{\Pa^e} \right) \right) (0) \right|^2  \leq C \sum_{y \in \Zd} \zeta_r(y)^2 \O_s'(C).
\end{equation*}
We use the estimate~\eqref{e.Oavgs} and the inequality $ \sum_{y \in \Zd} \zeta_r(y)^2 \leq C r^{-d}$ to obtain
\begin{equation*}
\sum_{e \in \mathcal{B}_d} \left| \left( \Phi_{r^2} * \left( \nabla \left[ \chi_p^e \right]_{\Pa} - \nabla \left[ \chi_p^e \right]_{\Pa^e} \right) \right) (0) \right|^2  \leq \O_s' \left( \frac{C}{r^d} \right).
\end{equation*}
This completes the proof of the estimate of the first term on the right-hand side of~\eqref{splitestimate}.

\smallskip

\textit{Step 2. Estimate of the term~\eqref{splitestimate}-(ii).} In this step, we prove the inequality
\begin{equation*}
\left| \left( \Phi_{r^2} * \left( \nabla \left[ \chi_p \right]_{\Pa} - \nabla \left[ \chi_p^e \right]_{\Pa} \right) \right) (0) \right|^2  \leq \O_s'\left( \frac{C}{r^d}\right).
\end{equation*}
To prove this estimate, we distinguish three cases. We recall the two endpoints of the bond $e$ are $x$ and $y$; they are fixed through the proof.

\textbf{Case 1.} ($x \notin \C_\infty$ and $y \notin \C_\infty$) or $\a = \a^e$. In that case, one has the identities $\C_\infty = \C_\infty^e$ and $\nabla \chi_p = \nabla \chi_p^e$. They imply
\begin{equation*}
\left|\left( \Phi_{r^2} * \left( \nabla \left[ \chi_p \right]_{\Pa} - \nabla \left[ \chi_p^e \right]_{\Pa} \right) \right) (0)  \right|^2 =0.
\end{equation*}

\textbf{Case 2.} $\C_\infty \neq \C_\infty^e$. In that case, we have $\left[ \chi_p \right]_{\Pa}  = \left[\chi_p^e \right]_{\Pa}$ by definition of the latter quantity. It implies
\begin{equation*}
\left|\left( \Phi_{r^2} * \left( \nabla \left[ \chi_p  \right]_{\Pa} - \nabla \left[ \chi_p^e \right]_{\Pa} \right) \right) (0) \right|^2 =0.
\end{equation*}

\textbf{Case 3.} $x,y \in \C_\infty$ and $\C_\infty = \C_\infty^e$ and $\a \neq \a^e$. By definitions of the correctors $\chi_p$ and $\chi_p^e$, we have the identity
\begin{equation} \label{eq.corrcorre}
- \nabla \cdot \left( \a \nabla \left( \chi_p - \chi_p^e \right) \right) = \left( \a - \a^e \right) \left( \left\{x,y \right\} \right) \left( p \cdot (x-y) + \chi_p^e(x) - \chi_p^e(y) \right)  \left( \delta_x -   \delta_y \right) \hspace{4mm} \mbox{in}~\C_\infty.
\end{equation}
We solve~\eqref{eq.corrcorre} by using Proposition~\ref{p.green} and recall the notation $\nabla G(e , \cdot)$ introduced there. Note that the function $$\chi_p - \chi_p^e -  \left( \a - \a^e \right) (x,y) \left( p \cdot (x-y) + \chi_p^e(x) - \chi_p^e(y) \right) \nabla G ( e, \cdot) ~ \mbox{is}~ \a-\mbox{harmonic}.$$ By the sublinear growth of the corrector stated in~\eqref{minscale.sublin}, the $L^2(\C_\infty)$-bound stated in~\eqref{gradientgreenbounded} on the gradient of the function $\nabla G(e,\cdot)$ and a version of the Poincar\'e inequality on the percolation cluster (see for instance the proof of Proposition~\ref{p.regularitytheory}), one can show that the function $$\chi_p - \chi_p^e -  \left( \a - \a^e \right) (x,y) \left( p \cdot (x-y) + \chi_p^e(x) - \chi_p^e(y) \right)  \nabla G ( e, \cdot) ~\mbox{has a sublinear growth.}$$ This implies that this function is constant. In particular, it proves the identity
\begin{equation} \label{eq:TV08361802}
\nabla \chi_p - \nabla \chi_p^e =  \left( \a - \a^e \right) (x,y) \left( p \cdot (x-y) + \chi_p^e(x) - \chi_p^e(y) \right) \nabla  \nabla G(e,\cdot).
\end{equation}
We now estimate the right side of~\eqref{eq:TV08361802} thanks to the Lipschitz bound on the corrector stated in Proposition~\ref{lipboundcorr} and~\eqref{nablachibounded}. We distinguish two cases depending on the value of the conductance~$\a^e(e)$:
\begin{itemize}
\item If $\a^e(e) = \tilde{\a}(e) \neq 0,$ then~\eqref{nablachibounded} implies the estimate
\begin{equation} \label{eq:TV08432802}
\left| \chi_p^e(x) - \chi_p^e(y) \right| \indc_{\{ y \in \C_\infty^e, \tilde{\a}(e) \neq 0  \}}\leq \left| \nabla \chi_p^e (y) \right|  \indc_{\{ y \in \C_\infty^e \}} \leq C \left( \X^e(x) \right)^{d/2};
\end{equation}
\item If $\a^e(e) = \tilde{\a}(e) = 0$, then there exists a path going from $x$ to $y$ which stays in the cube $\cu_{\Pa^e} (x)$ and its neighbors (its neighbors because we may not have $\cu_{\Pa^e} (x) = \cu_{\Pa^e} (y)$ or we may have $x, y \in \C_\infty \setminus \C_* \left(\cu_{\Pa^e} (x)\right)$). Combining this remark with Lemma~\ref{resamplepartition}, we obtain
\begin{align*}
\left| \chi_p^e(x) - \chi_p^e(y) \right| & \leq C \sum_{z \in \C_\infty^e \cap B \left(x, C \size \left(\cu_\Pa(x) \right)\right)} \left| \nabla \chi_p^e \right|(z) \\
& \leq C  \size \left(\cu_\Pa (x) \right)^{d} \left\| \nabla \chi_p^e\right\|_{\underline{L}^2 \left( \C_\infty^e \cap B\left(x, C \size \left(\cu_\Pa(x) \right) \right) \right)}.
\end{align*}
Using the Lipschitz bounds on the corrector, we deduce that
\begin{equation} \label{eq:TV08421802}
\left| \chi_p^e(x) - \chi_p^e(y) \right|  \indc_{\{ x, y \in \C_\infty^e, \tilde{\a}(e) = 0  \}} \leq \size \left(\cu_\Pa (x) \right)^{d} \left( \X^e(x) \right)^{d/2}.
\end{equation}
\end{itemize}
Combining the estimates~\eqref{eq:TV08432802} and~\eqref{eq:TV08421802}, we obtain the inequality
\begin{equation} \label{estimatecorrectorgreensfunction1}
\left|\nabla \left( \Phi_{r^2} * \left( \left[ \chi_p \right]_{\Pa} - \left[ \chi_p^e \right]_{\Pa} \right) \right) (0) \right|^2  \\ \leq  \left| \nabla \left( \Phi_{r^2} * \left( \left[ \nabla G(e,\cdot) \right]_{\Pa} \right) \right) (0) \right|^2 \size \left(\cu_\Pa (x) \right)^{2d} \left( \X^e(x) \right)^{d}.
\end{equation}
In the next step of the proof, we treat the coarsening in the right-hand side of~\eqref{estimatecorrectorgreensfunction1}. To this end, we prove that there exist a constant $C:= C(d) < \infty$ and a (random) vector field $\gamma_r : E_d \rightarrow \R^d$ satisfying the estimate, for each edge $e'=(x',y') \in E_d$,
\begin{equation*}
\left| \gamma_r (e') \right| \leq C \size(\cu_\Pa(x'))^{2d}  \zeta_r \left(x' \right)
\end{equation*}
such that for each function $u : \C_\infty \mapsto \R$ satisfying $\left\langle \nabla u , \nabla u \right\rangle_{\C_\infty} < \infty$,
\begin{equation} \label{propertygammaR}
 \left( \Phi_{r^2} * \nabla \left[u \right]_{\Pa}  \right) (0) =  \left\langle \gamma_r, \nabla u   \right\rangle_{\C_\infty}.
\end{equation}
We first write
\begin{equation*}
\left( \Phi_{r^2} * \nabla \left[u \right]_{\Pa} \right) (0)   = \sum_{z \in \Zd}  \Phi_{r^2} \left( z  \right) \nabla \left[u \right]_{\Pa} (z).
\end{equation*}
Given two neighboring points $z , z' \in \Zd$, note that the values $\left[u \right]_{\Pa} (z)$ and $\left[u \right]_{\Pa} (z')$ are only different if the points $z$ and $z'$ belong to two different cubes of the partition $\Pa$. In that case, we have
\begin{align*}
 \left[u \right]_{\Pa} (z ) - \left[u \right]_{\Pa} (z' )
& = u(\zbar(\cu_\Pa(z))) - u(\zbar(\cu_\Pa(z'))).
\end{align*}
Recall that there exists a path between $\zbar(\cu_\Pa(z))$ and $\zbar(\cu_\Pa(z'))$ which lies entirely in the set $\cu_\Pa(z) \cup \cu_\Pa(z')$. We denote this path by $p_{z,z'} \subseteq E_d$. Summing over the edges along this path, we find that 
\begin{equation*}
u(\zbar(\cu_\Pa(z))) -u(\zbar(\cu_\Pa(z'))) = \sum_{e' \in p_{z,z'}} \nabla u (e') = \sum_{e' \in E_d} \nabla u (e') \indc_{\{ e' \in p_{z,z'}\}}.
\end{equation*}
If the points $z$ and $z'$ belong to the same cube of the partition $\Pa$, we keep the same notation with the convention $p_{z,z'} = \emptyset$. Consequently, we have for each pair of neighboring points $z,z' \in \Zd$,
\begin{equation*}
\left[u \right]_{\Pa} (z) - \left[u \right]_{\Pa} (z' ) = \sum_{e' \in E_d} \nabla u (e') \indc_{\{ e' \in p_{z,z'}\}}.
\end{equation*}
Using this formula, we can rewrite
\begin{equation*}
\left( \Phi_{r^2} * \nabla \left[u \right]_{\Pa}  \right) (0) =  \left\langle \gamma_r,   \nabla u  \right\rangle_{\C_\infty},
\end{equation*}
where $\gamma_r$ is the vector-valued random field defined by the formula, for each edge $e'  \in E_d$,
\begin{equation*}
\gamma_r (e') =\begin{pmatrix}
\sum_{z \in \Zd}  \Phi_{r^2} \left( z \right) \indc_{\{ e' \in p_{z,z + \e_1}\}} \\[3mm]
\vdots \\[3mm]
\sum_{z \in \Zd}  \Phi_{r^2} \left( z \right) \indc_{\{ e' \in p_{z,z + \e_d}\}} \\
\end{pmatrix} .
\end{equation*}
For each pair of neighboring points $z,z' \in \Zd$ such that $\cu_\Pa (z) \neq \cu_\Pa (z')$, the path between the points $\zbar(\cu_\Pa(z))$ and $\zbar(\cu_\Pa(z'))$ lies in the set $\cu_\Pa(z) \cup \cu_\Pa(z')$. In particular an edge $e' = \left(x',y' \right)$ belongs to the path $p_{z,z'}$ only if $\dist \left( z , \partial \cu_\Pa(x') \right) \leq 1$. This argument implies the inequality
\begin{equation} \label{eq:TV10142302}
\left| \gamma_r (e') \right| \leq C \sum_{z \, : \, \dist \left( z , \partial \cu_\Pa(x') \right) \leq 1} \Phi_{r^2} \left( z  \right).
\end{equation}
We use the inequality~\eqref{trickexponentialpoly} on the discrete heat kernel and note that function $\zeta$ satisfies the estimate, for each triadic cube $\cu \in \T$,
\begin{equation*}
\sup_{\cu } \zeta_r  \leq C \size \left( \cu \right)^{\frac{d+1}{2}} \inf_{\cu} \zeta_r.
\end{equation*}
As a consequence of the two previous displays, we can rewrite the estimate~\eqref{eq:TV10142302}
\begin{align} \label{estimateeta}
\left| \gamma_r (e') \right| & \leq  \sum_{z \, : \, \dist \left( z , \partial \cu_\Pa(x') \right) \leq 1} \Phi_{r^2} \left(z \right) \\
						& \leq C \size(\cu_\Pa(x'))^{d-1}  \sup_{z \, : \, \dist \left( z , \partial \cu_\Pa(x') \right) \leq 1} \zeta_r  \notag \\
						& \leq C \size(\cu_\Pa(x'))^{2d}  \zeta_r (x' ), \notag
\end{align}
which is the desired inequality. The proof of ~\eqref{propertygammaR} is complete.

\smallskip 

Applying the property~\eqref{propertygammaR} to the function $u =  \nabla  G(e,\cdot) $, the inequality~\eqref{estimatecorrectorgreensfunction1} becomes
\begin{equation*}
\left|\left( \Phi_{r^2} * \left( \nabla \left[ \chi_p \right]_{\Pa}  - \nabla \left[ \chi_p^e \right]_{\Pa} \right) \right) (0) \right|^2  \\ \leq  \left| \left\langle \gamma_r, \nabla \nabla G(e,\cdot) \right\rangle_{\C_\infty} \right|^2 \size \left(\cu_\Pa (x) \right)^{2d} \left( \X^e(x) \right)^{d}.
\end{equation*}
We apply Proposition~\ref{morecomplicatedgreen} and denote by $w_{\gamma_r} : \C_\infty \to \R^d$ the solution of the equation
\begin{equation*}
- \nabla \cdot  \a \nabla w_{\gamma_r}  = - \nabla \cdot \gamma_r \mbox{ in } \C_\infty,
\end{equation*}
so that, for each edge $e'$ in the infinite cluster,
\begin{equation*}
\nabla  w_{\gamma_r}(e')  = \sum_{e'' \subseteq \C_\infty}  \gamma_r(e'') \nabla \nabla G\left(e'',e'\right) =    \sum_{e'' \subseteq \C_\infty}  \gamma_r(e'') \nabla \nabla G(e',e'')  = \left\langle \gamma_r, \nabla \nabla  G(e', \cdot)  \right\rangle_{\C_\infty}.
\end{equation*}
This implies the identity
\begin{equation*}
\left| w_{\gamma_r} (x) - w_{\gamma_r} (y) \right|=  \left| \left\langle \gamma_r, \nabla \nabla  G( e , \cdot )  \right\rangle_{\C_\infty} \right|,
\end{equation*}
and consequently
\begin{equation*}
\left| \left( \Phi_{r^2} * \left( \nabla \left[ \chi_p \right]_\Pa - \nabla \left[ \chi_p^e \right]_{\Pa} \right) \right) (0)  \right|^2 \leq  \left| w_{\gamma_r}(x) -  w_{\gamma_r}(y) \right|^2  \size \left(\cu_\Pa (x) \right)^{2d} \left( \X^e(x) \right)^{d}.
\end{equation*}

We now combine Cases 1, 2 and 3 to obtain the estimate
\begin{equation*}
\sum_{e \in \mathcal{B}_d} \left| \left( \Phi_{r^2} * \left( \nabla \left[ \chi_p \right]_{\Pa} - \left[ \chi_p^e \right]_{\Pa} \right) \right) (0) \right|^2 \leq C  \sum_{x,y \in \C_\infty, |x-y|_1 = 1}\left| w_{\gamma_r}(x) -  w_{\gamma_r}(y) \right|^2  \size \left(\cu_\Pa (x) \right)^{2d}   \sum_{e \in \mathcal{B}_d^x} \left( \X^e(x) \right)^d,
\end{equation*}
where we used the notation $\mathcal{B}_d^x := \left\lbrace \{ x , y \} \, : \, y \in \Zd, y \sim x \right\rbrace$ to denote the set of bonds connecting the point $x$ to another vertex of $\Zd$.

Using that for each pair of points $x,y \in \C_\infty$ with $|x-y|_1 = 1$, there exists a path which is contained in the infinite cluster $\C_\infty$, the cube $\cu_\Pa (x)$ and its neighbors (the path is simply $(x,y)$ if $\a(\{x,y\}) \neq 0$), we obtain
\begin{multline} \label{resultatbrut}
\sum_{e \in \mathcal{B}_d} \left| \left( \Phi_{r^2} * \left( \nabla \left[ \chi_p \right]_{\Pa} - \nabla \left[ \chi_p^e \right]_{\Pa} \right) \right) (0) \right|^2 \\ \leq C\sum_{z \in \C_\infty}\left|\nabla  w_{\gamma_r}\right|^2(z)  \size \left(\cu_\Pa (z) \right)^{3d} \left( \sum_{x \in \Zd , \, \dist \left(\cu_\Pa(x) , \cu_\Pa (z)\right) \leq 1, \newline  e \in \mathcal{B}_d^x}\left(  \X^e(x) \right)^d \right).
\end{multline} 
To estimate the term on the right-hand side, we first note that, by definition of the function $w_{\gamma_r}$ and the inequality~\eqref{e.Oavgs},
\begin{equation*}
\sum_{z \in \C_\infty}\left|\nabla  w_{\gamma_r} \right|^2(z)  \leq C  \sum_{z \in \C_\infty} \left| \gamma_r(z) \right|^2 
\leq C \sum_{z \in \Zd} C \size(\cu_\Pa(z))^{4d} \zeta_r \left( z\right)^2
 \leq \O_s' \left( C r^{-d}\right).
\end{equation*}
To complete the proof, we use the two lemmas stated below; their proofs are postponed to Appendices~\ref{appC} and~\ref{appB}.

\begin{proposition}[Meyers estimate] \label{meyercorrector}
There exist a constant $C := C(d,\lambda,\p) < \infty$, two exponents $s := s(d, \lambda, \p) > 0$ and $\ep := \ep(d, \lambda, \p) >0$ and a random variable $\M_{\mathrm{Meyers}} \leq \O_s(C)$ such that for each integer $m\in \N $ satisfying $3^m \geq \M_{\mathrm{Meyers}}$, each vector field $\xi : E_d \to \R$ satisfying
\begin{equation*}
\xi (x,y) = 0 \mbox{ if} ~ \a(x,y) = 0 \mbox{ or } x,y \notin \C_\infty,
\end{equation*}
and each function $v : \C_\infty \mapsto \R$ solution of the equation
\begin{equation*}
- \nabla \cdot \left( \a \nabla v \right) = - \nabla \cdot \xi \mbox{ in } \C_\infty,
\end{equation*}
on has the estimate
\begin{multline} \label{eqmeyercorrector}
\left( \frac{1}{|\cu_m|} \sum_{x \in \cu_m \cap \C_\infty} \left| \nabla v \right|^{2+\ep}(x)  \right)^{\frac1{2+\ep}} \\
\leq C \left( \frac1{\left| \frac 43 \cu_m \right|}\sum_{x \in \frac 43 \cu_m\cap \C_\infty} \left|\nabla v \right|^{2}(x) \right)^{\frac1{2}} + C  \left( \frac1{\left| \frac 43 \cu_m \right|} \sum_{x \in \frac 43\cu_m \cap \C_\infty} \left| \xi \right|^{2+\ep}(x) \right)^{\frac1{2+\ep}}.
\end{multline}
\end{proposition}

\begin{lemma}[Minimal scale] \label{minscalecorrector}
\begin{sloppypar}There exist a constant $C := C(d, \p ,\lambda) < \infty$, an exponent $s := s(d, \p ,\lambda) > 0$ and a random variable $\M_1 \leq \O_s' (C)$ such that for each integer $m \in \N$ satisfying $3^m \geq \M_1$,\end{sloppypar}
\begin{equation} \label{eqminscalecorrector}
3^{-dm} \sum_{z \in \cu_m}\size \left(\cu_\Pa (z) \right)^{\frac{3d(2+\ep)}{\ep}}  \left( \sum_{x \in \Zd, \, \dist \left(\cu_\Pa(x) , \cu_\Pa (z)\right) \leq 1, \, e \in \mathcal{B}_d^x}\left(\X^e(x) \right)^d \right)^{\frac{2+\ep}{\ep}} \leq C
\end{equation}
where $\ep := \ep(d, \p , \lambda) > 0$ is the exponent which appears in Proposition~\ref{meyercorrector}.
\end{lemma}

\begin{definition}[The partition $\mathcal{U}$] \label{def.partitionU}
We define the following family of ``good cubes"
\begin{equation*}
\mathcal{G}:= \left\{\cu\in\T\,:\, \mbox{\eqref{eqmeyercorrector} and \eqref{eqminscalecorrector} hold}\right\}
\end{equation*}
in which a deterministic Meyers estimate and a minimal scale inequality hold. By Lemma~\ref{minscalecorrector} and Proposition~\ref{meyercorrector}, this collection satisfies the assumption of Proposition~\ref{p.partitions} (but not the assumption~\eqref{e.goodlocal}). We denote by~$\mathcal{U}$  the partition thus obtained. By the Property~(iii) of Proposition~\ref{p.partitions}, one has the inequality
\begin{equation*}
\size\left( \cu_{\mathcal{U}}(x) \right) \leq \O_s (C),
\end{equation*}
for some exponent $s := s(d , \p, \lambda) > 0$ and some constant $C := C(d , \p ,\lambda) < \infty$.
\end{definition}
Using the properties of the partition $\mathcal{U}$ and H\"{o}lder inequality, one obtains
\begin{align*}
\lefteqn{ \sum_{z \in  \C_\infty}\left|\nabla w_{\gamma_r} \right|^2(z)  \size \left(\cu_\Pa (z) \right)^{3d} \left( \sum_{x \in \Zd, \, \dist \left(\cu_\Pa(x) , \cu_\Pa (z)\right) \leq 1, \, e \in \mathcal{B}_d^x}\left( 1 + \X^e(x) \right)^d \right) } \qquad & \\ &
 = \sum_{\cu \in \mathcal{U}} \sum_{z \in \cu \cap \C_\infty}\left|\nabla w_{\gamma_r} \right|^2(z)  \\ & \qquad \times \size \left(\cu_\Pa (z) \right)^{3d} 
 \left( \sum_{x \in \Zd, \, \dist \left(\cu_\Pa(x) , \cu_\Pa (z)\right) \leq 1, \, e \in \mathcal{B}_d^x}\left( 1 + \X^e(x) \right)^d \right) \\ &
\leq  \sum_{\cu \in \mathcal{U}} |\cu| \left( \frac{1}{|\cu|} \sum_{z \in \cu \cap \C_\infty}\left|\nabla w_{\gamma_r}\right|^{2+\ep}(z) \right)^\frac{2}{2+\ep} \\ &
\qquad \times  \left( \frac{1}{|\cu|} \sum_{z \in \cu} \left( \size \left(\cu_\Pa (z) \right) \right)^{\frac{3d(2+\ep)}{\ep}}   \left( \sum_{x \in \Zd, \, \dist \left(\cu_\Pa(x) , \cu_\Pa (z)\right) \leq 1, \, e \in \mathcal{B}_d^x}\left( 1 + \X^e(x) \right)^d \right)^{\frac{2+\ep}{\ep}}\right)^\frac{\ep}{2+\ep} \\&
\leq C  \sum_{\cu \in \mathcal{U}} \left(  \sum_{x \in \frac 43 \cu \cap \C_\infty} \left|\nabla w_{\gamma_r} \right|^{2}(x) + |\cu| \left( \frac1{\left| \frac 43 \cu \right|} \sum_{\frac 43\cu \cap \C_\infty} \left|\gamma_r\right|^{2+\ep}(x)\right)^{\frac2{2+\ep}} \right).
\end{align*}
To estimate the term on the right-hand side, we note that the cube $\frac 43 \cu$ is included in the set $\bigcup_{\cu' \in \mathcal{U},\, \dist(\cu', \cu) \leq 1} \cu'$ and the cardinality of the set $ \left\lbrace \cu' \in \mathcal{U}\, : \, \dist(\cu', \cu) \leq 1 \right\rbrace$ is bounded by a constant depending only on the dimension $d$. This leads to the upper bound
\begin{equation*}
\sum_{\cu \in \mathcal{U}}  \sum_{x \in \frac 43 \cu \cap \C_\infty} \left|\nabla w_{\gamma_r} \right|^{2}(x) \leq C \sum_{ x \in \C_\infty} \left|\nabla w_{\gamma_r} \right|^{2}(x)  \leq \O_s'\left(\frac{C}{r^d}\right).
\end{equation*}
To estimate the second term on the right-hand side, we recall the discrete $l^1-l^t$-estimate: for any finite sequence of nonnegative real numbers $(b_i)_{0\leq i \leq n} \in \R_+^{n+1}$ and any exponent $t \geq 1$, $\sum_{i = 0}^n b_i^t \leq \left( \sum_{i = 0}^n b_i \right)^t$. Using this inequality, we obtain
\begin{align*}
\sum_{\cu \in \mathcal{U}}   |\cu| \left( \frac1{\left| \frac 43 \cu \right|} \sum_{x \in \frac 43\cu \cap \C_\infty} \left|\gamma_r\right|^{2+\ep}(x) \right)^{\frac2{2+\ep}}  & \leq C \sum_{\cu \in \mathcal{U}} |\cu|^{1 - \frac 2{2+\ep}} \sum_{x \in \frac 43\cu \cap \C_\infty} \gamma_r^{2}(x) \\
&  \leq C \sum_{x \in \C_{\infty}} \left|\gamma_r\right|^2(x) \size \left(\cu_\mathcal{U}(x) \right)^{d \left( 1 - \frac 2{2+\ep} \right)} \\
& \leq C \sum_{x \in \Zd} \zeta_{r}(x)^2 \size \left(\cu_\Pa(x) \right)^{4d} \size \left(\cu_\mathcal{U}(x) \right)^{d \left( 1 - \frac 2{2+\ep} \right)}.
\end{align*}
Using the inequalities $ \size \left(\cu_\mathcal{U}(x) \right) \leq \O_s'(C)$, $\size \left(\cu_\Pa(x) \right) \leq \O_s'(C)$ and~\eqref{e.Oavgs}, we obtain
\begin{equation*}
\sum_{\cu \in \mathcal{U}}   |\cu| \left( \frac1{\left| \frac 43 \cu \right|} \sum_{x \in \frac 43\cu \cap \C_\infty} \left|\gamma_r\right|^{2+\ep}(x) \right)^{\frac2{2+\ep}} \leq \O_s' \left(\frac C{r^d} \right).
\end{equation*}
The proof of Result 2, and thus of Proposition~\ref{boundspatialaveragecorrector}, is complete.

\section{Optimal $L^q$ estimates for first order corrector}\label{section4}

In this section, we show how to obtain $L^q$ optimal scaling estimates on the corrector (Theorem~\ref{mainthm}) from Proposition~\ref{boundspatialaveragecorrector}. We first restate the result.

\begin{reptheorem}{mainthm}[Optimal $L^q$ estimates for first order corrector]
There exist two exponents $s := s(d, \p , \lambda) > 0$, $k := k(d, \p , \lambda) < \infty$ and a constant $C(d, \p,  \lambda) < \infty$ such that for each radius $R \geq 1$, each exponent $q \geq 1$ and each $p \in \Rd$,
\begin{equation} \label{e.mainthmbis52}
\left( R^{-d} \sum_{x \in \C_\infty \cap B_R} \left| \chi_p(x) - \left( \chi_p \right)_{\C_\infty \cap B_R} \right|^q  \right)^\frac 1q \leq  \left\{
  \begin{array}{lcl}
    \O_s \left( C |p| q^k \log^\frac12 R \right) & \mbox{if } d=2, \\
     \O_s \left( C|p| q^k \right) & \mbox{if } d\geq 3. \\
  \end{array}
\right.
\end{equation}
\end{reptheorem}

Before starting the proof, we mention that, in this section, we need to keep track of the dependence on the parameter $q$ of the constants as it will be useful in the next section to obtain the $L^\infty$ bounds on the corrector.

\begin{proof}[Proof of Theorem~\ref{mainthm}.]
As in the proof of Proposition~\ref{boundspatialaveragecorrector}, we assume that $|p| = 1$ to ease the notations. Additionally, note that, by the Jensen inequality, it is enough to prove Theorem~\ref{mainthm} in the case $q \geq 2$. We consequently make this assumption for the rest of the proof. The proof of this theorem is split into two steps.
\begin{itemize}
\item In Step 1, we use Proposition~\ref{boundspatialaveragecorrector} and the multiscale Poincar\'e inequality (Proposition~\ref{multiscalepoincare}) to show, for each radius $R \geq 1$,
\begin{equation*}
\left( R^{-d} \sum_{x \in B_R} \left| \left[ \chi_p \right]_{\Pa}(x) - \left( \left[ \chi_p\right]_{\Pa} \right)_{B_R} \right|^q  \right)^\frac 1q \leq  \left\{
  \begin{array}{lcl}
    \O_s \left( C q^k \log^\frac12 R \right) & \mbox{if } d=2, \\
     \O_s \left( C q^k\right) & \mbox{if } d\geq 3, \\
  \end{array}
\right.
\end{equation*}
where the parameters $C , k$ and $s$ depend only on $s , \p, \lambda.$
\item In Step 2, we remove the coarsening, thanks to Proposition~\ref{l.coarseLs}, to obtain
\begin{equation*}
\left( R^{-d} \sum_{x \in \C_\infty \cap B_R} \left| \chi_p(x) - \left( \chi_p \right)_{\C_\infty \cap B_R} \right|^q  \right)^\frac 1q \leq  \left\{
  \begin{array}{lcl}
    \O_s \left( C q^k \log^\frac12 R \right) & \mbox{if } d=2, \\
     \O_s \left( C q^k \right) & \mbox{if } d\geq 3. \\
  \end{array}
\right.
\end{equation*}
\end{itemize}

\emph{Step 1.} We apply Proposition~\ref{multiscalepoincare} to the function $u = \left[ \chi_p \right]_{\Pa}$ and obtain, for each $R \geq 1$,
\begin{equation*}
\left\| \left[ \chi_p \right]_{\Pa} - \left(\left[ \chi_p \right]_{\Pa}\right)_{B_R}  \right\|_{L^q\left(B_R\right)} \leq C  \left( \sum_{x \in \Zd} e^{ - \frac{|x|}{2R} } \left( \int_{0}^{2R} r \left| \Phi_{r^2} * \nabla \left[ \chi_p \right]_{\Pa} (x) \right|^2 \, dr  \right)^{\frac q2} \right)^\frac1q.
\end{equation*}
To study the term on the right-hand side, we split the integral into two terms
\begin{equation} \label{splittingtheeqaution}
\int_{0}^{2R} r \left|  \Phi_{r^2} * \nabla \left[ \chi_p \right]_{\Pa} (x) \right|^2 \, dr = \int_{0}^{1} r \left| \Phi_{r^2} * \nabla \left[ \chi_p \right]_{\Pa} (x) \right|^2 \, dr + \int_{1}^{2R} r \left| \Phi_{r^2} * \nabla \left[ \chi_p \right]_{\Pa} (x) \right|^2 \, dr.
\end{equation}
By Proposition~\ref{boundspatialaveragecorrector}, we know that for each radius $r \geq 1$ and each point $x \in \Rd$,
\begin{equation*}
\left| \Phi_{r^2} * \nabla \left[ \chi_p \right]_{\Pa} (x) \right| \leq \O_s\left( C r^{- \frac d2} \right).
\end{equation*}
This implies
\begin{equation*}
\left| \Phi_{r^2} * \nabla \left[ \chi_p \right]_{\Pa} (x) \right|^2 \leq \O_{ s}\left( C r^{- d} \right).
\end{equation*}
The second term on the right-hand side can be estimated by using Proposition~\ref{boundspatialaveragecorrector} and the inequality~\eqref{e.Oavgs}. We obtain
\begin{equation*}
\int_{1}^{2R} r \left| \Phi_{r^2} * \nabla \left[ \chi_p \right]_{\Pa} (x) \right|^2 \, dr \leq  \left\{
  \begin{array}{lcl}
    \O_s \left( C \log R \right) & \mbox{if } d=2, \\
     \O_s \left( C \right) & \mbox{if } d\geq 3. \\
  \end{array}
\right.
\end{equation*}
To estimate the first term on the right-hand side of~\eqref{splittingtheeqaution}, we use Proposition~\ref{p.lipboundcorr} which reads, for each point $x \in \Rd$,
\begin{equation} \label{eq:TV16351602}
 \left| \nabla \left[ \chi_p \right]_{\Pa} (x) \right| \leq \O_s \left( C \right).
\end{equation}
Combining the inequality~\eqref{eq:TV16351602} with~\eqref{e.Oavgs}, we obtain
\begin{equation*}
\int_{0}^{1} r \left| \Phi_{r^2} * \nabla \left[ \chi_p \right]_{\Pa} (x) \right|^2 \, dr \leq \O_s \left( C \right).
\end{equation*}
Combining the previous displays shows
\begin{equation*}
\int_{0}^{2R} r \left| \Phi_{r^2} * \nabla \left[ \chi_p \right]_{\Pa} (x) \right|^2 \, dr  \leq \left\{  \begin{array}{lcl}
    \O_s \left( C \log R \right) & \mbox{if } d=2, \\
     \O_s \left( C \right) & \mbox{if } d\geq 3. \\
  \end{array}
\right.
\end{equation*}
We obtain
\begin{equation*}
\left( \int_{0}^{2R} r \left| \Phi_{r^2} * \nabla \left[ \chi_p \right]_{\Pa} (x) \right|^2 \, dr \right)^\frac q2  \leq \left\{  \begin{array}{lcl}
    \O_{\frac{2s}{q}} \left( C^{\frac q2} \left( \log R \right)^\frac q2 \right) & \mbox{if } d=2, \\
     \O_{\frac{2s}{q}}  \left( C^{\frac q2} \right) & \mbox{if } d\geq 3. \\
  \end{array}
\right.
\end{equation*}
We apply~\eqref{e.Oavgs} and keep track of the dependence of the constants in the exponent $q$ thanks to~\eqref{est.cte.Oavgs}. We obtain
\begin{equation*}
\sum_{x \in \Zd} R^{-d}e^{ - \frac{|x|}{2R} } \left( \int_{0}^{2R} r \left| \Phi_{r^2} * \nabla \left[ \chi_p \right]_{\Pa} (x) \right|^2 \, dr  \right)^{\frac q2}  \leq \left\{  \begin{array}{lcl}
    \O_{\frac{2s}{q}} \left( \left( \frac{q}{s \ln(2)} \right)^\frac qs C^{\frac q2}  \left( \log R \right)^\frac q2 \right) & \mbox{if } d=2, \\
     \O_{\frac{2s}{q}}  \left( \left( \frac{q}{s \ln(2)} \right)^\frac qs C^{\frac q2}  \right) & \mbox{if } d\geq 3. \\
  \end{array}
\right.
\end{equation*}
This eventually yields
\begin{equation*}
\left( \sum_{x \in \Zd} R^{-d}e^{ - \frac{|x|}{2R} } \left( \int_{0}^{2R} r \left| \Phi_{r^2} * \nabla \left[ \chi_p \right]_{\Pa} (x) \right|^2 \, dr  \right)^{\frac q2} \right)^\frac 1q  \leq \left\{  \begin{array}{lcl}
    \O_s \left( q^\frac 1s C  \left( \log R \right)^\frac 12 \right) & \mbox{if } d=2, \\
     \O_s  \left( q^\frac 1s C  \right) & \mbox{if } d\geq 3. \\
  \end{array}
\right.
\end{equation*}
We now set $k := \frac 1s + \frac 32$; this exponent depends only on the parameters $d , \p , \lambda$. By applying Proposition~\ref{multiscalepoincare}, we obtain
\begin{equation} \label{quotetheresultwithcoarsenedfunctions}
\left( R^{-d} \sum_{x \in \Zd} \left| \left[ \chi_p \right]_{\Pa}(x) - \left( \left[ \chi_p\right]_{\Pa} \right)_{B_R} \right|^q  \right)^\frac 1q \leq  \left\{
  \begin{array}{lcl}
    \O_s \left( C q^k \log^\frac12 R \right) & \mbox{if } d=2, \\
     \O_s \left( C q^k \right) & \mbox{if } d\geq 3. \\
  \end{array}
\right.
\end{equation}
The proof of Step 1 is complete.

\smallskip

\emph{Step 2.} In this step, we remove the coarsening from~\eqref{quotetheresultwithcoarsenedfunctions} thanks to Proposition~\ref{l.coarseLs}. We write
\begin{align} \label{splitcorrLq}
\left(3^{-dm}\sum_{x \in \C_\infty \cap  \cu_m} \left| \chi_p(x) - \left( \chi_p \right)_{\C_\infty \cap  \cu_m} \right|^q \right)^\frac 1q 
& \leq 2 \inf_{a \in \R}  \left(  3^{-dm} \sum_{x \in \C_\infty \cap  \cu_m} \left| \chi_p(x) - a \right|^q \right)^\frac 1q  \notag \\
&\leq 2 \left(3^{-dm}\sum_{x \in \C_\infty \cap  \cu_m} \left| \chi_p(x) - \left( \left[ \chi_p \right]_\Pa \right)_{\cu_m} \right|^q \right)^\frac 1q \\
&  \leq 2 \left( 3^{-dm}\sum_{x \in \C_\infty \cap  \cu_m} \left| \chi_p(x) - \left[ \chi_p \right]_{\Pa}(x)  \right|^q \right)^\frac 1q  \notag \\
		& \qquad + 2 \left(  3^{-dm}\sum_{x \in \C_\infty \cap  \cu_m} \left| \left[ \chi_p \right]_{\Pa}(x) - \left( \left[ \chi_p \right]_{\Pa} \right)_{\cu_m} \right|^q  \right)^\frac 1q. \notag
\end{align}
To estimate the first term on the right-hand side, we first use the inclusion~\eqref{cubeandneighbors} and Proposition~\ref{l.coarseLs} to obtain
\begin{align*}
\sum_{x \in \C_\infty \cap  \cu_m} \left| \chi_p(x) - \left[ \chi_p \right]_{\Pa}(x) \right|^q & \leq \sum_{x\in \C_*\left( \cu_{m+1} \right)} \left| \chi_p(x) - \left[ \chi_p \right]_{\Pa}(x) \right|^q \\ 
& \leq  C \sum_{x \in \C_*\left(\cu_{m+1}\right)}  \size(\cu_\Pa(x))^{qd} \left| \nabla \chi_p \right|^q(x)   \\
& \leq C \sum_{x \in \C_*\left(\cu_{m+1}\right)}  \size(\cu_\Pa(x))^{qd} \left| \nabla \chi_p \right|^q(x).
\end{align*}
By the Lipschitz bounds on the gradient of the corrector and the property of the partition $\Pa$, we have, for each point $x \in \Zd$,
\begin{equation*}
\size(\cu_\Pa(x))^{qd} \left| \nabla \chi_p \right|^q(x) \indc_{\{ x \in \C_\infty \}} \leq \O_{\frac sq} \left(C^q\right).
\end{equation*}
Additionally, for each point $x \in \Zd \setminus \cu_{m+1}$,
\begin{equation*}
    \indc_{\{ x \in \C_*\left( \cu_{m+1} \right)\} } \leq \frac{\size(\cu_{\Pa}(x))^{d+1}}{|x|^{d+1} \vee 1} \leq \O_s\left(\frac{C}{\left(|x| \vee 1\right)^{d+1}}\right).
\end{equation*}
We then use the inequalities~\eqref{e.Oavgs} and~\eqref{est.cte.Oavgs} to obtain
\begin{align*}
\sum_{ x \in \C_*\left( \cu_{m+1} \right)}  \size(\cu_\Pa(x))^{qd} \left| \nabla \chi_p \right|^q(x) & = \sum_{x \in \Zd}  \size(\cu_\Pa(x))^{qd} \left| \nabla \chi_p \right|^q(x) \indc_{\{x \in \C_*\left( \cu_{m+1} \right) \}}\\
																			& \leq \O_{\frac sq} \left( 3^{dm}  \left( \frac{q}{s \ln(2)} \right)^\frac qs  C^q  \right) \\
																			& \leq \O_{\frac sq} \left( 3^{dm}  q^\frac qs  C^q  \right).
\end{align*}
By~\eqref{e.sizecluster}, we obtain
\begin{equation} \label{eq:TV09211702}
  \left( 3^{-dm}\sum_{x \in \C_\infty \cap  \cu_m} \left| \chi_p(x) - \left[ \chi_p \right]_{\Pa}(x) \right|^q \right)^\frac 1q \leq \O_{s} \left( q^\frac 1s  C  \right).
\end{equation}
To estimate the second term on the right-hand side of~\eqref{splitcorrLq}, we note that
\begin{equation} \label{eq:TV09221702}
  \left(  3^{-dm} \sum_{x \in \C_\infty \cap  \cu_m} \left| \left[ \chi_p \right]_{\Pa}(x) - \left( \left[ \chi_p \right]_{\Pa} \right)_{\cu_m} \right|^q  \right)^\frac 1q \leq  \left(  3^{-dm} \sum_{x \in \cu_m} \left| \left[ \chi_p \right]_{\Pa}(x)  -  \left( \left[ \chi_p \right]_{\Pa} \right)_{ \cu_m} \right|^q \right)^\frac 1q.
\end{equation}
We combine~\eqref{eq:TV09211702},~\eqref{eq:TV09221702} and apply~\eqref{quotetheresultwithcoarsenedfunctions} to obtain
\begin{equation*}
 \left(  3^{-dm} \sum_{x \in \C_\infty \cap  \cu_m}  \left| \chi_p(x) - \left( \chi_p \right)_{\C_\infty \cap  \cu_m} \right|^q \right)^\frac 1q  \leq \left\{  \begin{array}{lcl}
    \O_{s} \left(q^k C  m^\frac 12 \right) & \mbox{if } d=2, \\
     \O_{s}  \left(q^k C  \right) & \mbox{if } d\geq 3, \\
  \end{array}
\right.
\end{equation*}
for some exponents $k := k(d , \p , \lambda),$ $s:= s (d , \p , \lambda) >0 $ and some constant $C:= C(d , \p , \lambda) < \infty$. The result of Theorem~\ref{mainthm} requires to prove the previous inequality for a general ball $B_R$ and not a triadic cube. This result is obtained by selecting, for each radius $R \geq 1$, the integer $m$ such that $3^{m} < R \leq 3^{m+1}$ and by performing a similar analysis.
\end{proof}

\section{Optimal $L^\infty$ estimates for the first order corrector} \label{section5}
In this section, we prove the $L^\infty$ bound on the corrector, Theorem~\ref{realmainthm}.

\begin{reptheorem}{realmainthm}[Optimal $L^\infty$ estimates for first order correctors]
There exist an exponent $s := s(d , \p , \lambda) > 0$ and a constant $C := C(d , \p , \lambda) < \infty$ such that for each $x,y \in \Zd$ and each $p \in \Rd$,
\begin{equation*}
\left| \chi_p(x) - \chi_p(y) \right| \indc_{\{ x,y \in \C_\infty\}}\leq \left\{  \begin{array}{lcl}
    \O_{s} \left(C  |p| \log^\frac 12  |x-y|  \right) & \mbox{if } d=2, \\
     \O_{s}  \left(C|p|  \right) & \mbox{if } d\geq 3. \\
  \end{array}
\right.
\end{equation*}
\end{reptheorem}

\begin{proof}
First by the stationarity of the gradient of the corrector, we can assume that $y = 0$. Without loss of generality, we can also assume $|p| = 1$. We want to prove, for each point $x \in \Zd$,
\begin{equation*}
\left| \chi_p(x) - \chi_p(0) \right| \indc_{\{ 0,x \in \C_\infty\}}\leq \left\{  \begin{array}{lcl}
    \O_{s} \left(C  \log^\frac 12 |x| \right) & \mbox{if } d=2, \\
     \O_{s}  \left(C  \right) & \mbox{if } d\geq 3. \\
  \end{array}
\right.
\end{equation*}
We record that, for every exponent $q >0$ and every point $x \in \R_+$
\begin{equation*}
\exp(x) \geq \frac{x^q}{q^q \exp(-q)}.
\end{equation*}
This implies, for each triplet $s,q,\theta > 0$, 
\begin{equation} \label{est.exp.pol}
X \leq \O_s (\theta) \implies \E[ X^q ] \leq 2 \theta^q \left( \frac{q}{s} \right)^{\frac{q}{s} }  \exp \left( \frac{q}{s} \right).
\end{equation}
We split the proof into six steps.
\begin{itemize}
\item In Step 1, we prove that for each exponent $q \geq 1$ and each integer $m \in \N$,
\begin{equation*}
\E \left[ \left| \left[ \chi_p \right]_\Pa (0) - 3^{-2dm}\sum_{y \in \cu_m} \sum_{z \in y + \cu_m} \left[ \chi_p \right]_\Pa (z) \right|^q \right] \leq  \left\{ \begin{array}{lcl}
    C^q q^{qk} m^{\frac q2} & \mbox{if } d=2, \\
     C^q q^{qk} & \mbox{if } d\geq 3. \\
  \end{array}
\right.
\end{equation*}
\item In Step 2, we use the result of Step 1 to prove that for each exponent $q \geq 1$ and each integer $m \in \N$,
\begin{equation*}
\E \left[ \left| \left[ \chi_p \right]_\Pa (x) - 3^{-2dm} \sum_{y \in \cu_m} \sum_{z \in x + y+ \cu_m} \left[ \chi_p \right]_\Pa (z) \right|^q \right] \leq  \left\{ \begin{array}{lcl}
    C^q q^{qk} m^{\frac q2} & \mbox{if } d=2, \\
     C^q q^{qk} & \mbox{if } d\geq 3. \\
  \end{array}
\right.
\end{equation*}
Note that this statement is not just a consequence of Step 1 and the stationarity of the corrector since the partition $\Pa$ is not stationary.
\item In Step 3, we prove that for each $q \geq 1$ and $m \in \N$, chosen such that $3^m \leq |x| < 3^{m+1}$,
\begin{equation*}
\E \left[ \left| \left[ \chi_p \right]_\Pa (0) - 3^{-2dm} \sum_{y \in \cu_m} \sum_{z \in x + y+ \cu_m} \left[ \chi_p \right]_\Pa (z) \right|^q \right] \leq  \left\{ \begin{array}{lcl}
    C^q q^{qk} m^{\frac q2} & \mbox{if } d=2, \\
     C^q q^{qk} & \mbox{if } d\geq 3. \\
  \end{array}
\right.
\end{equation*}
\item In Step 4, we combine Steps 2 and 3 to obtain, for each exponent $q \geq 1$,
\begin{equation} \label{est.step4.linfty}
\E \left[ \left| \left[ \chi_p \right]_\Pa (x) -\left[ \chi_p \right]_\Pa (0) \right|^q \right] \leq \left\{ \begin{array}{lcl}
    C^q q^{qk} m^{\frac q2} & \mbox{if } d=2, \\
     C^q q^{qk} & \mbox{if } d\geq 3. \\
  \end{array}
\right. 
\end{equation}
\item In Step 5, we prove that there exist an exponent $s := s(d , \p, \lambda) > 0$ and a constant $C := C(d , \p, \lambda ) < \infty$ such that
\begin{equation} \label{est.step5.linfty}
\left| \left[ \chi_p \right]_\Pa (x) -\left[ \chi_p \right]_\Pa (0) \right| \leq \left\{  \begin{array}{lcl}
    \O_{s} \left(C  \log^\frac 12 |x| \right) & \mbox{if } d=2, \\
     \O_{s}  \left(C  \right) & \mbox{if } d\geq 3. \\
  \end{array}
\right.
\end{equation}
\item In Step 6, we remove the coarsening and show that
\begin{equation*}
\left| \chi_p(x) - \chi_p(0) \right| \indc_{\{ 0,x \in \C_\infty\}}\leq \left\{  \begin{array}{lcl}
    \O_{s} \left(C  \log^\frac 12 |x| \right) & \mbox{if } d=2, \\
     \O_{s}  \left(C  \right) & \mbox{if } d\geq 3. \\
  \end{array}
\right.
\end{equation*}
\end{itemize}

\medskip

\textit{Step 1.}  The main tool of this step is the following inequality which was proved in Step 1 of the proof of Theorem~\ref{mainthm}, for each $m \in \N$, and each $q \geq 1$,
\begin{equation} \label{maintoolstep1Linfty}
\left( 3^{-dm} \sum_{y \in  \cu_m} \left| \left[ \chi_p \right]_{\Pa}(y) - \left( \left[ \chi_p\right]_{\Pa} \right)_{\cu_m} \right|^q  \right)^\frac 1q \leq  \left\{
  \begin{array}{lcl}
    \O_s \left( C q^k \sqrt{m} \right) & \mbox{if } d=2, \\
     \O_s \left( C q^k\right) & \mbox{if } d\geq 3. \\
  \end{array}
\right.
\end{equation}
Note that this implies, by~\eqref{est.exp.pol},
\begin{equation} \label{step1LinftyLp}
\E \left[3^{-dm} \sum_{y \in \cu_m} \left| \left[ \chi_p \right]_{\Pa}(y) - \left( \left[ \chi_p\right]_{\Pa} \right)_{\cu_m} \right|^q   \right] \leq
\left\{
  \begin{array}{lcl}
    C^q q^{qk} m^{\frac q2} & \mbox{if } d=2, \\
     C^q q^{qk} & \mbox{if } d\geq 3. \\
  \end{array}
\right.
\end{equation}
For some fixed point $y \in \Zd$, the stationarity of the corrector stated in~\eqref{e.stat.corr} implies the identity, for almost every environment $\a \in \Omega$,
\begin{equation*}
\left( \left[ \chi_p \right]_{\Pa}(-y) - \left( \left[ \chi_p\right]_{\Pa} \right)_{\cu_m} \right) (\a) = \left( \left[ \chi_p \right]_{\Pa_y}(0) - \left( \left[ \chi_p\right]_{\Pa_y} \right)_{y +\cu_m} \right) (\tau_y \a),
\end{equation*}
where we recall the notation $\Pa_y = y + \Pa(\tau_{-y} \a)$. Using the stationarity property~\eqref{Pisstat}, we obtain, for each exponent $q \geq 1$,
\begin{equation*}
\E \left[ \left| \left[ \chi_p \right]_{\Pa_y}(0) - \left( \left[ \chi_p\right]_{\Pa_y} \right)_{y +\cu_m} \right|^q \right] =  \E \left[ \left| \left[ \chi_p \right]_{\Pa}(-y) - \left( \left[ \chi_p\right]_{\Pa} \right)_{\cu_m} \right|^q \right].
\end{equation*}
Since this result applies for each point $y \in \Zd$, we can sum over all the points in the cube $\cu_m$. We obtain
\begin{equation*}
3^{-dm} \sum_{y \in \cu_m} \E \left[ \left| \left[ \chi_p \right]_{\Pa_y}(0) - \left( \left[ \chi_p\right]_{\Pa_y} \right)_{y +\cu_m} \right|^q \right] = 3^{-dm} \sum_{y \in \cu_m}  \E \left[ \left| \left[ \chi_p \right]_{\Pa}(-y) - \left( \left[ \chi_p\right]_{\Pa} \right)_{\cu_m} \right|^q \right] .
\end{equation*}
Thus, by~\eqref{step1LinftyLp},
\begin{equation} \label{estwithy.5}
\E \left[ 3^{-dm} \sum_{y \in \cu_m}   \left| \left[ \chi_p \right]_{\Pa_y}(0) - \left( \left[ \chi_p\right]_{\Pa_y} \right)_{y +\cu_m} \right|^q \right]  \leq \left\{
  \begin{array}{lcl}
    C^q q^{qk} m^{\frac q2} & \mbox{if } d=2, \\
     C^q q^{qk} & \mbox{if } d\geq 3. \\
  \end{array}
\right.
\end{equation}
We now remove the translation of the partition and prove, for each point $z \in \Zd$
\begin{equation} \label{est.PPy}
\left| \left[ \chi_p \right]_{\Pa_y}(z) -  \left[ \chi_p \right]_{\Pa}(z) \right| \leq \O_s (C).
\end{equation}
To prove this inequality, note that, by definition of the coarsening~\eqref{d.coarsenedfnct}, we have
\begin{equation*}
\left[ \chi_p \right]_{\Pa_y}(z) -  \left[ \chi_p \right]_{\Pa}(z) =  \chi_p\left(\bar{z}(\cu_{\Pa_y} (z)) \right) - \chi_p\left(\bar{z}(\cu_{\Pa} (z))\right),
\end{equation*}
and by definition of the two partitions $\Pa$ and $\Pa_y$, there exists a path connecting the cubes $\cu_{\Pa_y} (z)$ and $\cu_{\Pa} (z)$ which lies in the ball $ B\left(z , C \max\left( \size\left( \cu_{\Pa_y}(z)\right), \size\left( \cu_{\Pa}(z)\right) \right) \right)$. To simplify the notation in the following computation, we denote by $R' = C \max\left( \size\left( \cu_{\Pa_y}(z)\right), \size\left( \cu_{\Pa}(z)\right) \right) $. We have the estimate
\begin{equation*}
\left| \left[ \chi_p \right]_{\Pa_y}(z) -  \left[ \chi_p \right]_{\Pa}(z) \right| \leq \sum_{x \in \C_\infty \cap B_{R'}(z)} \left| \nabla \chi_p \right| (x).
\end{equation*}
By Proposition~\ref{p.regularitytheory}, the bounds $R' \leq \O_s(C)$ and $\X(z) \leq \O_s(C)$ and the assumption $|p|=1$, we have
\begin{equation*}
\sum_{y \in \C_\infty \cap  B_R'(z) } \left| \nabla \chi_p \right| (y)  \leq \O_s(C).
\end{equation*} 
Combining the previous displays completes the proof of~\eqref{est.PPy}. To remove the parameter $y$ in~\eqref{estwithy.5}, we compute
\begin{align} \label{removingthey}
\lefteqn{ \E \left[3^{-dm} \sum_{y \in \cu_m}   \left| \left[ \chi_p \right]_{\Pa}(0) - \left( \left[ \chi_p\right]_{\Pa} \right)_{y +\cu_m} \right|^q  \right]} \qquad & \\  & 
		\leq 2^q \E \left[ 3^{-dm} \sum_{y \in \cu_m}   \left| \left[ \chi_p \right]_{\Pa_y}(0) - \left( \left[ \chi_p\right]_{\Pa_y} \right)_{y +\cu_m} \right|^q \right] \notag \\ &
		\quad + 2^q  \E \left[ 3^{-dm} \sum_{y \in \cu_m}   \left| \left[ \chi_p \right]_{\Pa_y}(0) - \left( \left[ \chi_p\right]_{\Pa_y} \right)_{y +\cu_m} -\left[ \chi_p \right]_{\Pa}(0) - \left( \left[ \chi_p\right]_{\Pa} \right)_{y +\cu_m} \right|^q  \right]. \notag
\end{align}
By~\eqref{est.PPy} and~\eqref{e.Oavgs}, we have, for each point $y \in \cu_m$,
\begin{equation*}
\left| \left[ \chi_p \right]_{\Pa_y}(0) - \left( \left[ \chi_p\right]_{\Pa_y} \right)_{y +\cu_m} -\left[ \chi_p \right]_{\Pa}(0) - \left( \left[ \chi_p\right]_{\Pa} \right)_{y +\cu_m} \right| \leq \O_s(C),
\end{equation*}
and thus
\begin{equation*}
\E \left[ \left| \left[ \chi_p \right]_{\Pa_y}(0) - \left( \left[ \chi_p\right]_{\Pa_y} \right)_{y +\cu_m} -\left[ \chi_p \right]_{\Pa}(0) - \left( \left[ \chi_p\right]_{\Pa} \right)_{y +\cu_m} \right|^q \right] \leq C^q q^{qk}.
\end{equation*}
Summing over the points $y \in \cu_m$ yields
\begin{equation*}
  \E \left[ 3^{-dm} \sum_{y \in \cu_m}   \left| \left[ \chi_p \right]_{\Pa_y}(0) - \left( \left[ \chi_p\right]_{\Pa_y} \right)_{y +\cu_m} -\left[ \chi_p \right]_{\Pa}(0) - \left( \left[ \chi_p\right]_{\Pa} \right)_{y +\cu_m} \right|^q \right] \leq C^q q^{qk}.
\end{equation*}
By the previous display and~\eqref{estwithy.5}, we have
\begin{equation*}
\E \left[ 3^{-dm} \sum_{y \in \cu_m}   \left| \left[ \chi_p \right]_{\Pa}(0) - \left( \left[ \chi_p\right]_{\Pa} \right)_{y +\cu_m} \right|^q \right] \leq \left\{
  \begin{array}{lcl}
    C^q q^{qk} m^{\frac q2} & \mbox{if } d=2, \\
     C^q q^{qk} & \mbox{if } d\geq 3. \\
  \end{array}
\right.
\end{equation*}
By the Jensen inequality, we obtain
\begin{equation*}
\E \left[ \left| 3^{-dm} \sum_{y \in \cu_m}   \left[ \chi_p \right]_{\Pa}(0) - \left( \left[ \chi_p\right]_{\Pa} \right)_{y +\cu_m} \right|^q \right] \leq \left\{
  \begin{array}{lcl}
    C^q q^{qk} m^{\frac q2} & \mbox{if } d=2, \\
     C^q q^{qk} & \mbox{if } d\geq 3. \\
  \end{array}
\right.
\end{equation*}
Notice that
\begin{equation*}
3^{-dm} \sum_{y \in \cu_m}   \left[ \chi_p \right]_{\Pa}(0) - \left( \left[ \chi_p\right]_{\Pa} \right)_{y +\cu_m}  = \left[ \chi_p \right]_\Pa (0) - 3^{-2dm}\sum_{y \in \cu_m} \sum_{z \in y + \cu_m} \left[ \chi_p \right]_\Pa (z).
\end{equation*}
Combining the two previous displays completes the proof of Step 1.

\medskip

\textit{Step 2.} By the stationarity of the corrector~\eqref{e.stat.corr}, for almost every environment $\a \in \Omega$, every pair of points $y,z \in \Zd$,
\begin{equation*}
\left[ \chi_p \right]_{\Pa}(z) (\a) = \left[ \chi_p \right]_{\Pa_y}(z + y) (\tau_y \a).
\end{equation*}
Using this property, we have
\begin{align*}
\E \left[ \left| \left[ \chi_p \right]_{\Pa_x} (x) - 3^{-2dm} \sum_{y \in \cu_m} \sum_{z \in x + y+ \cu_m} \left[ \chi_p \right]_{\Pa_x} (z) \right|^q \right] &= 
\E \left[ \left| \left[ \chi_p \right]_\Pa (0) - 3^{-2dm} \sum_{y \in \cu_m} \sum_{z \in y+ \cu_m} \left[ \chi_p \right]_\Pa (z) \right|^q \right] \\
& \leq \left\{
  \begin{array}{lcl}
    C^q q^{qk} m^{\frac q2} & \mbox{if } d=2, \\
     C^q q^{qk} & \mbox{if } d\geq 3. \\
  \end{array}
\right.
\end{align*}
Performing the same computation as in~\eqref{removingthey}, we can replace the partition $\Pa_x$ by $\Pa$ in the previous display. This yields
\begin{equation*}
\E \left[ \left| \left[ \chi_p \right]_{\Pa} (x) - 3^{-2dm} \sum_{y \in \cu_m} \sum_{z \in x + y+ \cu_m} \left[ \chi_p \right]_{\Pa} (z) \right|^q \right] \leq \left\{
  \begin{array}{lcl}
    C^q q^{qk} m^{\frac q2} & \mbox{if } d=2, \\
     C^q q^{qk} & \mbox{if } d\geq 3. \\
  \end{array}
\right.
\end{equation*}
This completes the proof of Step 2.

\medskip

\textit{Step 3.} This step is similar to Step 1; the main ingredient is slightly different and presented below. The objective is to prove the following inequality: for $m \in \N$ such that $3^m \leq |x| < 3^{m+1}$, and for each $q \geq 1$,
\begin{equation*}
\left( 3^{-dm} \sum_{y \in \cu_m}  \left| \left[ \chi_p \right]_{\Pa}(y) - \left( \left[ \chi_p\right]_{\Pa} \right)_{x + \cu_m} \right|^q  \right)^\frac 1q \leq  \left\{
  \begin{array}{lcl}
    \O_s \left( C q^k \sqrt{m} \right) & \mbox{if } d=2, \\
     \O_s \left( C q^k\right) & \mbox{if } d\geq 3. \\
  \end{array}
\right.
\end{equation*}
To prove this result, we note that $x + \cu_m \subseteq \cu_{m+2}$ and compute
\begin{multline} \label{eq:TV10402302}
\left( 3^{-dm} \sum_{y \in \cu_m}  \left| \left[ \chi_p \right]_{\Pa}(y) - \left( \left[ \chi_p\right]_{\Pa} \right)_{x + \cu_m} \right|^q  \right)^\frac 1q \\
\leq  \left( 3^{-dm} \sum_{y \in \cu_{m+2}}  \left| \left[ \chi_p \right]_{\Pa}(y) - \left( \left[ \chi_p\right]_{\Pa} \right)_{\cu_{m+2}} \right|^q  \right)^\frac 1q +  \left| \left( \left[ \chi_p\right]_{\Pa} \right)_{\cu_{m+2}} -  \left( \left[ \chi_p\right]_{\Pa} \right)_{x + \cu_m} \right|.
\end{multline} 
The first term is estimated by~\eqref{step1LinftyLp} (replacing the cube $\cu_m$ by $\cu_{m+2}$). We estimate the second term on the right side of~\eqref{eq:TV10402302} as follows
\begin{align*}
 \left| \left( \left[ \chi_p\right]_{\Pa} \right)_{\cu_{m+2}} -  \left( \left[ \chi_p\right]_{\Pa} \right)_{x + \cu_m} \right|&  \leq 3^{-dm}\sum_{y \in x + \cu_m}  \left| \left[ \chi_p\right]_{\Pa}(y)  - \left( \left[ \chi_p\right]_{\Pa} \right)_{\cu_{m+2}} \right| \\&
\leq C \left( 3^{-dm} \sum_{y \in \cu_{m+2}}  \left| \left[ \chi_p\right]_{\Pa}(y)  - \left( \left[ \chi_p\right]_{\Pa} \right)_{\cu_{m+2}} \right|^q \right)^\frac 1q. 
\end{align*}
Combining the two previous displays with the inequality~\eqref{maintoolstep1Linfty} shows
\begin{equation*}
\left( 3^{-dm} \sum_{y \in  \cu_m} \left| \left[ \chi_p \right]_{\Pa}(y) - \left( \left[ \chi_p\right]_{\Pa} \right)_{x + \cu_m} \right|^q  \right)^\frac 1q \leq  \left\{
  \begin{array}{lcl}
    \O_s \left( C q^k \sqrt{m} \right) & \mbox{if } d=2, \\
     \O_s \left( C q^k\right) & \mbox{if } d\geq 3. \\
  \end{array}
\right.
\end{equation*}
With the same proof as in Step 1, we obtain, for each $q \geq 1$
\begin{equation*}
\E \left[ \left|3^{-dm} \sum_{y \in \cu_m}   \left[ \chi_p \right]_{\Pa}(0) - \left( \left[ \chi_p\right]_{\Pa} \right)_{y + x +\cu_m}\right|^q \right] \leq \left\{
  \begin{array}{lcl}
    C^q q^{qk} m^{\frac q2} & \mbox{if } d=2, \\
     C^q q^{qk} & \mbox{if } d\geq 3. \\
  \end{array}
\right.
\end{equation*}
But note that
\begin{equation*}
 \left[ \chi_p \right]_\Pa (0) - 3^{-2dm} \sum_{y \in \cu_m} \sum_{z \in x + y+ \cu_m} \left[ \chi_p \right]_\Pa (z)  =  3^{-dm}\sum_{y \in \cu_m}   \left[ \chi_p \right]_{\Pa}(0) - \left( \left[ \chi_p\right]_{\Pa} \right)_{y + x +\cu_m} .
\end{equation*}
Combining the two previous displays completes the proof of Step 3.

\medskip

\textit{Step 4.} In this step, we first split the integral,
\begin{align*}
\E \left[ \left| \left[ \chi_p \right]_\Pa (x) -\left[ \chi_p \right]_\Pa (0) \right|^q \right] & \leq 2^q \E \left[ \left| \left[ \chi_p \right]_\Pa (0) - 3^{-2dm} \sum_{y \in \cu_m} \sum_{z \in x + y+ \cu_m} \left[ \chi_p \right]_\Pa (z) \right|^q \right] \\ &
\qquad +  2^q \E \left[ \left| \left[ \chi_p \right]_\Pa (x) - 3^{-2dm} \sum_{y \in \cu_m} \sum_{z \in x + y+ \cu_m} \left[ \chi_p \right]_\Pa (z) \right|^q \right] .
\end{align*}
Combining the results of Step 2 and Step 3, we have, for $m \in \N$ chosen such that $3^m \leq |x| \leq 3^{m+1}$ and for each $q \geq 1$,
\begin{equation*}
\E \left[ \left| \left[ \chi_p \right]_\Pa (x) -\left[ \chi_p \right]_\Pa (0) \right|^q \right] \leq \left\{
  \begin{array}{lcl}
    C^q q^{qk} m^{\frac q2} & \mbox{if } d=2, \\
     C^q q^{qk} & \mbox{if } d\geq 3. \\
  \end{array}
\right.
\end{equation*}
Since $m \leq \frac{\log |x|}{\log 3}$, the proof of Step 3 is complete.

\medskip

\textit{Step 5.} First we extend the result of Step 4 to the case $0 < q  < 1$. By the Jensen inequality, we have, for each $ 0 < q  \leq 1$,
\begin{equation*}
\E \left[ \left| \left[ \chi_p \right]_\Pa (x) -\left[ \chi_p \right]_\Pa (0) \right|^q \right] \leq \E \left[ \left| \left[ \chi_p \right]_\Pa (x) -\left[ \chi_p \right]_\Pa (0) \right|^2\right]^{\frac q2} \leq \left\{
  \begin{array}{lcl}
    C^q \log^{\frac q2} |x| & \mbox{if } d=2, \\
     C^q & \mbox{if } d\geq 3. \\
  \end{array}
\right.
\end{equation*}

We now prove the main result of this step. We first deal with the case $d = 2$. Select an exponent $s > 0$ depending only on $d , \p , \lambda$ such that $s < \frac{1}{k}$, where $k$ is the exponent (depending only on $d , \p , \lambda$) which appears in~\eqref{est.step4.linfty}. We then compute
\begin{align*}
\E \left[\exp \left(  \left( \frac{ \left| \left[ \chi_p \right]_\Pa (x) -\left[ \chi_p \right]_\Pa (0) \right|}{\log^\frac 12 |x|} \right)^s \right) \right]  
		& = \sum_{l = 0}^{\infty} \frac{1}{\fact{l}} \E \left[ \frac{ \left| \left[ \chi_p \right]_\Pa (x) -\left[ \chi_p \right]_\Pa (0) \right|^{sl}}{\log^\frac{sl}{2} |x|} \right] \\
		& \leq  \sum_{l = 0 }^{\lfloor \frac 1s \rfloor}  \frac{C^{sl}}{\fact{l}} +  \sum_{l = \lceil\frac 1s \rceil}^{\infty} \frac{C^{sl} \left( sl \right)^{skl}}{\fact{l}} \\
		& < \infty,
\end{align*}
by the Stirling formula. We now set $\sigma := \min\left( \frac{\log 2}{\log \left(  \sum_{l = 0 }^{\lfloor \frac 1s \rfloor}  \frac{C^{sl}}{\fact{l}} +  \sum_{l = \lceil\frac 1s \rceil}^{\infty} \frac{C^{sl} \left( sl \right)^{skl}}{\fact{l}}\right) }, 1 \right) > 0$. Note that $\sigma$ depends only on $d , \p ,  \lambda $. With this value of $\sigma$, we have
\begin{equation*}
\E \left[\exp \left( \sigma \left( \frac{ \left| \left[ \chi_p \right]_\Pa (x) -\left[ \chi_p \right]_\Pa (0) \right|}{\log^\frac 12 |x|} \right)^s \right) \right] \leq 2.
\end{equation*}
From this computation, we obtain
\begin{equation*}
\left| \left[ \chi_p \right]_\Pa (x) -\left[ \chi_p \right]_\Pa (0) \right| \leq \O_s \left( \sigma^{-\frac 1s} \log^\frac 12 |x| \right).
\end{equation*}
Setting $C := \sigma^{-\frac 1s} $, we obtain~\eqref{est.step5.linfty}. The proof in dimension $d \geq 3$ follows the same lines and is even simpler since we do not have the square root of the logarithm.

\medskip

\textit{Step 6.} In this step, we remove the coarsening. To this end, we prove, for each point $y \in \Zd$,
\begin{equation*}
\left| \chi_p(y) - \left[ \chi_p \right]_\Pa (y) \right| \indc_{\{ y \in \C_\infty \}} \leq \O_s  (C).
\end{equation*}
Note that if a point $y$ belongs to the infinite cluster then there exists a path connecting $y$ to $\bar{z} \left( \cu_\Pa (y) \right)$ which lies in the cube $\cu_\Pa (y)$ and its neighbors. Consequently we have the estimate
\begin{equation*}
\left| \chi_p(y) - \left[ \chi_p \right]_\Pa (y) \right| \indc_{ \{ y \in \C_\infty \}} \leq \sum_{x \in \C_\infty \cap B(y , C  \size \left( \cu_\Pa (y) \right))} \left| \nabla \chi_p \right| (x).
\end{equation*}
Applying Proposition~\ref{p.regularitytheory} gives
\begin{equation} \label{eq:TV10432302}
\left| \chi_p(y) - \left[ \chi_p \right]_\Pa (y) \right| \indc_{\{ y \in \C_\infty \}} \leq \O_s  (C).
\end{equation}
We deduce that
\begin{multline*}
\left| \chi_p(x) - \chi_p(0) \right| \indc_{\{ 0,x \in \C_\infty\}} \\ \leq  \left| \chi_p(0) - \left[ \chi_p \right]_\Pa (0) \right| \indc_{\{ 0 \in \C_\infty \}} +  \left| \chi_p(x) - \left[ \chi_p \right]_\Pa (x) \right| \indc_{\{ x \in \C_\infty \}}  +  \left| \left[ \chi_p \right]_\Pa (x) -\left[ \chi_p \right]_\Pa (0) \right|.
\end{multline*}
Combining the result of Step 5 with the inequality~\eqref{eq:TV10432302} shows
\begin{equation*}
\left| \chi_p(x) - \chi_p(0) \right| \indc_{\{ 0,x \in \C_\infty\}}\leq \left\{  \begin{array}{lcl}
    \O_{s} \left(C  \log^\frac 12 |x| \right) & \mbox{if } d=2, \\
     \O_{s}  \left(C  \right) & \mbox{if } d\geq 3. \\
  \end{array}
\right.
\end{equation*}
The proof of Step 6 is complete.
\end{proof}

\appendix

\section{Proof of the $L^q$ multiscale Poincar\'e inequality} \label{appA}
In this appendix, we prove the $L^q$ multiscale Poincar\'e inequality stated in Proposition~\ref{multiscalepoincare}. Contrary to the rest of the article, we prove the result in the continuous setting. This choice is motivated by the two reasons listed below:
\begin{enumerate}
    \item The argument relies on the statement of Proposition D.1 and Remark D.6 of~\cite{armstrong2017quantitative}, which are stated in the continuous setting;
    \item We need to rescale equations and use results of elliptic regularity; it is thus easier to work in the continuous setting.
\end{enumerate}
The discrete version of the inequality stated in Proposition~\ref{multiscalepoincare} can be deduced from the continuous one by standard arguments.

We first introduce a few definitions pertaining to the continuous setting. We denote by $C_c^\infty (\Rd , \R)$ (resp.  $C^\infty (\Rd , \R)$) the set of smooth compactly supported (resp. smooth) functions in $\Rd$, by $\mathcal{S}(\Rd)$ the Schwartz space, i.e.,
\begin{equation*}
\mathcal{S}(\Rd) := \left\{ f \in C^\infty (\Rd , \R) : \forall \left( k , \alpha_1 , \cdots , \alpha_d\right) \in \N^{d+1}, \quad \sup_{x \in \Rd} |x|^k \left| \partial_1^{\alpha_1} \cdots \partial_d^{\alpha_d} f (x) \right| < \infty \right\}
\end{equation*} 
and by $\mathcal{S}'(\Rd)$ its topological dual, the space of tempered distribution. Given a domain $U \subseteq \Rd$ a domain, we denote by $C_c^\infty (U , \R)$ (resp.  $C^\infty (U , \R)$) the set of smooth compactly supported (resp. smooth) functions in $U$.

For $q\in [1,\infty)$, we denote the $L^q(U)$ and normalized $L^q(U)$ norms by
\begin{equation*}
\left\| w \right\|_{L^q(U)} := \left( \int_U \left|w(x)\right|^q\,dx\right)^{\frac1q} \mbox{ and } \left\| w \right\|_{\underline{L}^q(U)} := \left( \fint_U \left|w(x)\right|^q\,dx\right)^{\frac1q}.
\end{equation*}
For $k \in \N$, we denote by $W^{k,q}(U)$ the Sobolev space, by $W^{k,q}_0(U)$ the closure of $C_c^\infty (U, \R)$ in $W^{k,q}(U)$, and by $W^{k,q}_\mathrm{loc}(U)$ the space of local Sobolev functions. For $k \in \Z$ with $k < 0$, we denote by $W^{k,q}(U)$ the topological dual of $W^{-k,p}_0(U)$, with $p = \frac{q}{q-1}$.

\begin{proposition} [Multiscale Poincar\'e inequality, heat kernel version] 
For each radius $r>0$, we define the continuous heat kernel
\begin{equation} \label{def.heatker}
\tilde \Phi_{r^2}  : \left\{  \begin{array}{ccc}
\Rd & \to &\R \\
 x & \mapsto & r^{-d}\exp \left(- \frac{|x|^2}{r^2}\right). \\
\end{array} \right.
\end{equation}
For each exponent $q \geq 2$, there exists a constant $C:= C(d , q) < \infty$ such that for each tempered distribution $u \in W^{1,q}_{\mathrm{loc}} (\Rd) \cap \mathcal{S}'(\Rd)$ and each radius $R>0$,
\begin{equation} \label{heatmultscaleP}
\left\| u - (u)_{B_R}  \right\|_{\underline{L}^q\left(B_R\right)} \leq C  \left( \int_{\Rd} R^{-d}e^{ - \frac{|x|}{R} } \left( \int_{0}^{R} r \left| \tilde \Phi_{r^2} * \nabla u (x) \right|^2 \, dr  \right)^{\frac q2} \right)^\frac1q.
\end{equation}
Moreover the dependence on the $q$ variable of the constant $C$ can be estimated as follows, for each $q \geq 2$,
\begin{equation*}
C (d, q ) \leq A q^{\frac 32},
\end{equation*} 
for some constant $A := A(d) < \infty$.
\end{proposition}

Before starting the proof, we need to state the following proposition from~\cite[Proposition D.1 and Remark D.6]{armstrong2017quantitative} and to record a result from the elliptic regularity theory.
\begin{proposition}[Proposition D.1 and Remark D.6 of~\cite{armstrong2017quantitative}] \label{local.sob.est}
For each $q \geq 2$, there exists a constant $C := C(d , q ) < \infty$ such that for every tempered distribution $w \in \mathcal{S}'(\Rd)$,
\begin{equation} \label{eq:TV14151802}
\left\|w \right\|_{W^{-1,q}(B_1)} \leq C  \left( \int_{\Rd} e^{ - |x| } \left( \int_{0}^{1} r \left| \tilde \Phi_{r^2} * w (x) \right|^2 \, dr  \right)^{\frac q2} \right)^\frac1q.
\end{equation}
Moreover the constant $C$ satisfies, for each exponent $q \geq 2$,
$C (d, q) \leq A \sqrt{q},$ for some constant $A := A(d) < \infty$.
\end{proposition}

\begin{remark}
The statement of~\cite[Proposition D.1 and Remark D.6]{armstrong2017quantitative} is not identical to the statement of Proposition~\ref{local.sob.est}; one needs to perform the change of variables $r^2 := t$ to obtain the estimate~\eqref{eq:TV14151802} from the one of~\cite{armstrong2017quantitative}.
\end{remark}

\begin{remark}
The dependence on the $q$ variable of the constant $C$ is not explicit in~\cite{armstrong2017quantitative}. It can be recovered from the proof.
\end{remark}

We then record a result from the theory of elliptic regularity, it can be found in~\cite[Lemma 7.12 and Proposition 9.9]{gilbarg2015elliptic}.
\begin{proposition}[Lemma 7.12 and Proposition 9.9 of~\cite{gilbarg2015elliptic}] \label{elliptic.reg}
Let $V \subseteq \R^d$ be a bounded domain of~$\Rd$. Let $f \in L^p \left( V \right)$, $1<  p < \infty$, and let $w$ be the Newtonian potential of $f$, i.e.,
\begin{equation*}
w(x) := \int_{V} \Gamma (x-y) f(y) \, dy,
\end{equation*}
where $\Gamma$ is the fundamental solution of the Laplace equation, i.e.,
\begin{equation*}
\Gamma(x) :=  \left\{  \begin{array}{lcl}
   \frac{1}{2\pi} \log |x| & \mbox{if } d=2, \\
    \frac{1}{d(2-d) \omega_d} |x|^{2-d} & \mbox{if } d\geq 3, \\
  \end{array}
\right.
\end{equation*}
where $\omega_d$ is the volume of the unit sphere in $\Rd$. Then $w \in W^{2,p}(V), \Delta w = f$ a.e and
\begin{equation*}
\left\| \nabla^2 w\right\|_{L^p \left( V \right)} \leq C_0 \left\| f \right\|_{L^p \left( V \right)} 
\end{equation*}
and
\begin{equation*}
\left\|  w\right\|_{L^p \left( V \right)} + \left\| \nabla w\right\|_{L^p \left( V \right)} \leq C_1 \left\| f \right\|_{L^p \left( V \right)},
\end{equation*}
for some constants $C_1 :=  C_1(d, V) < \infty$ and $C_0:= C_0(d , p, V) < \infty$. Moreover, the dependence on $p$ of the constant $C_0$ can be explicited: 
\begin{equation*}
C_0(d , p, V) \leq A p, ~ \mbox{if}~ p \geq 2 \qquad \mbox{and} \qquad C_0(d , p, V) \leq A \frac1{p-1} ~\mbox{if} ~ 1 < p \leq 2,
\end{equation*}
for some $A:= A(d, V) < \infty$.
\end{proposition}

Before starting the proof, we mention that the dependence on the $p$ variable is not explicit in~\cite[Proposition 9.9]{gilbarg2015elliptic}; it can be recovered by keeping track of the constant $p$ in the application of the Marcinkiewicz interpolation theorem. Let us also mention that the case of the logarithmic potential is not considered in~\cite[Lemma 7.12]{gilbarg2015elliptic} (it is useful to obtain the estimate of the $L^p$ norm of $w$ in dimension 2). Nevertheless their proof is general enough to be applied in this setting.

\begin{proof}[Proof of Proposition~\ref{multiscalepoincare}]
Let $\psi \in C_c^\infty \left( B_{\frac 14} , \R \right)$ and $2 \leq q < \infty$. We denote by $p$ the conjugate exponent of $q$, i.e., $p := \frac{q}{q-1} \in (1 , 2].$ We split the proof into 5 steps.
\begin{itemize}
\item In Step 1, we show that there exists a constant $C:= C(d , \psi) < \infty$ such that, for each function $u \in W^{1,q} \left( B_1 \right),$
\begin{equation} \label{step1TVaf}
\left\| u - \psi * u \right\|_{W^{-1,q}\left(B_\frac 34\right)} \leq C  \left\| \nabla u \right\|_{W^{-1,q}(B_1)}.
\end{equation}
\item In Step 2, we prove that there exists a constant $C:= C(d , \psi) < \infty$ such that, for each function $u \in W^{1,q} \left( B_1 \right), $
\begin{equation} \label{ste2.multscale}
\left\| u - \psi * u(0) \right\|_{W^{-1,q}\left(B_\frac 34\right)} \leq C  \left\| \nabla u \right\|_{W^{-1,q}(B_1)}.
\end{equation}
\item In Step 3, we prove that there exists a constant $C:= C(d , q, \psi) < \infty$ such that, for each function $u \in W^{1,q} \left( B_1 \right), $
\begin{equation*}
\left\| u \right\|_{L^q\left(B_{\frac 12}\right)}  \leq C\left\|\nabla u \right\|_{W^{-1,q}\left(B_{\frac 34}\right)}  + C \left\|  u \right\|_{W^{-1,q}\left(B_\frac 34 \right)}
\end{equation*}
and that the constant $C$ satisfies $C(d , \psi , q) \leq A q $ for some $A := A(d, \psi) < \infty$.
\item In Step 4, we show that there exists a constant $C:= C(d , q, \psi) < \infty$ such that, for each function $u \in W^{1,q} \left( B_1 \right), $
\begin{equation*}
\left\| u- \left( u \right)_{B_{\frac 12}} \right\|_{L^q\left(B_{\frac 12}\right)}  \leq C\left\|\nabla u \right\|_{W^{-1,q}(B_1)}  
\end{equation*}
and that the constant $C$ satisfies $C(d , \psi , q) \leq A q $ for some constant $A := A(d, \psi) < \infty$.
\item In Step 5, we show that for each tempered distribution $u \in W^{1,q}_{\mathrm{loc}} (\Rd) \cap \mathcal{S}'(\Rd)$ and each radius $R>0$,
\begin{equation*}
\left\| u - (u)_{B_R}  \right\|_{\underline{L}^q\left(B_R\right)} \leq C  \left( \int_{\Rd} R^{-d}e^{ - \frac{|x|}{2R} } \left( \int_{0}^{2R} r \left| \tilde \Phi_{r^2} * \nabla u (x) \right|^2 \, dr  \right)^{\frac q2} \right)^\frac1q
\end{equation*}
and that the constant $C$ satisfies $C(d , q) \leq A q^{\frac 32} $ for some $A := A(d) < \infty$.
\end{itemize}

\medskip

\textit{Step 1.} We prove that there exists a constant $ C := C(d) < \infty$ such that
\begin{equation*}
\left\| u - u * \psi \right\|_{W^{-1,q}\left(B_\frac 34\right)}  \leq C  \left\| \nabla u \right\|_{W^{-1,q}(B_1)}.
\end{equation*}
Define, for each integer $n \in \N$,
\begin{equation*}
\psi_n := 2^{-dn} \psi \left( 2^n \cdot \right).
\end{equation*}
Since $\psi_n * u \rightarrow u$ in $L^q\left(B_\frac 34\right)$, we can use the triangle inequality to bound
\begin{equation} \label{littpaley.est}
\left\| u - \psi * u \right\|_{W^{-1,q}\left(B_\frac 34\right)}  \leq \sum_{n= 0}^\infty \left\| \psi_{n+1}  * u - \psi_n * u \right\|_{W^{-1,q}\left(B_\frac 34\right)} .
\end{equation}
Since the function $\psi_1 - \psi_0$ is compactly supported in $B_\frac14$ and of mean $0$, we can apply~\cite[Lemma 5.7]{armstrong2017quantitative}, to show that there exists a function $\Psi \in C_c^\infty \left(B_\frac14, \Rd \right)$ satisfying
\begin{equation*}
\nabla \cdot \Psi = \psi_1 - \psi_0 .
\end{equation*}
For each integer $n \in \N$, we denote
\begin{equation*}
\Psi_n := 2^{-dn} \Psi \left( 2^n \cdot \right),
\end{equation*}
by scaling invariance we also have
\begin{equation*}
2^{-n}\nabla \cdot \Psi_n = \psi_{n+1} - \psi_n.
\end{equation*}
For each function $g \in W^{1,p}_0\left(B_\frac 34\right)$, we have
\begin{align*}
\int_{\left(B_\frac 34\right)} \left(\psi_{n+1}  - \psi_n \right) * u(x) g(x) \, dx & = 2^{-n}\int_{\Rd}\int_{\Rd}  \nabla \cdot \Psi_n (x-y) u(y) g(x) \, dx dy \\
																& = 2^{-n} \int_{\Rd} \nabla u(y) \cdot \left( \int_{\Rd}  \Psi_n (x-y) g(x) \, dx \right) dy.
\end{align*}
By construction, the function $y \rightarrow \left( \int_{\Rd}  \Psi_n (x-y) g(x) \, dx \right)$ is supported in $B_1$, we can thus estimate
\begin{equation*}
\left| \int_{B_\frac 34} \left(\psi_{n+1}  - \psi_n \right) * u(x) g(x) \, dx \right| \leq 2^{-n} \left\| \int_{\Rd}  \Psi_n (x-\cdot) g(x) \, dx   \right\|_{W^{1,p}_0\left(B_1\right)}  \left\| \nabla u \right\|_{W^{-1,q}(B_1)}.
\end{equation*}
Moreover, one can check that there exists a constant $C:= C(d , \psi) < \infty$ such that
\begin{equation*}
 \left\|  \int_{\Rd}  \Psi_n (x-\cdot) g(x) \, dx  \right\|_{W^{1,p}_0\left(B_1\right)}  \leq C \left\| g\right\|_{W^{1,p}_0\left(B_1\right)} = C \left\| g\right\|_{W^{1,p}_0\left(B_\frac 34\right)}.
\end{equation*}
Taking the supremum over the functions $g \in W^{1,p}_0\left(B_\frac 34\right)$ of norm less than $1$ and combining this result with~\eqref{littpaley.est}, we obtain
\begin{equation*}
\left\| u - \psi * u \right\|_{W^{-1,q}\left(B_\frac 34\right)} \leq C  \left\| \nabla u \right\|_{W^{-1,q}(B_1)},
\end{equation*}
for some constant $C : = C(d) < \infty$. The proof of Step 1 is complete.

\medskip

\textit{Step 2.} We split the norm
\begin{equation} \label{firststep2multscale}
\left\| u - \psi * u (0) \right\|_{W^{-1,q}\left(B_\frac 34\right)}  \leq \left\| u - \psi * u  \right\|_{W^{-1,q}\left(B_\frac 34\right)}  + \left\| \psi * u  -\psi * u(0) \right\|_{W^{-1,q}\left(B_\frac 34\right)} .
\end{equation} 
But note that, for each point $x \in B_\frac 34$, 
\begin{equation}\label{est.ab}
\left| \psi * u (x) - \psi * u (0)\right|  \leq C \left\| \nabla u \right\|_{W^{-1,q}(B_1)}.
\end{equation}
The proof of this estimate is similar to the previous step. By~\cite[Lemma 5.7]{armstrong2017quantitative}, we represent $\psi (\cdot - x) - \psi$ in the form
\begin{equation*}
\nabla \cdot \Psi_x = \psi (\cdot - x) - \psi
\end{equation*}
where $\Psi_x \in C_c^\infty \left( B_1 , \R \right)$ is bounded in $W^{1,p}_0(B_1)$ uniformly in the point $x \in B_{\frac 34}$. We then prove~\eqref{est.ab} thanks to an integration by parts. From this argument, we deduce that
\begin{equation*}
 \left\| \psi * u  -\psi * u(0) \right\|_{W^{-1,q}\left(B_\frac 34\right)}  \leq C \left\| \nabla u \right\|_{W^{-1,q}(B_1)}.
\end{equation*}
Combining this estimate with~\eqref{firststep2multscale} and the estimate~\eqref{step1TVaf} completes the proof of Step 2.

\medskip

\textit{Step 3.} 
Let  $\eta \in C_c^\infty \left(B_1\right)$ be a cutoff function satisfying
\begin{equation*}
\indc_{B_{\frac 12}} \leq \eta \leq \indc_{B_{\frac 34}}, \qquad \mbox{and} \qquad \left|  \nabla^2 \eta \right| +  \left|  \nabla \eta \right| \leq C.
\end{equation*}
For any function $f \in L^p\left( B_1 \right)$, we denote by $w_f$ the Newtonian potential of $f$ introduced in Proposition~\ref{elliptic.reg} with the open set $V = B_1$. We then compute
\begin{align*}
\int_{B_1} \eta(x) u(x) f(x) \, dx & = \int_{B_1} \eta(x) u(x) \Delta w_f (x) \, dx \\
							& = \int_{B_1}\nabla \eta (x)  u(x) \nabla w_f (x)  + \eta (x)\nabla u (x) \nabla w_f (x) \, dx \\
							& \leq \left\| u \right\|_{W^{-1,q}\left(B_{\frac 34}\right)} \left\| \nabla\eta \nabla w_f \right\|_{W^{1,p}_0\left(B_{\frac 34}\right)} +  \left\| \nabla u \right\|_{W^{-1,q}\left(B_{\frac 34}\right)}  \left\| \eta \nabla w_f \right\|_{W^{1,p}_0\left(B_{\frac 34}\right)} .
\end{align*}
By the properties of the function $\eta$ and by Proposition~\ref{elliptic.reg}, we have
\begin{equation*}
 \left\| \nabla\eta \nabla w_f \right\|_{W^{1,p}_0\left(B_{\frac 34}\right)} + \left\| \eta \nabla w_f \right\|_{W^{1,p}_0\left(B_{\frac 34}\right)} =  \left\| \nabla\eta \nabla w_f \right\|_{W^{1,p}_0(B_1)} + \left\| \eta \nabla w_f \right\|_{W^{1,p}_0(B_1)} \leq C \left\| f \right\|_{L^p\left( B_1 \right)},
\end{equation*}
for some constant $C:= C(d ,p, \eta ) < \infty$ satisfying
\begin{equation*}
C(d , p, \eta) \leq A \frac1{p-1},
\end{equation*}
with $A:= A(d, \eta) < \infty$. Consequently
\begin{align*}
\left\| u \right\|_{L^q\left(B_{\frac 12}\right)}  \leq \left\| \eta u \right\|_{L^q\left(B_{1}\right)} 
										& = \sup_{f \in L^p\left( B_1 \right) , \left\| f \right\|_{L^p\left( B_1 \right)} = 1} \int_{B_1} \eta(x) u(x) f(x) \, dx \\
										& \leq C \left( \left\| u \right\|_{W^{-1,q}\left(B_\frac 34\right)} + \left\| \nabla u \right\|_{W^{-1,q}\left(B_{\frac 34}\right)} \right).
\end{align*}
The proof of Step 3 is complete.

\medskip

\textit{Step 4.} Applying the main result of Step 3 to the function $u - \psi * u(0)$, we have
\begin{equation*}
\left\| u - \psi * u(0) \right\|_{L^q\left(B_{\frac 12}\right)}  \leq C \left( \left\| u - \psi * u(0) \right\|_{W^{-1,q}\left(B_\frac 34\right)} + \left\| \nabla u \right\|_{W^{-1,q}(B_1)} \right).
\end{equation*}
Then by Step 2, we obtain
\begin{equation*}
\left\| u - \psi * u(0) \right\|_{L^q\left(B_{\frac 12}\right)}  \leq C \left\| \nabla u \right\|_{W^{-1,q}(B_1)} .
\end{equation*}
But we have, for each $a \in \R$
\begin{equation*}
\left\| u- \left( u \right)_{B_{\frac 12}} \right\|_{L^q\left(B_{\frac 12}\right)} \leq 2 \left\| u-a \right\|_{L^q\left(B_{\frac 12}\right)} .
\end{equation*}
Thus
\begin{equation*}
\left\| u- \left( u \right)_{B_{\frac 12}} \right\|_{L^q\left(B_{\frac 12}\right)}  \leq 2 \inf_{a \in \R}\left\| u-a \right\|_{L^q\left(B_{\frac 12}\right)} \leq 2 \left\| u - \psi * u(0) \right\|_{L^q\left(B_{\frac 12}\right)} .
\end{equation*}
Combining the previous displays completes the proof of Step 4.

\medskip

\textit{Step 5.} Applying the result of Step 4 and Proposition~\ref{local.sob.est}, we obtain, for each $q \geq 2$ and each $u \in \mathcal{S}' \left( \Rd \right) \cap W^{1,q}_{\mathrm{loc}} \left( \Rd \right)$,
\begin{equation*}
\left\| u- \left( u \right)_{B_{\frac 12}} \right\|_{L^q\left(B_{\frac 12}\right)}  \leq C  \left( \int_{\Rd} e^{ - |x| } \left( \int_{0}^{1} r \left| \tilde \Phi_{r^2} * \nabla u (x) \right|^2 \, dr  \right)^{\frac q2} \right)^\frac1q,
\end{equation*}
for some constant $C:= C(d , q)$ satisfying $C (d, q) \leq A q^\frac 32$. Rescaling the previous estimates eventually shows
\begin{equation*}
\left\| u - (u)_{B_R}  \right\|_{\underline{L}^q\left(B_R\right)} \leq C  \left( \int_{\Rd} R^{-d}e^{ - \frac{|x|}{2R} } \left( \int_{0}^{2R} r \left| \tilde \Phi_{r^2} * \nabla u (x) \right|^2 \, dr  \right)^{\frac q2} \right)^\frac1q.
\end{equation*}
and the proof of Proposition~\ref{multiscalepoincare} is complete.
\end{proof}

\section{Elliptic inequalities on the supercritical percolation cluster}\label{appC} In this section, we record some simple elliptic inequalities, the Caccioppoli inequality and the Meyers estimate. These inequalities were written in~\cite{AD2} for harmonic functions. In our context, we need to apply these results when the right-hand term is not $0$ but the divergence of a vector field.

\begin{proposition}[{Caccioppoli inequality}]
\label{l.caccioppoli}
Assume that we are given a function $u : \C_\infty \rightarrow \R$ and a vector field $\xi : E_d \rightarrow \R$ satisfying the following condition
\begin{equation} \label{vectorfieldcluster}
\xi (x,y) = 0 \mbox{ if} ~ \a(x,y) = 0 \mbox{ or } x,y \notin \C_\infty.
\end{equation}
In particular, gradients of functions defined on the infinite cluster satisfy this condition by~\eqref{defvectorfieldcluster}. Assume that $u$ is solution of the equation,
\begin{equation*}
- \nabla \cdot \a \nabla u  = - \nabla \cdot \xi \mbox{ in } \C_\infty.
\end{equation*}
Select two bounded sets  $U, V \subseteq \Zd$ such that $V \subseteq U$ and $\dist(V,\partial U) \geq r\geq 1$. 
Then there exists a constant $C(\lambda)<\infty$ such that 
\begin{equation}
\label{e.caccioppoli}
\int_{\C_\infty \cap V} \left| \nabla u\right|^2(x)\,dx \leq \frac{C}{r^2} \int_{\C_\infty \cap \left( U \setminus V \right)} \left| u(x) \right|^2\,dx +C \int_{\C_\infty \cap U }  |\xi|^2(x) \, dx.
\end{equation}
\end{proposition}

\begin{proof}
The strategy of the proof follows the standard technique to prove the Caccioppoli inequality, we select a cutoff function $\eta : \Zd \rightarrow \R$ satisfying 
\begin{equation}
\label{e.etaprops}
\indc_{V} \leq \eta \leq 1, ~ \eta \equiv 0 \ \mbox{on} \ \Rd \setminus U,  \mbox{ and }  \forall x,y \in \Zd \mbox{ such that } x \sim y, ~ \left| \eta(x) - \eta(y) \right|^2 \leq \frac{C\left( \eta(x)+ \eta(y) \right)}{r^2},
\end{equation}
test the equation satisfied by $u$ with the function $\eta u$ and perform standard computations.
\end{proof}

The second elliptic estimate needed in this article is the Meyers estimate. This estimate was also proved in~\cite{AD2} in the case of $\a$-harmonic functions.

\begin{repproposition}{meyercorrector}[Meyers estimate]
There exist a constant $C := C(d,\lambda,\p) < \infty$, two exponents $s := s(d, \lambda, \p) > 0$ and $\ep := \ep(d, \lambda, \p) >0$ and a random variable $\M_{\mathrm{Meyers}} \leq \O_s(C)$ such that for each $m\in \N $ with $3^m \geq \M_{\mathrm{Meyers}}$, and each function $v : \C_\infty \rightarrow \R$ satisfying
\begin{equation*}
- \nabla \cdot \a \nabla v = - \nabla \cdot \xi \mbox{ in } \C_\infty,
\end{equation*}
for some vector field $\xi : E_d \rightarrow \R$ satisfying~\eqref{vectorfieldcluster}, the following estimate holds,
\begin{multline} \label{repeqmeyercorrector}
\left( \frac{1}{|\cu_m|} \int_{\cu_m \cap \C_\infty} \left| \nabla v \right|^{2+\ep}(x)\,dx  \right)^{\frac1{2+\ep}} \\
\leq C \left( \frac1{\left| \frac 43 \cu_m \right|}\int_{\frac 43 \cu_m\cap \C_\infty} \left|\nabla v \right|^{2}(x)\,dx \right)^{\frac1{2}} + C  \left( \frac1{\left| \frac 43 \cu_m \right|} \int_{\frac 43\cu_m \cap \C_\infty} \left| \xi \right|^{2+\ep}(x)\,dx \right)^{\frac1{2+\ep}}.
\end{multline}
\end{repproposition}

\begin{proof}[Proof of Proposition~\ref{meyercorrector}]
The results of Proposition 3.8 and Definition 3.9 of~\cite{AD2} can be adapted in our context to prove~\eqref{repeqmeyercorrector}. The Meyers estimate is a consequence of the three following ingredients: the Caccioppoli inequality, the Sobolev inequality and the Gehring's lemma. Proposition~\ref{l.caccioppoli} provides a version of the Caccioppoli inequality well-suited to deal with a divergence form right-hand side. The Sobolev inequality is valid for any functions. The usual version of the Gehring's Lemma (see for instance Theorem 6.6 \& Corollary 6.1 of~\cite{Giu}), is general enough to be applied in this context.

\end{proof}

\section{Proof of Lemma~\ref{minscalecorrector}} \label{appB} 
The objective of this appendix is to prove Lemma~\ref{minscalecorrector} which is restated below.

\begin{replemma}{minscalecorrector}[Minimal scale]
\begin{sloppypar}There exist a constant $C := C(d, \p ,\lambda) < \infty$, an exponent $s := s(d, \p ,\lambda) > 0$ and a random variable $\M_1 \leq \O_s' (C)$ such that for each integer $m \in \N$ satisfying $3^m \geq \M_1$,\end{sloppypar}
\begin{equation*} \label{eqminscalecorrectorapp}
3^{-dm} \sum_{z \in \cu_m}\size \left(\cu_\Pa (z) \right)^{\frac{3d(2+\ep)}{\ep}}  \left( \sum_{x \in \Zd, \, \dist \left(\cu_\Pa(x) , \cu_\Pa (z)\right) \leq 1, \, e \in \mathcal{B}_d^x}\left( 1 + \X^e(x) \right)^d \right)^{\frac{2+\ep}{\ep}} \leq C,
\end{equation*}
where $\ep := \ep(d, \p , \lambda) > 0$ is the exponent which appears in Proposition~\ref{meyercorrector}.
\end{replemma}

\begin{proof}[Proof of Lemma~\ref{minscalecorrector}]
First, notice that one can rewrite
\begin{align*}
\lefteqn{ 3^{-dm} \sum_{z \in \cu_m}\size \left(\cu_\Pa (z) \right)^{\frac{3d(2+\ep)}{\ep}}  \left( \sum_{x \in \Zd, \, \dist \left(\cu_\Pa(x) , \cu_\Pa (z)\right) \leq 1, \, e \in \mathcal{B}_d^x}\left( 1 + \X^e(x) \right)^d \right)^{\frac{2+\ep}{\ep}} } \qquad & \\ &
\leq C 3^{-dm} \sum_{z \in \cu_m}\size \left(\cu_\Pa (z) \right)^{\frac{3d(2+\ep)+2}{\ep}}  \sum_{x \in \Zd, \, \dist \left(\cu_\Pa(x) , \cu_\Pa (z)\right) \leq 1, \, e \in \mathcal{B}_d^x} \left( 1 + \X^e(x) \right)^{d \frac{2+\ep}{\ep}} \\
& \leq C 3^{-dm}  \sum_{x \in \Zd, \, \dist \left(\cu_\Pa(x) , \cu_m\right) \leq 1, \, e \in \mathcal{B}_d^x}\size \left(\cu_\Pa (x) \right)^{\frac{3d(2+\ep)+2}{\ep}} \left( 1 + \X^e(x) \right)^{d \frac{2+\ep}{\ep}}.
\end{align*}
By Property~(iv) of Proposition~\ref{p.partitions} applied with the exponent $t = \frac{6d(2+\ep)+4}{\ep}$, it is clear that for each integer $m \in \N$ satisfying $3^m \geq \M_t(\Pa)$, we have:
\begin{enumerate}
\item The inequality $\sup_{x \in \cu_{m+1}} \size \left( \cu_\Pa (x) \right) \leq 3^{\frac{dm}{d+t}},$ which implies
\begin{equation*}
\left\lbrace x \in \Zd , \, \dist \left(\cu_\Pa(x), \cu_m\right)\leq 1 \right\rbrace \subseteq \cu_{m+1};
\end{equation*}
\item The estimate
\begin{align*}
 \left( 3^{-dm} \sum_{x \in \Zd , \, \dist \left(\cu_\Pa(x), \cu_m\right)\leq 1}\size \left(\cu_\Pa (x) \right)^{\frac{6d(2+\ep)+4}{\ep}} \right)^\frac 12 &
\leq C \left( 3^{-d(m+1)} \sum_{x \in \cu_{m+1}}\size \left(\cu_\Pa (x) \right)^{\frac{6d(2+\ep)+4}{\ep}} \right)^\frac 12 \\ & \leq C.
\end{align*}
\end{enumerate}
Thus by the Cauchy-Schwarz inequality, it is enough to prove that there exists a random variable $\M$ satisfying $\M \leq \O_s'(C)$, such that for each integer $m \in \N$ satisfying $3^m \geq \M$,
\begin{equation} \label{minscalebrut1}
3^{-dm} \sum_{x \in \cu_m, \, e \in \mathcal{B}_d^x}  \left( \X^e(x)\right)^{\frac{d(4+2\ep)}{\ep}} \leq C.
\end{equation}
Unfortunately, we cannot prove this exact statement; we will prove a slightly weaker estimate, Lemma~\ref{minscaleforminscale}, which is still strong enough to deduce Proposition~\ref{boundspatialaveragecorrector}. Define, for each constant $C > 0$, the random variable
\begin{multline*}
\X_{C} := \inf  \left\lbrace r \in [1,\infty) \, : \, \forall r',R' \in [r,\infty), \mbox{ with } r' \leq R', \, \forall u\in \A(\C_{\infty} \cap B_{R'}) \right. \\ \left. 
 \left\| \nabla u   \right\|_{\underline{L}^2(\C_{\infty} \cap B_{r'})} \leq C   \frac{r'}{R'} \left\|  \nabla u \right\|_{\underline{L}^2(\C_{\infty}  \cap B_{R'})} \right\rbrace,
\end{multline*}
and, for each point $x \in \Zd$,
\begin{equation*}
\X_{C} (x) :=  \X_{C} \circ \tau_x.
\end{equation*}
Denote by $C_0 := C_0 (d,\p,\lambda) <\infty$ the constant appearing in Proposition~\ref{p.regularitytheory}. By definition we have
\begin{equation*}
\X_{C_0}= \X.
\end{equation*}
Note also that the map $C \mapsto \X_C$ is non-increasing. We have the following lemma.
\begin{lemma} \label{minscaleforminscale}
For every integrability parameter $t > 0$, there exist a constant $C(d,\p,\lambda,t) < \infty$, an exponent $s(d,\p,\lambda,t)> 0$ and a random variable $\M_t^\X$ satisfying
\begin{equation*}
\M_t^{\X} \leq \O_s'(C)
\end{equation*}
such that for every integer $m \in \N$ satisfying $3^m \geq \M_t^{\X}$, one has the inequality
\begin{equation*}
3^{-dm} \sum_{x \in \cu_m, \, e \in \mathcal{B}_d^x} \left| \X_{C_0^2}^e(x)\right|^t \leq C.
\end{equation*}
\end{lemma}

\begin{remark}
\begin{enumerate}
\item This statement is weaker than~\eqref{minscalebrut1} since, for each $x \in \Zd$ and $e \in \mathcal{B}_d^x$,
\begin{equation*}
\X_{C_0^2}^e(x) \leq \X_{C_0}^e(x) = \X^e(x).
\end{equation*}
Nevertheless it is enough to prove Result 2, since we only need to replace $C_0$ by $C_0^2$ in every computation involving the estimates on the random variables $\chi_p^e(x)$; the result remains unchanged.
\item Applying this result with $t= \frac{d (4 + 2 \ep)}{\ep}$ completes the proof of Lemma~\ref{minscalecorrector}.
\end{enumerate}
\end{remark}

\end{proof} 
There remains to prove Lemma~\ref{minscaleforminscale}. Before starting the proof, we introduce a few ingredients and preliminary results. First define, for each radius $R > 0$, and each constant $C >0$,
\begin{multline} \label{definitionlocalisedminscale.def}
\X_{R,C} := \inf  \left\lbrace r \in [1,R] \, : \, \forall r',R' \in [r,R], \mbox{ with } r' \leq R', \, \forall u\in \A(\C_{\mathrm{max}}(B_{R}) \cap B_{R'}) \right. \\ \left. 
 \left\| \nabla u   \right\|_{\underline{L}^2(\C_{\mathrm{max}}(B_{R}) \cap B_{r'})} \leq C   \frac{r'}{R'} \left\|  \nabla u \right\|_{\underline{L}^2(\C_{\mathrm{max}}(B_{R})  \cap B_{R'})} \right\rbrace,
\end{multline}
where $\C_{\mathrm{max}}(B_{R})$ denotes the largest cluster contained in $B_{R}$; if there is more than one candidate, we break ties using a deterministic procedure. Similarly we define, for each point $x \in \Zd$,
\begin{equation*}
\X_{C} (x) :=  \X_{C} \circ \tau_x.
\end{equation*}

This random variable is defined on the enlarged probability space $\Omega \times \Omega$ and is measurable with respect to the $\sigma$-algebra $\F(x +B_R) \otimes \{ \emptyset, \Omega \}$ (it depends on the edges in the ball $x+ B_R$ of the first variable and does not depend on the edges of the second variable). Consequently the random variables $\X_{R,C}(x)$ and $\X_{R,C}(y)$ are independent if $|x-y| > 2R$.

Note also that $\X_{R,C}$ is decreasing in the $C$ variable and, for $R \geq \M_t(\Pa)$, it is increasing in the $R$ variable. We thus denote by, for each $C \geq 1$

\begin{equation*}
\X_{C} := \lim_{R \rightarrow \infty} \X_{R,C} = \limsup_{R \geq 1}  \X_{R,C} \in [1,\infty].
\end{equation*}

By Proposition~\ref{p.regularitytheory}, we know that there exists a constant $C_0 := C_0(d,\p,\lambda) < \infty$ such that
\begin{equation} \label{minscalenewnotation}
\X_{C_0}= \X \leq \O_s'(C). 
\end{equation}
thus, for each radius $R \geq 1$,
\begin{equation*}
\X_{R,C_0} \indc_{\{ R \geq  \M_t(\Pa)\}} \leq \X_{C_0} \leq \O_s'(C). 
\end{equation*}
Moreover, for each radius $R \geq 1$, we have
\begin{equation*}
\X_{R,C_0} \indc_{\{ R \in [1, \M_t(\Pa)] \}} \leq \M_t(\Pa) \leq \O_s'(C). 
\end{equation*}
Combining the two previous displays yields, for each $R \geq 1$,
\begin{equation*}
\X_{R,C_0} \leq \O_s'(C).
\end{equation*}
We now prove the following inequality, for each $R,C > 1$,
\begin{equation} \label{estimatelocalXX}
\X_{C^2}\leq \X_{R,C} + R \indc_{\{ R \leq \M_t(\Pa)\}} + \X_{C} \indc_{\{  \X_{C} > R \}}.
\end{equation}
We split the proof of this inequality into two cases.

\smallskip

\textit{Case 1.} If $\X_C > R$, then since $C \geq 1$ and the map $C \mapsto \X_C$ is decreasing, the inequality ~\eqref{estimatelocalXX} is a consequence of the estimate
\begin{equation*}
\X_{C^2}\leq  \X_{C} \leq  \X_{R,C}+  R \indc_{\{ R \leq \M_t(\Pa)\}} + \X_{C} \indc_{\{  \X_{C} > R \}}.
\end{equation*}

\textit{Case 2.} If $\X_C \leq R$ and $R \leq \M_t(\Pa)$, then
\begin{equation*}
\X_{C^2} \leq \leq \X_C  R \indc_{\{ R \leq \M_t(\Pa)\}} \leq \X_{R,C} + R \indc_{\{ R \leq \M_t(\Pa)\}} + \X_{C} \indc_{\{  \X_{C} > R \}}.
\end{equation*}

\textit{Case 3.} If $\X_C \leq R$ and $R \geq \M_t(\Pa)$ then $\C_{\mathrm{max}}(B_R)$ is equal to the maximal connected component of $\C_\infty \cap B_R$ and we have, for each pair of radii $r,R' > R$ satisfying $R' \geq r$
\begin{equation*}
 \left\| \nabla u  \right\|_{\underline{L}^2(\C_\infty\cap B_r)} \leq C   \frac{r}{R'} \left\|  \nabla u \right\|_{\underline{L}^2(\C_\infty\cap B_{R'})}.
\end{equation*}
Moreover, for each pair of radii $r,R' \in [\X_{R,C}, R]$ with $R' \geq r$, we have
\begin{equation*}
\left\| \nabla u  \right\|_{\underline{L}^2(\C_\infty\cap B_r)} \leq C   \frac{r}{R'} \left\|  \nabla u \right\|_{\underline{L}^2(\C_\infty\cap B_{R'})}.
\end{equation*}
Combining the two previous displays and using $C^2 \geq C$ yields for each pair of radii $r,R' \geq \X_{R,C}$ with $R' \geq r$,
\begin{equation*}
\left\| \nabla u \right\|_{\underline{L}^2(\C_\infty\cap B_r)} \leq C^2   \frac{r}{R'} \left\|  \nabla u \right\|_{\underline{L}^2(\C_\infty\cap B_{R'})}
\end{equation*}
and thus by definition of the random variable $\X_{C^2}$,
\begin{equation*}
\X_{C^2} \leq \X_{R,C}
\end{equation*}
and the proof of the inequality~\eqref{estimatelocalXX} is complete.

For $x \in \Zd, e = \{ x,y \} \in \mathcal{B}_d, C,R \in [1,\infty)$, denote by $\X_{R,C}^e(x)$ the translated and resampled random variable
\begin{multline*}
\X_{R,C}^e(x) := \inf  \left\lbrace r \in [1,R] \, : \,\mbox{such that } \forall 1 \leq  r' \leq R' \leq R,\, u\in \A^e(\C_{\mathrm{max}}^e(B_R(x)) \cap B_{R'}(x)) \right. \\ \left. 
 \left\| \nabla u  \right\|_{\underline{L}^2(\C_{\mathrm{max}}^e(B_R) \cap B_{r'}(x))} \leq C   \frac{r'}{R'} \left\|  \nabla u\right\|_{\underline{L}^2\left(\C_{\mathrm{max}}^e(B_R(x)) \cap B_{R'}(x) \right)} \right\rbrace.
\end{multline*}
We also define, for each point $x \in \Zd$
\begin{equation*}
\X_{C}^e(x) := \lim_{R \rightarrow \infty} \X_{R,C}^e(x) = \limsup_{R \geq 1}  \X_{R,C}^e(x) \in [1,\infty].
\end{equation*}

The second ingredient in the proof of Lemma~\ref{minscaleforminscale} is the following minimal scale lemma. It is an adaptation of~\cite[Lemma 2.3]{AD2} and will be used in the proof of Lemma~\ref{minscaleforminscale}.

\begin{lemma} \label{lemmafromperc.0}
Fix $K\geq 1$, $s>0$ and $\beta>0$ and suppose that $\{ X_n \}_{n\in\N}$ is a sequence of random variables satisfying, for every $n\in\N$,
\begin{equation*} 
X_n \leq K + \O_s\left( K3^{-n\beta} \right). 
\end{equation*}
Then there exists $C(s,\beta,K) <\infty$ such that the random scale
\begin{equation*}
M := \sup \left\{ 3^n\in\N \,:\, X_n \geq K + 1 \right\}
\end{equation*}
satisfies the estimate
\begin{equation*} 
M \leq \O_{s\beta}(C). 
\end{equation*}
\end{lemma}

\begin{proof}
This result can be deduced by applying~\cite[Lemma 2.3]{AD2} to the sequence of random variables $X_n' = \max \left( X_n - K, 0 \right). $
\end{proof}

We now turn to the proof of Lemma~\ref{minscaleforminscale}.

\begin{proof}[Proof of Lemma~\ref{minscaleforminscale}]
Fix an exponent $t \in (0, \infty)$ and $m,n \in \N$ with $m > n$. Using the inequality~\eqref{estimatelocalXX}, we have
\begin{align} \label{e.Lambdatbounds}
\lefteqn{3^{-dm} \sum_{x \in \cu_m, \, e \in \mathcal{B}_d^x} \left| \X_{C_0^2}^e(x)\right|^t} \qquad & \\ & 
  \leq  C 3^{-dm} \sum_{x \in \cu_m, \, e \in \mathcal{B}_d^x} \left| \X_{3^n, C_0}^e(x)\right|^t +  C  3^{-dm} \sum_{x \in \cu_m, \, e \in \mathcal{B}_d^x} \left| \X_{C_0}^e(x) \right|^t \indc_{\left\lbrace \X_{C_0}^e(x) > 3^n \right\rbrace}   \notag \\ &
\qquad + C 3^{-dm} \sum_{x \in \cu_m}    3^{tn} \indc_{\{ 3^n \leq \M_t\left(\Pa^e\right)\}} \circ \tau_x . \notag
\end{align}
Since $\X_{C_0}^e(x) \leq \O_s'(C)$, for every $t,t' > 0, $ there exist an exponent $s'(d,\p,\lambda,t,t') > 0$ and a constant $C'(d,\p,\lambda,t,t') < \infty$ such that
\begin{equation*}
3^{-dm} \sum_{x \in \cu_m, \, e \in \mathcal{B}_d^x} \left|  \X_{C_0}^e(x) \right|^t \indc_{\left\lbrace \X_{C_0}^e(x) > 3^n \right\rbrace}  \leq \O_{s'}'(C' 3^{-nt'} )
\end{equation*}
and
\begin{equation*}
3^{-dm} \sum_{x \in \cu_m, \, e \in \mathcal{B}_d^x}  3^{nt} \indc_{\{ 3^n \geq \M_t(\Pa)\}} \circ \tau_x \leq \O_{s'}'(C' 3^{-nt'} ).
\end{equation*}
Combining the previous displays yields
\begin{equation*}
3^{-dm} \sum_{x \in \cu_m, \, e \in \mathcal{B}_d^x} \left|  \X_{C_0^2}^e(x)\right|^t \leq C 3^{-dm} \sum_{x \in \cu_m, \, e \in \mathcal{B}_d^x} \left|  \X_{3^n, C_0}^e(x)\right|^t  +  \O_{s'}'(C' 3^{-nt'} ).
\end{equation*}
Moreover, notice that by definition of the localized random variable $\X_{3^n, C_0}^e(x)$, we have for each $x \in \Zd$
\begin{equation*}
  \sum_{e \in \mathcal{B}_d^x} \left| \X_{3^n, C_0}^e(x)\right|^t \leq 2d \times 3^{nt}.
\end{equation*}
The proof of the lemma is then the same as the proof of Steps 1 and 2 of~\cite[Proposition 2.1]{AD2} with the random variable $3^{-dm} \sum_{x \in \cu_m, \, e \in \mathcal{B}_d^x} \left| \X_{C_0^2}^e(x)\right|^t$ instead of $\Lambda_t(z + \cu_m, \S)$ and the random variable $3^{-dm} \sum_{x \in z + \cu_n, \, e \in \mathcal{B}_d^x} \left| \X_{3^n, C_0^2}^e(x)\right|^t$ instead of $\Lambda_t(z' + \cu_n, \S_\mathrm{loc}(z'))$. We rewrite it for completeness.

\smallskip

We denote by
\begin{equation*}
Z := 3^{-dm} \sum_{x \in \cu_m, \, e \in \mathcal{B}_d^x} \left| \X_{3^n, C_0}^e(x)\right|^t  = \frac{|\cu_n|}{|\cu_m|} \sum_{z \in 3^n \Zd \cap \cu_m}  3^{-dm} \sum_{x \in z +\cu_n, \, e \in \mathcal{B}_d^x} \left| \X_{3^n, C_0}^e(x)\right|^t .
\end{equation*}
We first prove that there exists a constant $C := C(d,\p, \lambda,t) < \infty$ such that
\begin{equation} \label{e.Zbound}
Z \leq C + \O_1' \left(3^{nt-d(m-n)} \right). 
\end{equation}
\begin{sloppypar}\noindent To this end, choose an enumeration $\{ z^j : 1 \leq j \leq 3^{d(m-n-2)} \}$ of the elements of the set $3^{n+2}\Zd \cap \cu_m$. For each $1\leq j \leq  3^{d(m-n-2)}$, we let $\{ z^{i,j} \,:\, 1 \leq i \leq 3^{2d}\}$ be an enumeration of the elements of the set $3^n \Zd \cap \left( z^j + \cu_{n+2} \right)$, such that, for every $1\leq j,j'\leq 3^{d(m-n-2)}$ and $1\leq i \leq 3^{2d}$, $z^j - z^{j'} = z^{i,j} - z^{i,j'}$. The point of this is that, for every $1 \leq i \leq 3^{2d}$ and $1 \leq j < j' \leq 3^{d(m-n-2)}$, we have $\dist\left(z^{i,j} + \cu_n ,z^{i,j'}+ \cu_n\right) \geq 3^{n+1}$ and therefore, $3^{-dm} \sum_{x \in z^{i,j} + \cu_n, \, e \in \mathcal{B}_d^x} \left| \X_{3^n, C_0^2}^e(x)\right|^t$ and $3^{-dm} \sum_{x \in z^{i,j'} + \cu_n, \, e \in \mathcal{B}_d^x} \left| \X_{3^n, C_0^2}^e(x)\right|^t $ are independent.
Now fix $h > 0$ and compute, using the H\"older inequality and the independence, \end{sloppypar}
\begin{align*}
\lefteqn{
\log \E\left[ \exp \left( h3^{-nt} Z \right)  \right] 
} \qquad & \\
& = \log \E \left[ \prod_{i =1}^{3^{2d}}  \prod_{ j =1 }^{ 3^{d(m-n-2)}} \exp\left( h 3^{-nt-d(m-n)} 3^{-dm} \sum_{x \in z^{i,j} + \cu_n, \, e \in \mathcal{B}_d^x} \left|  \X_{3^n, C_0^2}^e(x)\right|^t \right) \right] \\
& \leq 3^{-{2d}} \sum_{i =1}^{3^{2d}} \log \E \left[  \prod_{ j =1 }^{ 3^{d(m-n-2)}}   \exp\left( h 3^{-nt-d(m-n-2)}  3^{-dm} \sum_{x \in z^{i,j} + \cu_n, \, e \in \mathcal{B}_d^x} \left|  \X_{3^n, C_0^2}^e(x)\right|^t  \right) \right] \\
& \leq 3^{-{2d}} \sum_{i =1}^{3^{2d}}   \sum_{ j =1 }^{ 3^{d(m-n-2)}}  \log \E \left[  \exp\left( h 3^{-nt-d(m-n-2)} 3^{-dm} \sum_{x \in z^{i,j} + \cu_n, \, e \in \mathcal{B}_d^x} \left|  \X_{3^n, C_0^2}^e(x)\right|^t  \right) \right].
\end{align*}
This inequality can be rewritten
\begin{multline*}
\log \E\left[ \exp \left( h3^{-nt} Z \right)  \right]  \\
\leq 3^{-2d} \sum_{z'\in 3^n\Zd\cap(z+\cu_m)}  \log \E \left[  \exp\left( h 3^{-nt-d(m-n-2)} 3^{-dm} \sum_{x \in z' + \cu_n, \, e \in \mathcal{B}_d^x} \left|  \X_{3^n, C_0^2}^e(x)\right|^t  \right) \right].
\end{multline*}
We use the inequality
\begin{equation*} \label{}
\forall y \in [0,1], \quad \exp(y) \leq 1+2y
\end{equation*}
to obtain, for every $h \in [0,(2d)^{-t}3^{d(m-n-2)}]$, 
\begin{equation*} \label{}
\exp\left( h 3^{-nt-d(m-n-2)}\sum_{x \in z' + \cu_n, \, e \in \mathcal{B}_d^x} \left| \X_{3^n, C_0^2}^e(x)\right|^t \right)
\leq 1+2h 3^{-nt-d(m-n-2)} \sum_{x \in z' + \cu_n, \, e \in \mathcal{B}_d^x} \left|   \X_{3^n, C_0^2}^e(x)\right|^t .
\end{equation*}
Taking the expectation in the previous display and using the elementary inequality
\begin{equation*}
\forall y\geq 0, \ \log(1+y) \leq y,
\end{equation*}
we obtain
\begin{align*} \label{}
\log \E\left[ \exp \left( h3^{-nt} Z \right)  \right] 
& \leq 3^{d(m-n)} \log \left( 1 + 2 h 3^{-nt-d(m-n-1)} \E \left[  \sum_{x \in z' + \cu_n, \, e \in \mathcal{B}_d^x} \left|   \X_{3^n, C_0^2}^e(x)\right|^t \right] \right)  \\
& \leq 2h 3^{-nt+d} \E \left[  \sum_{x \in z' + \cu_n, \, e \in \mathcal{B}_d^x} \left|   \X_{3^n, C_0^2}^e(x)\right|^t \right] \\
& \leq Ch3^{-nt}.
\end{align*}
Taking $h:= (2d)^{-t}3^{d(m-n-2)}$ yields
\begin{equation*}
\E\left[ \exp \left(  (2d)^{-t}3^{d(m-n-2)-nt} Z \right)  \right] \leq \exp\left( C3^{d(m-n)-nt} \right). 
\end{equation*}
From this and Chebyshev's inequality, we obtain a constant~$C$ such that 
\begin{equation*}
\P \left[ Z \geq C + h \right] \leq \exp\left( -h C^{-1}3^{d(m-n)-nt} \right).
\end{equation*}
This implies~\eqref{e.Zbound}.

\smallskip

\emph{Step 2.} We complete the proof by applying a union bound. Combining~\eqref{e.Lambdatbounds} and~\eqref{e.Zbound} yields
\begin{equation*} \label{}
 \sum_{x \in \cu_m, \, e \in \mathcal{B}_d^x} \left|  \X_{3^n, C_0}^e(x)\right|^t \leq C + \O_1 \left( C3^{nt-d(m-n)} \right)+ \O_{s'}\left( C3^{-nt'} \right).
\end{equation*}
We set
\begin{equation*} \label{}
n:= \left\lceil \frac{dm}{d+t+1} \right\rceil \mbox{ and } t' = 1
\end{equation*}
so that the previous line becomes 
\begin{equation*} \label{}
\sum_{x \in  \cu_m, \, e \in \mathcal{B}_d^x} \left| \X_{3^n, C_0}^e(x)\right|^t \leq C + \O_1' \left( C3^{-\frac{d}{d+t+1}m} \right)+ \O_{s'}'\left( C3^{-\frac{d}{d+t+1}m}  \right).
\end{equation*}
Thus, by~\eqref{e.improves} and~\eqref{e.Osums}, we obtain the existence of two exponents $s := s(d , \p, \lambda , t) > 0$, $\beta := \beta(d , \p, \lambda , t) > 0$ and of a constant $C_0 := C_0(d , \p, \lambda , t) < \infty$ such that
\begin{equation*}
\sum_{x \in \cu_m, \, e \in \mathcal{B}_d^x} \left| \X_{3^n, C_0}^e(x)\right|^t  \leq C_0 + \O_{s}'\left( C_03^{-\beta m}  \right).
\end{equation*}
Define 
\begin{equation*} \label{}
\M_t^\X:= \sup\left\{ 3^m \,:\,  \sum_{x \in \cu_m, \, e \in \mathcal{B}_d^x} \left|  \X_{3^n, C_0}^e(x)\right|^t  \geq C_0 +1 \right\}.
\end{equation*}
We apply Lemma~\ref{lemmafromperc.0} with $X_n =\sum_{x \in \cu_m, \, e \in \mathcal{B}_d^x} \left|  \X_{3^n, C_0}^e(x)\right|^t$ and $K = C_0$ to obtain the inequality
\begin{equation*}
\M_t^\X \leq \O_{s \beta} \left( C \right).
\end{equation*}
The proof is complete.

\end{proof}

\bibliographystyle{abbrv}
\bibliography{holes}

\newcommand{\noop}[1]{} \def\cprime{$'$}
\begin{thebibliography}{10}

\bibitem{ABDH}
S.~Andres, M.~T. Barlow, J.-D. Deuschel, and B.~M. Hambly.
\newblock Invariance principle for the random conductance model.
\newblock {\em Probab. Theory Related Fields}, 156(3-4):535--580, 2013.

\bibitem{andres2018quenched}
S.~Andres, A.~Chiarini, J.-D. Deuschel, and M.~Slowik.
\newblock Quenched invariance principle for random walks with time-dependent
  ergodic degenerate weights.
\newblock {\em The Annals of Probability}, 46(1):302--336, 2018.

\bibitem{andres2015invariance}
S.~Andres, J.-D. Deuschel, and M.~Slowik.
\newblock Invariance principle for the random conductance model in a degenerate
  ergodic environment.
\newblock {\em The Annals of Probability}, 43(4):1866--1891, 2015.

\bibitem{andres2016heat}
S.~Andres, J.-D. Deuschel, and M.~Slowik.
\newblock Heat kernel estimates for random walks with degenerate weights.
\newblock {\em Electronic Journal of Probability}, 21, 2016.

\bibitem{AN17}
S.~Andres and S.~Neukamm.
\newblock Berry-{E}sseen theorem and quantitative homogenization for the random
  conductance model with degenerate conductances.
\newblock {\em Stoch. Partial Differ. Equ. Anal. Comput.}, 7(2):240--296, 2019.

\bibitem{AP96}
P.~Antal and A.~Pisztora.
\newblock On the chemical distance for supercritical {B}ernoulli percolation.
\newblock {\em Ann. Probab.}, 24(2):1036--1048, 1996.

\bibitem{AD2}
S.~Armstrong and P.~Dario.
\newblock Elliptic regularity and quantitative homogenization on percolation
  clusters.
\newblock {\em Communications on Pure and Applied Mathematics}, 2016.

\bibitem{armstrong2018iterative}
S.~Armstrong, A.~Hannukainen, T.~Kuusi, and J.-C. Mourrat.
\newblock An iterative method for elliptic problems with rapidly oscillating
  coefficients.
\newblock {\em arXiv preprint arXiv:1803.03551}, 2018.

\bibitem{AKM2}
S.~Armstrong, T.~Kuusi, and J.-C. Mourrat.
\newblock The additive structure of elliptic homogenization.
\newblock {\em Inventiones mathematicae}, 208(3):999--1154, 2017.

\bibitem{armstrong2017quantitative}
S.~Armstrong, T.~Kuusi, and J.-C. Mourrat.
\newblock {\em Quantitative stochastic homogenization and large-scale
  regularity}, volume 352 of {\em Grundlehren der mathematischen
  Wissenschaften}.
\newblock Springer International Publishing, 2019.

\bibitem{armstrong2017optimal}
S.~Armstrong and J.~Lin.
\newblock Optimal quantitative estimates in stochastic homogenization for
  elliptic equations in nondivergence form.
\newblock {\em Archive for Rational Mechanics and Analysis}, 225(2):937--991,
  2017.

\bibitem{AM}
S.~Armstrong and J.-C. Mourrat.
\newblock Lipschitz regularity for elliptic equations with random coefficients.
\newblock {\em Arch. Ration. Mech. Anal.}, 219(1):255--348, 2016.

\bibitem{AS}
S.~Armstrong and C.~K. Smart.
\newblock Quantitative stochastic homogenization of convex integral
  functionals.
\newblock {\em Ann. Sci. \'Ec. Norm. Sup\'er. (4)}, 49(2):423--481, 2016.

\bibitem{AL1}
M.~Avellaneda and F.-H. Lin.
\newblock Compactness methods in the theory of homogenization.
\newblock {\em Comm. Pure Appl. Math.}, 40(6):803--847, 1987.

\bibitem{AL2}
M.~Avellaneda and F.-H. Lin.
\newblock {$L^p$} bounds on singular integrals in homogenization.
\newblock {\em Comm. Pure Appl. Math.}, 44(8-9):897--910, 1991.

\bibitem{Ba}
M.~T. Barlow.
\newblock Random walks on supercritical percolation clusters.
\newblock {\em Ann. Probab.}, 32(4):3024--3084, 2004.

\bibitem{BD}
M.~T. Barlow and J.-D. Deuschel.
\newblock Invariance principle for the random conductance model with unbounded
  conductances.
\newblock {\em Ann. Probab.}, 38(1):234--276, 2010.

\bibitem{bella2018liouville}
P.~Bella, B.~Fehrman, and F.~Otto.
\newblock A {L}iouville theorem for elliptic systems with degenerate ergodic
  coefficients.
\newblock {\em The Annals of Applied Probability}, 28(3):1379--1422, 2018.

\bibitem{BDKY}
I.~Benjamini, H.~Duminil-Copin, G.~Kozma, and A.~Yadin.
\newblock Disorder, entropy and harmonic functions.
\newblock {\em Ann. Probab.}, 43(5):2332--2373, 2015.

\bibitem{BB}
N.~Berger and M.~Biskup.
\newblock Quenched invariance principle for simple random walk on percolation
  clusters.
\newblock {\em Probab. Theory Related Fields}, 137(1-2):83--120, 2007.

\bibitem{BP}
M.~Biskup and T.~M. Prescott.
\newblock Functional {CLT} for random walk among bounded random conductances.
\newblock {\em Electron. J. Probab.}, 12:no. 49, 1323--1348, 2007.

\bibitem{BK}
R.~M. Burton and M.~Keane.
\newblock Density and uniqueness in percolation.
\newblock {\em Comm. Math. Phys.}, 121(3):501--505, 1989.

\bibitem{chiarini2016invariance}
A.~Chiarini and J.-D. Deuschel.
\newblock Invariance principle for symmetric diffusions in a degenerate and
  unbounded stationary and ergodic random medium.
\newblock In {\em Annales de l'Institut Henri Poincar{\'e}, Probabilit{\'e}s et
  Statistiques}, volume~52, pages 1535--1563. Institut Henri Poincar{\'e},
  2016.

\bibitem{DM1}
G.~Dal~Maso and L.~Modica.
\newblock Nonlinear stochastic homogenization.
\newblock {\em Ann. Mat. Pura Appl. (4)}, 144:347--389, 1986.

\bibitem{DM2}
G.~Dal~Maso and L.~Modica.
\newblock Nonlinear stochastic homogenization and ergodic theory.
\newblock {\em J. Reine Angew. Math.}, 368:28--42, 1986.

\bibitem{Da93}
E.~B. Davies.
\newblock Large deviations for heat kernels on graphs.
\newblock {\em J. London Math. Soc. (2)}, 47(1):65--72, 1993.

\bibitem{de1989invariance}
A.~De~Masi, P.~Ferrari, S.~Goldstein, and W.~D. Wick.
\newblock An invariance principle for reversible {M}arkov processes.
  applications to random motions in random environments.
\newblock {\em Journal of Statistical Physics}, 55(3-4):787--855, 1989.

\bibitem{De99}
T.~Delmotte.
\newblock Parabolic {H}arnack inequality and estimates of {M}arkov chains on
  graphs.
\newblock {\em Rev. Mat. Iberoamericana}, 15(1):181--232, 1999.

\bibitem{deuschel2018quenched}
J.-D. Deuschel, T.~Nguyen, and M.~Slowik.
\newblock Quenched invariance principles for the random conductance model on a
  random graph with degenerate ergodic weights.
\newblock {\em Probability Theory and Related Fields}, 170(1-2):363--386, 2018.

\bibitem{EGMN}
A.-C. Egloffe, A.~Gloria, J.-C. Mourrat, and T.~N. Nguyen.
\newblock Random walk in random environment, corrector equation and homogenized
  coefficients: from theory to numerics, back and forth.
\newblock {\em IMA J. Numer. Anal.}, 35(2):499--545, 2015.

\bibitem{gilbarg2015elliptic}
D.~Gilbarg and N.~S. Trudinger.
\newblock {\em Elliptic partial differential equations of second order}.
\newblock springer, 2015.

\bibitem{Giu}
E.~Giusti.
\newblock {\em Direct methods in the calculus of variations}.
\newblock World Scientific Publishing Co., Inc., River Edge, NJ, 2003.

\bibitem{GNO2}
A.~Gloria, S.~Neukamm, and F.~Otto.
\newblock An optimal quantitative two-scale expansion in stochastic
  homogenization of discrete elliptic equations.
\newblock {\em ESAIM Math. Model. Numer. Anal.}, 48(2):325--346, 2014.

\bibitem{GNO3}
A.~Gloria, S.~Neukamm, and F.~Otto.
\newblock A regularity theory for random elliptic operators.
\newblock {\em arXiv preprint arXiv:1409.2678}, 2014.

\bibitem{GNO}
A.~Gloria, S.~Neukamm, and F.~Otto.
\newblock Quantification of ergodicity in stochastic homogenization: optimal
  bounds via spectral gap on {G}lauber dynamics.
\newblock {\em Invent. Math.}, 199(2):455--515, 2015.

\bibitem{GO1}
A.~Gloria and F.~Otto.
\newblock An optimal variance estimate in stochastic homogenization of discrete
  elliptic equations.
\newblock {\em Ann. Probab.}, 39(3):779--856, 2011.

\bibitem{GO2}
A.~Gloria and F.~Otto.
\newblock An optimal error estimate in stochastic homogenization of discrete
  elliptic equations.
\newblock {\em Ann. Appl. Probab.}, 22(1):1--28, 2012.

\bibitem{GO5}
A.~Gloria and F.~Otto.
\newblock The corrector in stochastic homogenization: near-optimal rates with
  optimal stochastic integrability, \noop{2015}{preprint, arXiv:1510.08290}.

\bibitem{GO3}
A.~Gloria and F.~Otto.
\newblock Quantitative results on the corrector equation in stochastic
  homogenization.
\newblock {\em J. Eur. Math. Soc.}, \noop{3001}{in press, arXiv:1409.0801}.

\bibitem{LNO}
A.~Lamacz, S.~Neukamm, and F.~Otto.
\newblock Moment bounds for the corrector in stochastic homogenization of a
  percolation model.
\newblock {\em Electron. J. Probab.}, 20:no. 106, 30, 2015.

\bibitem{M}
P.~Mathieu.
\newblock Quenched invariance principles for random walks with random
  conductances.
\newblock {\em J. Stat. Phys.}, 130(5):1025--1046, 2008.

\bibitem{MP}
P.~Mathieu and A.~Piatnitski.
\newblock Quenched invariance principles for random walks on percolation
  clusters.
\newblock {\em Proc. R. Soc. Lond. Ser. A Math. Phys. Eng. Sci.},
  463(2085):2287--2307, 2007.

\bibitem{MR}
P.~Mathieu and E.~Remy.
\newblock Isoperimetry and heat kernel decay on percolation clusters.
\newblock {\em Ann. Probab.}, 32(1A):100--128, 2004.

\bibitem{mourrat2011variance}
J.-C. Mourrat.
\newblock Variance decay for functionals of the environment viewed by the
  particle.
\newblock In {\em Annales de l'Institut Henri Poincar{\'e}, Probabilit{\'e}s et
  Statistiques}, volume~47, pages 294--327. Institut Henri Poincar{\'e}, 2011.

\bibitem{mourrat2012quantitative}
J.-C. Mourrat.
\newblock A quantitative central limit theorem for the random walk among random
  conductances.
\newblock {\em Electronic Journal of Probability}, 17, 2012.

\bibitem{mourrat2016efficient}
J.-C. Mourrat.
\newblock Efficient methods for the estimation of homogenized coefficients.
\newblock {\em Found. Comput. Math.}, 19(2):435--483, 2019.

\bibitem{NS}
A.~Naddaf and T.~Spencer.
\newblock Estimates on the variance of some homogenization problems, 1998,
  {unpublished preprint}.

\bibitem{PP96}
M.~Penrose and A.~Pisztora.
\newblock Large deviations for discrete and continuous percolation.
\newblock {\em Adv. in Appl. Probab.}, 28(1):29--52, 1996.

\bibitem{P96}
A.~Pisztora.
\newblock Surface order large deviations for {I}sing, {P}otts and percolation
  models.
\newblock {\em Probab. Theory Related Fields}, 104(4):427--466, 1996.

\bibitem{procaccia2016}
E.~Procaccia, R.~Rosenthal, and A.~Sapozhnikov.
\newblock Quenched invariance principle for simple random walk on clusters in
  correlated percolation models.
\newblock {\em Probability theory and related fields}, 166(3-4):619--657, 2016.

\bibitem{SS}
V.~Sidoravicius and A.-S. Sznitman.
\newblock Quenched invariance principles for walks on clusters of percolation
  or among random conductances.
\newblock {\em Probab. Theory Related Fields}, 129(2):219--244, 2004.

\end{thebibliography}

\end{document}